\documentclass[reqno,10pt]{amsart}
\usepackage{amscd,amssymb,verbatim,array}
\usepackage{hyperref}
\usepackage{amsmath}

\numberwithin{equation}{section} \theoremstyle{plain}

\newtheorem{theorem}{Theorem}[section]
\newtheorem{lemma}[theorem]{Lemma}

\newtheorem{corollary}[theorem]{Corollary}

\theoremstyle{definition}

\theoremstyle{remark}

\numberwithin{equation}{section}

\newcommand{\an}{\operatorname{an}}

\newcommand{\Det}{\operatorname{Det}}
\newcommand{\B}{\mathcal{B}}
\newcommand{\E}{\mathcal{E}}

\newcommand{\even}{\operatorname{even}}
\newcommand{\trivial}{\operatorname{trivial}}
\newcommand{\cyl}{\operatorname{cyl}}
\newcommand{\odd}{\operatorname{odd}}
\newcommand{\gr}{\operatorname{gr}}

\newcommand{\Tan}{\operatorname{tan}}
\newcommand{\Nor}{\operatorname{nor}}
\newcommand{\Dim}{\operatorname{dim}}
\newcommand{\Mod}{\operatorname{mod}}
\newcommand{\Min}{\operatorname{min}}
\newcommand{\Max}{\operatorname{max}}
\newcommand{\rk}{\operatorname{rk}}
\newcommand{\rank}{\operatorname{rank}}
\newcommand{\pr}{\operatorname{pr}}

\newcommand{\Ker}{\operatorname{ker}}

\newcommand{\Spec}{\operatorname{Spec}}
\newcommand{\Sign}{\operatorname{sign}}
\newcommand{\Tr}{\operatorname{Tr}}

\newcommand{\Res}{\operatorname{Res}}

\newcommand{\Imm}{\operatorname{Im}}
\newcommand{\I}{\operatorname{I}}
\newcommand{\II}{\operatorname{II}}
\newcommand{\III}{\operatorname{III}}
\newcommand{\J}{\operatorname{J}}
\newcommand{\K}{\operatorname{K}}
\newcommand{\LL}{\operatorname{L}}
\newcommand{\Dom}{\operatorname{Dom}}

\newcommand{\rel}{\operatorname{rel}}
\newcommand{\Abs}{\operatorname{abs}}

\newcommand{\ddet}{\operatorname{det}}
\newcommand{\Id}{\operatorname{Id}}
\newcommand{\SF}{\operatorname{SF}}
\newcommand{\Mas}{\operatorname{Mas}}
\newcommand{\doub}{\operatorname{doub}}
\newcommand{\double}{\operatorname{double}}
\newcommand{\m}{\operatorname{m}}

\newcommand{\R}{\operatorname{r}}
\newcommand{\A}{\operatorname{a}}

\newcommand{\supp}{\operatorname{supp}}

\begin{document}

\title[Comparison of two constructions of the refined analytic torsion]
{The comparison of two constructions of the refined analytic torsion on compact manifolds with boundary}
\author{Rung-Tzung Huang}

\address{Department of Mathematics, National Central University, Chung-Li 320,Taiwan}

\email{rthuang@math.ncu.edu.tw}

\author{Yoonweon Lee}

\address{Department of Mathematics, Inha University, Incheon, 402-751, Korea}

\email{yoonweon@inha.ac.kr}

%\date{\today}

\begin{abstract}
The refined analytic torsion on compact Riemannian manifolds with boundary has been discussed by B. Vertman ([22], [23]) and the authors ([11], [12])
but these two constructions are completely different.
Vertman used a double of de Rham complex consisting of
the minimal and maximal closed extensions of a flat connection and the authors used well-posed boundary conditions
${\mathcal P}_{-, {\mathcal L}_{0}}$, ${\mathcal P}_{+, {\mathcal L}_{1}}$ for the odd signature operator.
In this paper we compare these two constructions by using the BFK-gluing formula for zeta-determinants, the adiabatic method
for stretching cylinder part near boundary and the deformation method
used in [6] when the odd signature operator comes from a Hermitian flat connection and all de Rham cohomologies vanish.
\end{abstract}
\maketitle

\section{Introduction}

The refined analytic torsion was introduced by M. Braverman and T. Kappeler ([4], [5]) on an odd dimensional closed Riemannian manifold
with a flat bundle as an analytic analogue of the refined combinatorial torsion introduced by M. Farber and V. Turaev ([20], [21], [8], [9]).
Even though these two objects do not coincide exactly, they are closely related.
The refined analytic torsion is defined by using the graded zeta-determinant of the odd signature operator and is described as
an element of the determinant line of the cohomologies.
Specially, when the odd signature operator comes from an acyclic Hermitian connection on a closed manifold, the refined analytic torsion is a complex number,
whose modulus part is the Ray-Singer analytic torsion and the phase part is the $\rho$-invariant determined by the given odd signature operator
and the odd signature operator defined by the trivial connection acting on the trivial line bundle.

The refined analytic torsion on compact Riemannian manifolds with boundary has been discussed by B. Vertman ([22], [23]) and the authors ([11], [12]) but these two constructions are completely different.
Vertman used a double of de Rham complex consisting of
the minimal and maximal closed extensions of a flat connection.
On the other hand, the authors introduced well-posed boundary conditions
${\mathcal P}_{-, {\mathcal L}_{0}}$, ${\mathcal P}_{+, {\mathcal L}_{1}}$ for the odd signature operator to define the refined analytic torsion on compact Riemannian manifolds with boundary.
In this paper we are going to compare these two constructions when the odd signature operator comes from a Hermitian connection and all de Rham cohomologies vanish.
For comparison of the Ray-Singer analytic torsion part
we are going to use the BFK-gluing formula for zeta-determinants proven in [7] and the adiabatic method for stretching cylinder part near boundary.
For comparison of the eta invariant part we are going to use the deformation method used in [6]. These methods were used in [12], where the authors discussed the gluing formula
of the refined analytic torsion with respect to the well-posed boundary conditions
${\mathcal P}_{-, {\mathcal L}_{0}}$, ${\mathcal P}_{+, {\mathcal L}_{1}}$.
Hence this work is a continuation of [12].

We now begin with the description of the odd signature operator near boundary.

\vspace{0.3 cm}

\section{The refined analytic torsion on manifolds with boundary}
In this section we first describe the odd signature operator $\B$ near boundary and introduce the well-posed boundary conditions
${\mathcal P}_{-, {\mathcal L}_{0}}$, ${\mathcal P}_{+ ,{\mathcal L}_{1}}$ for the odd signature operator.
We then review the construction of the refined analytic torsions with respect to ${\mathcal P}_{- ,{\mathcal L}_{0}}$, ${\mathcal P}_{+ ,{\mathcal L}_{1}}$
discussed in [11].

Let $(M, g^{M})$ be a compact oriented odd dimensional Riemannian
manifold with boundary $Y$, where $g^{M}$ is assumed to be a product metric near the boundary $Y$.
We denote the dimension of $M$ by $m = 2r - 1$. Suppose that $\rho : \pi_{1}(M) \rightarrow
GL(n, {\Bbb C})$ is a representation of the fundamental group and $E = {\widetilde M}
\times_{\rho} {\Bbb C}^{n}$ is the associated flat bundle, where ${\widetilde M}$
is a universal covering space of $M$. We choose a flat connection
$\nabla$ and extend it to a covariant differential

$$
\nabla : \Omega^{\bullet}(M, E) \rightarrow \Omega^{\bullet + 1}(M, E).
$$

\noindent
Using the Hodge star operator $\ast_{M}$, we define an involution
$\Gamma = \Gamma(g^{M}) : \Omega^{\bullet}(M, E) \rightarrow \Omega^{m - \bullet}(M, E)$ by

\begin{equation}\label{E:1.1}
\Gamma \omega := i^{r} (-1)^{\frac{q(q+1)}{2}} \ast_{M} \omega, \qquad \omega \in \Omega^{q}(M, E),
\end{equation}

\noindent
where $r$ is given as above by $r = \frac{m+1}{2}$.
It is straightforward to see that $\Gamma^{2} = \Id$.
We define the odd signature operator $\B$ by

\begin{equation}\label{E:1.2}
    \B\ = \  \B(\nabla,g^M) \ := \ \Gamma\,\nabla \ + \ \nabla\,\Gamma:\,\Omega^\bullet(M,E)\ \longrightarrow \  \Omega^\bullet(M,E).
\end{equation}

\noindent
Then $\B$ is an elliptic differential operator of order $1$.
Let $N$ be a collar neighborhood of $Y$ which is isometric to $[0, 1) \times Y$. Then we have a natural isomorphism

\begin{eqnarray}\label{E:1.3}
\Psi : \Omega^{p}(N, E|_{N}) & \rightarrow & C^{\infty}([0, 1), \Omega^{p}(Y, E|_{Y}) \oplus \Omega^{p-1}(Y, E|_{Y}))  \\
\omega_{1} + dx \wedge \omega_{2} & \mapsto & \left( \begin{array}{clcr} \omega_{1} \\ \omega_{2} \end{array} \right)  \nonumber
\end{eqnarray}

\noindent
Using the product structure we can
induce a flat connection $\nabla^{Y} : \Omega^{\bullet}(Y, E|_{Y})
\rightarrow \Omega^{\bullet}(Y, E|_{Y})$ from $\nabla$ and the Hodge star operator $\ast_{Y} : \Omega^{\bullet}(Y, E|_{Y}) \rightarrow
\Omega^{m-1-\bullet}(Y, E|_{Y})$ from $\ast_{M}$.
We define two maps $\beta$, $\Gamma^{Y}$ by

\vspace{0.2 cm}

\begin{equation}\label{E:1.4}
 \begin{aligned}
  \beta  & :  \Omega^{p}(Y, E|_{Y}) \rightarrow  \Omega^{p}(Y, E|_{Y}), \quad \beta(\omega) = (-1)^{p} \omega \\
  \Gamma^{Y}  & :  \Omega^{p}(Y, E|_{Y}) \rightarrow  \Omega^{m-1-p}(Y,
E|_{Y}),\quad \Gamma^{Y}(\omega) = i^{r-1} (-1)^{\frac{p(p+1)}{2}}
\ast_{Y} \omega.
\end{aligned}
\end{equation}

\vspace{0.2 cm}

\noindent
It is straightforward that

\begin{equation}\label{E:1.5}
\beta^{2} = \Id,  \qquad  \Gamma^{Y} \Gamma^{Y} = \Id.
\end{equation}

\noindent
Then simple computation shows that

\begin{equation}\label{E:1.6}
 \Gamma = i \beta \Gamma^{Y} \left( \begin{array}{clcr} 0 & -1 \\ 1 & 0 \end{array} \right), \qquad
\nabla = \left( \begin{array}{clcr} 0 & 0 \\ 1 & 0 \end{array} \right) \nabla_{\partial_{x}} +
\left( \begin{array}{clcr} 1 & 0 \\ 0 & -1 \end{array} \right) \nabla^{Y},
\end{equation}

\noindent
where $\partial_{x}$ is the inward normal derivative to the boundary $Y$ on $N$.
Hence the odd signature operator $\B$ is expressed, under the isomorphism (\ref{E:1.3}), by

\begin{equation}\label{E:1.8}
 \B = - i \beta \Gamma^{Y} \left\{ \left( \begin{array}{clcr} 1 & 0 \\
0 & 1 \end{array} \right) \nabla_{\partial_{x}} + \left( \begin{array}{clcr}  0 & -1
\\ -1 & 0 \end{array} \right) \left( \nabla^{Y} + \Gamma^{Y} \nabla^{Y} \Gamma^{Y} \right) \right\}.
\end{equation}

\noindent
We denote

\begin{equation}  \label{E:1.9}
{\mathcal \gamma} := - i \beta \Gamma^{Y} \left( \begin{array}{clcr} 1 & 0 \\ 0 & 1 \end{array} \right), \qquad
{\mathcal A} := \left( \begin{array}{clcr}  0 & -1 \\ -1 & 0 \end{array} \right) \left( \nabla^{Y} + \Gamma^{Y} \nabla^{Y} \Gamma^{Y} \right)
\end{equation}

\noindent
so that $\B$ has the form of

\noindent
\begin{equation}  \label{E:1.10}
 \B =  {\mathcal \gamma} \left( \partial_{x} + {\mathcal A} \right) \qquad \text{with} \quad {\mathcal \gamma}^{2} = - \Id, \quad
{\mathcal \gamma} {\mathcal A} = - {\mathcal A} {\mathcal \gamma}.
\end{equation}

\noindent
Since $\nabla_{\partial_{x}} \nabla^{Y} = \nabla^{Y} \nabla_{\partial_{x}}$, we have

\begin{equation}\label{E:1.11}
\B^{2} = - \left( \begin{array}{clcr}  1 & 0 \\ 0 & 1 \end{array} \right) \nabla_{\partial_{x}}^{2} +
\left( \begin{array}{clcr}  1 & 0 \\ 0 & 1 \end{array} \right) \left( \nabla^{Y} + \Gamma^{Y} \nabla^{Y} \Gamma^{Y} \right)^{2} =
\left( - \nabla_{\partial_{x}}^{2} + \B_{Y}^{2} \right) \left( \begin{array}{clcr}  1 & 0 \\ 0 & 1 \end{array} \right),
\end{equation}

\noindent
where
$$
\B_{Y} = \Gamma^{Y} \nabla^{Y} + \nabla^{Y} \Gamma^{Y}.
$$

We next choose a Hermitian inner product $h^{E}$.
All through this paper we assume that $\nabla$ is a Hermitian connection with respect to $h^{E}$,
which means that $\nabla$ is compatible with $h^{E}$, {\it i.e.}
for any $\phi$, $\psi \in C^{\infty}(E)$,

$$
d h^{E}(\phi, \psi) = h^{E}(\nabla \phi, \psi) + h^{E}(\phi, \nabla \psi).
$$

\noindent
The Green formula for $\B$ is given as follows (cf. [11]).

\vspace{0.2 cm}

\begin{lemma} \label{Lemma:1.1}
(1) For $\phi \in \Omega^{q}(M, E)$, $\psi \in \Omega^{m-q}(M,E)$,
$\hspace{0.2 cm}  \langle \Gamma \phi, \hspace{0.1 cm} \psi \rangle_{M} \hspace{0.1 cm}
= \hspace{0.1 cm} \langle \phi, \hspace{0.1 cm} \Gamma \psi \rangle_{M}$.   \newline
(2) For $\phi \in \Omega^{q}(M, E)$, $\psi \in \Omega^{q+1}(M,E)$,
$$
\langle \nabla \phi, \hspace{0.1 cm} \psi \rangle_{M} \hspace{0.1 cm} = \hspace{0.1 cm}
\langle \phi, \hspace{0.1 cm} \Gamma \nabla \Gamma \psi \rangle_{M} \hspace{0.1 cm} - \hspace{0.1 cm}
\langle \phi_{\Tan}|_{Y}, \hspace{0.1 cm} \psi_{\Nor}|_{Y} \rangle_{Y}.
$$
(3) For $\phi$, $\psi \in \Omega^{\even}(M, E)$ or $\Omega^{\odd}(M,E)$,
$$
\langle \B \phi, \hspace{0.1 cm} \psi \rangle_{M} \hspace{0.1 cm} - \hspace{0.1 cm}
\langle \phi, \hspace{0.1 cm} \B \psi \rangle_{M} \hspace{0.1 cm} = \hspace{0.1 cm}
\langle \phi_{\Tan}|_{Y}, \hspace{0.1 cm} i \beta \Gamma^{Y} (\psi_{\Tan}|_{Y}) \rangle_{Y} \hspace{0.1 cm} - \hspace{0.1 cm}
\langle \phi_{\Nor}|_{Y}, \hspace{0.1 cm} i \beta \Gamma^{Y} (\psi_{\Nor}|_{Y}) \rangle_{Y}
\hspace{0.1 cm} = \hspace{0.1 cm} \langle \phi|_{Y}, \hspace{0.1 cm} {\mathcal \gamma} (\psi|_{Y}) \rangle_{Y}.
$$
\end{lemma}

\vspace{0.2 cm}

\noindent
{\it Remark} : In the assertions (2) and (3) the signs on the inner products on $Y$ are different from those in [11] because in [11]
$\partial_{x}$ is an outward normal derivative.

\vspace{0.3 cm}

We note that $\B_{Y}$ is a self-adjoint elliptic operator on $Y$.
Putting ${\mathcal H}^{\bullet}(Y, E|_{Y}) := \Ker \B^{2}_{Y}$, ${\mathcal H}^{\bullet}(Y, E|_{Y})$ is a finite dimensional vector space and we have
$$
\Omega^{\bullet}(Y, E|_{Y}) = \Imm \nabla^{Y} \oplus \Imm \Gamma^{Y} \nabla^{Y} \Gamma^{Y} \oplus {\mathcal H}^{\bullet}(Y, E|_{Y}).
$$

\noindent
If $\nabla \phi = \Gamma \nabla \Gamma \phi = 0$ for $\phi \in \Omega^{\bullet}(M, E)$, simple computation shows that $\phi$ is expressed,
near the boundary $Y$, by

\begin{equation} \label{E:1.12}
\phi = \nabla^{Y} \phi_{\Tan} + \phi_{\Tan, h} + dx \wedge ( \Gamma^{Y} \nabla^{Y} \Gamma^{Y} \phi_{\Nor} + \phi_{\Nor, h}), \qquad
\phi_{\Tan, h}, \hspace{0.1 cm} \phi_{\Nor, h} \in {\mathcal H}^{\bullet}(Y, E|_{Y}).
\end{equation}

\noindent
We define ${\mathcal K}$ by

\begin{equation} \label{E:1.13}
{\mathcal K} := \{ \phi_{\Tan, h} \in {\mathcal H}^{\bullet}(Y, E|_{Y}) \mid \nabla \phi = \Gamma \nabla \Gamma \phi = 0 \},
\end{equation}

\noindent
where $\phi$ has the form (\ref{E:1.12}).
If $\phi$ satisfies $\nabla \phi = \Gamma \nabla \Gamma \phi = 0$, so is $\Gamma \phi$ and hence

\begin{equation} \label{E:2.47}
\Gamma^{Y} {\mathcal K} = \{ \phi_{\Nor, h} \in {\mathcal H}^{\bullet}(Y, E|_{Y}) \mid \nabla \phi = \Gamma \nabla \Gamma \phi = 0 \},
\end{equation}

\noindent
where $\phi$ has the form (\ref{E:1.12}).
The second assertion in Lemma \ref{Lemma:1.1} shows that
${\mathcal K}$ is perpendicular to $\Gamma^{Y} {\mathcal K}$.
We then have the following decomposition (cf. Corollary 8.4 in [14], Lemma 2.4 in [11]).

\begin{equation}  \label{E:1.144}
{\mathcal K} \oplus \Gamma^{Y} {\mathcal K} = {\mathcal H}^{\bullet}(Y, E|_{Y}),
\end{equation}

\noindent
which shows that
$( {\mathcal H}^{\bullet}(Y, E|_{Y}), \hspace{0.1 cm} \langle \hspace{0.1 cm}, \hspace{0.1 cm} \rangle_{Y}, \hspace{0.1 cm} - i \beta \Gamma^{Y} )$
is a symplectic vector space with Lagrangian subspaces ${\mathcal K}$ and $\Gamma^{Y} {\mathcal K}$.
We denote by

\begin{equation}   \label{E:1.1555}
{\mathcal L}_{0} = \left( \begin{array}{clcr} {\mathcal K} \\ {\mathcal K} \end{array} \right), \qquad
{\mathcal L}_{1} = \left( \begin{array}{clcr} \Gamma^{Y} {\mathcal K} \\ \Gamma^{Y} {\mathcal K} \end{array} \right).
\end{equation}

\vspace{0.2 cm}

\noindent
{\it Remark} : Lemma 2.4 in [11] shows that ${\mathcal K}$ and $\Gamma^{Y} {\mathcal K}$ are the sets of all tangential and normal parts of the limiting values of extended $L^{2}$-solutions to $\B_{\infty}$ on $M_{\infty}$, respectively,
(See (\ref{E:2.77}) below for definitions of $\B_{\infty}$ and $M_{\infty}$).

\vspace{0.2 cm}

We next define the orthogonal projections
${\mathcal P}_{-, {\mathcal L}_{0}}$, ${\mathcal P}_{+, {\mathcal L}_{1}} :
\Omega^{\bullet}(Y, E|_{Y}) \oplus \Omega^{\bullet}(Y, E|_{Y}) \rightarrow \Omega^{\bullet}(Y, E|_{Y}) \oplus \Omega^{\bullet}(Y, E|_{Y})$ by
\begin{equation} \label{E:1.112}
\Imm {\mathcal P}_{-, {\mathcal L}_{0}} \hspace{0.1 cm} = \hspace{0.1 cm}
\left( \begin{array}{clcr} \Imm \nabla^{Y} \oplus {\mathcal K} \\ \Imm \nabla^{Y} \oplus {\mathcal K} \end{array} \right), \qquad
\Imm {\mathcal P}_{+, {\mathcal L}_{1}} \hspace{0.1 cm} = \hspace{0.1 cm}
\left( \begin{array}{clcr} \Imm \Gamma^{Y} \nabla^{Y} \Gamma^{Y} \oplus \Gamma^{Y} {\mathcal K} \\
\Imm \Gamma^{Y} \nabla^{Y} \Gamma^{Y} \oplus \Gamma^{Y} {\mathcal K} \end{array} \right) .
\end{equation}

\noindent
Then ${\mathcal P}_{-, {\mathcal L}_{0}}$, ${\mathcal P}_{+, {\mathcal L}_{1}}$ are pseudodifferential operators and give well-posed
boundary conditions for $\B$ and the refined analytic torsion.
We denote by $\B_{{\mathcal P}_{-, {\mathcal L}_{0}}}$ and $\B^{2}_{q, {\mathcal P}_{-, {\mathcal L}_{0}}}$ the realizations of $\B$ and $\B^{2}_{q}$
with respect to ${\mathcal P}_{-, {\mathcal L}_{0}}$, {\it i.e.}

\begin{eqnarray}\label{E:2.17}
\Dom \left( \B_{{\mathcal P}_{-, {\mathcal L}_{0}}} \right) & = &
\left\{ \psi \in \Omega^{\bullet}(M, E) \mid {\mathcal P}_{-, {\mathcal L}_{0}} \left( \psi|_{Y} \right) = 0 \right\},  \nonumber\\
\Dom \left( \B^{2}_{q, {\mathcal P}_{-, {\mathcal L}_{0}}} \right) & = &
\left\{ \psi \in \Omega^{q}(M, E) \mid {\mathcal P}_{-, {\mathcal L}_{0}} \left( \psi|_{Y} \right) = 0, \hspace{0.1 cm}
{\mathcal P}_{-, {\mathcal L}_{0}} \left( (\B \psi)|_{Y} \right) = 0 \right\}.
\end{eqnarray}

\noindent
We define $\B_{{\mathcal P}_{+, {\mathcal L}_{1}}}$, $\B^{2}_{q, {\mathcal P}_{+, {\mathcal L}_{1}}}$, $\B^{2}_{q, \Abs}$, $\B^{2}_{q, \rel}$
and $\B_{\Pi_{>}}$, $\B_{\Pi_{<}}$ (see Section 3) in the similar way.
The following result is straightforward (Lemma 2.11 in [11]).

\vspace{0.2 cm}

\begin{lemma}  \label{Lemma:1.2}
$$
\Ker \B^{2}_{q, {\mathcal P}_{-, {\mathcal L}_{0}}} = \ker \B^{2}_{q, \rel} = H^{q}(M, Y ; E), \qquad
\Ker \B^{2}_{q, {\mathcal P}_{+, {\mathcal L}_{1}}} = \ker \B^{2}_{q, \Abs} = H^{q}(M ; E).
$$
\end{lemma}

\vspace{0.2 cm}

We choose an Agmon angle $\theta$ by $- \frac{\pi}{2} < \theta < 0$.
For ${\frak D} = {\mathcal P}_{-, {\mathcal L}_{0}}$ or ${\mathcal P}_{+, {\mathcal L}_{1}}$, we define the zeta function $\zeta_{\B^{2}_{q, {\frak D}}}(s)$
and eta function $\eta_{\B_{\even, {\frak D}}}(s)$ by

\begin{eqnarray*}
\zeta_{\B^{2}_{q, {\frak D}}}(s) & = & \frac{1}{\Gamma(s)} \int_{0}^{\infty} t^{s-1}
\left( \Tr e^{- t \B^{2}_{q, {\frak D}}} - \Dim \Ker \B^{2}_{q, {\frak D}} \right) dt \hspace{0.1 cm}
= \sum_{0 \neq \lambda_{j} \in \Spec(\B^{2}_{q, {\frak D}})} \lambda_{j}^{-s} \\
\eta_{\B_{\even, {\frak D}}}(s) & = & \frac{1}{\Gamma(\frac{s+1}{2})} \int_{0}^{\infty} t^{\frac{s-1}{2}}
\Tr \left( \B e^{- t \B^{2}_{\even, {\frak D}}} \right) dt \hspace{0.1 cm} =
\sum_{0 \neq \lambda_{j} \in \Spec(\B_{\even, {\frak D}})} \Sign(\lambda_{j}) |\lambda_{j}|^{-s} .
\end{eqnarray*}

\noindent
It was shown in [11] that $\zeta_{\B^{2}_{q, {\frak D}}}(s)$ and $\eta_{\B_{\even, {\frak D}}}(s)$ have regular values at $s=0$.
We define the zeta-determinant and eta-invariant by

\begin{eqnarray}
\log \Det_{2\theta} \B^{2}_{q, {\frak D}} & := & - \zeta_{\B^{2}_{q, {\frak D}}}^{\prime}(0),  \label{E:1.15} \\
\eta(\B_{\even, {\frak D}}) & := & \frac{1}{2} \left( \eta_{\B_{\even, {\frak D}}}(0) + \Dim \Ker \B_{\even, {\frak D}} \right).   \label{E:1.16}
\end{eqnarray}

\vspace{0.2 cm}

\noindent
We denote

\begin{eqnarray}  \label{E:1.188}
& & \Omega^{q}_{-}(M, E) = \Imm \nabla \cap \Omega^{q}(M, E), \qquad
\Omega^{q}_{+}(M, E) = \Imm \Gamma \nabla \Gamma \cap \Omega^{q}(M, E),   \nonumber  \\
& & \Omega^{\even}_{\pm}(M, E) = \sum_{q = \even} \Omega^{q}_{\pm}(M, E),
\end{eqnarray}

\noindent
and denote by $\B_{\even}^{\pm}$ the restriction of $\hspace{0.1 cm} \B_{\even} \hspace{0.1 cm}$ to $\hspace{0.1 cm} \Omega^{\even}_{\pm}(M, E)$.
The graded zeta-determinant $\Det_{\gr, \theta} (\B_{\even, {\frak D}})$ of $\B_{\even}$ with respect to the boundary condition ${\frak D}$
is defined by

$$
\Det_{\gr, \theta} (\B_{\even, {\frak D}}) =  \frac{\Det_{\theta} \B^{+}_{\even, {\frak D}}}
{\Det_{\theta} \left( - \B^{-}_{\even, {\frak D}} \right)} .
$$

\vspace{0.2 cm}

We next define the projections ${\widetilde {\mathcal P}}_{0}$,
${\widetilde {\mathcal P}}_{1} : \Omega^{\bullet}(Y, E|_{Y}) \oplus \Omega^{\bullet}(Y, E|_{Y}) \rightarrow
\Omega^{\bullet}(Y, E|_{Y}) \oplus \Omega^{\bullet}(Y, E|_{Y})$ as follows.
For $\phi \in \Omega^{q}(M, E)$

\vspace{0.2 cm}

$$
{\widetilde {\mathcal P}}_{0} (\phi|_{Y}) = \begin{cases} {\mathcal P}_{-, {\mathcal L}_{0}} (\phi|_{Y}) \quad
\text{if} \quad q \quad \text{is} \quad \text{even} \\
{\mathcal P}_{+, {\mathcal L}_{1}} (\phi|_{Y}) \quad \text{if} \quad q \quad \text{is} \quad \text{odd} ,
\end{cases}
\qquad
{\widetilde {\mathcal P}}_{1} (\phi|_{Y}) = \begin{cases} {\mathcal P}_{+, {\mathcal L}_{1}} (\phi|_{Y}) \quad
\text{if} \quad q \quad \text{is} \quad \text{even} \\
{\mathcal P}_{-, {\mathcal L}_{0}} (\phi|_{Y}) \quad \text{if} \quad q \quad \text{is} \quad \text{odd} .
\end{cases}
$$

\vspace{0.2 cm}

\noindent
We denote by

\begin{equation}  \label{E:1.199}
l_{q} := \Dim \Ker \B_{Y, q}^{2}, \qquad l_{q}^{+} := \Dim {\mathcal K} \cap \Ker \B_{Y, q}^{2}, \qquad \text{and} \qquad
l_{q}^{-} := \Dim \Gamma^{Y} {\mathcal K} \cap  \Ker \B_{Y, q}^{2},
\end{equation}

\vspace{0.2 cm}

\noindent
so that $l_{q} = l_{q}^{+} + l_{q}^{-}$ and $l_{q}^{-} = l_{m-1-q}^{+}$.
Simple computation shows that $\log \Det_{\gr, \theta} (\B_{\even, {\mathcal P}_{-, {\mathcal L}_{0}}})$ and
$\log \Det_{\gr, \theta} (\B_{\even, {\mathcal P}_{+, {\mathcal L}_{1}}})$ are described as follows ([11]).

\begin{eqnarray}
(1) \label{E:1.17}
\hspace{0.2 cm} \log \Det_{\gr, \theta} (\B_{\even, {\mathcal P}_{-, {\mathcal L}_{0}}}) & = &
\frac{1}{2} \sum_{q=0}^{m} (-1)^{q+1} \cdot q \cdot \log \Det_{2\theta} \B^{2}_{q, {\widetilde {\mathcal P}_{0}}}
-  i \pi \hspace{0.1 cm} \eta(\B_{\even, {\mathcal P}_{-, {\mathcal L}_{0}}}) \nonumber \\
& + & \frac{\pi i }{2} \left( \frac{1}{4} \sum_{q=0}^{m-1} \zeta_{\B_{Y, q}^{2}}(0) + \sum_{q=0}^{r-2}(r-1-q) (l_{q}^{+} - l_{q}^{-})  \right). \\
(2) \label{E:1.18}
\hspace{0.2 cm}  \log \Det_{\gr, \theta} (\B_{\even, {\mathcal P}_{+, {\mathcal L}_{1}}})  & = &
\frac{1}{2} \sum_{q=0}^{m} (-1)^{q+1} \cdot q \cdot \log \Det_{2\theta} \B^{2}_{q, {\widetilde {\mathcal P}_{1}}}
- i \pi \hspace{0.1 cm} \eta(\B_{\even, {\mathcal P}_{+, {\mathcal L}_{1}}}) \nonumber \\
& - & \frac{\pi i }{2} \left( \frac{1}{4} \sum_{q=0}^{m-1} \zeta_{\B_{Y, q}^{2}}(0)  +  \sum_{q=0}^{r-2}(r-1-q) (l_{q}^{+} - l_{q}^{-}) \right).
\end{eqnarray}

\vspace{0.2 cm}

To define the refined analytic torsion we introduce the trivial connection $\nabla^{\trivial}$ acting on the trivial line bundle $M \times {\Bbb C}$
and define the corresponding odd signature operator $\B^{\trivial}_{\even} : \Omega^{\even}(M, {\Bbb C}) \rightarrow \Omega^{\even}(M, {\Bbb C})$
in the same way as
(\ref{E:1.2}). The eta invariant $\eta ( \B^{\trivial}_{\even, {\mathcal P}_{-, {\mathcal L}_{0}}/{\mathcal P}_{+, {\mathcal L}_{1}} } )$
associated to $\B^{\trivial}_{\even}$
and subject to the boundary condition ${\mathcal P}_{-, {\mathcal L}_{0}}/{\mathcal P}_{+, {\mathcal L}_{1}}$  is defined in the same way
as in (\ref{E:1.16}) by simply replacing $\B_{\even, {\mathcal P}_{-, {\mathcal L}_{0}}/{\mathcal P}_{+, {\mathcal L}_{1}} }$
by $\B^{\trivial}_{\even, {\mathcal P}_{-, {\mathcal L}_{0}}/{\mathcal P}_{+, {\mathcal L}_{1}} }$.
When $\nabla$ is acyclic in the de Rham complex,
the refined analytic torsion subject to the boundary condition ${\mathcal P}_{-, {\mathcal L}_{0}}/{\mathcal P}_{+, {\mathcal L}_{1}}$ is defined
by

\begin{eqnarray}
\log \rho_{\an, {\mathcal P}_{-, {\mathcal L}_{0}}} (g^{M}, \nabla) & = & \log \Det_{\gr, \theta} (\B_{\even, {\mathcal P}_{-, {\mathcal L}_{0}}}) +
\frac{\pi i}{2} (\rank E) \eta_{ \B^{\trivial}_{\even, {\mathcal P}_{-, {\mathcal L}_{0}}} } (0)  \\
\log \rho_{\an, {\mathcal P}_{+, {\mathcal L}_{1}}} (g^{M}, \nabla) & = & \log \Det_{\gr, \theta} (\B_{\even, {\mathcal P}_{+, {\mathcal L}_{1}}}) +
\frac{\pi i}{2} (\rank E) \eta_{ \B^{\trivial}_{\even, {\mathcal P}_{+, {\mathcal L}_{1}}} } (0)
\end{eqnarray}

\vspace{0.2 cm}

\noindent
The refined analytic torsion on a closed manifold is defined similarly.

On the other hand, B. Vertman also discussed the refined analytic torsion on a compact manifold with boundary
in a different way. He used the minimal and maximal extensions
of a flat connection, which will be explained briefly in Section 4. In this paper we are going to compare
$\rho_{\an, {\mathcal P}_{-, {\mathcal L}_{0}}} (g^{M}, \nabla)$ and $\rho_{\an, {\mathcal P}_{+, {\mathcal L}_{1}}} (g^{M}, \nabla)$ with
the refined analytic torsion constructed by Vertman when the odd signature operator comes from an acyclic Hermitian connection.
For this purpose in the next two sections we are going to compare the Ray-Singer analytic torsion and eta invariant subject to the boundary
condition ${\mathcal P}_{-, {\mathcal L}_{0}}$ and ${\mathcal P}_{+, {\mathcal L}_{1}}$ with those subject to the relative and absolute boundary conditions,
respectively.

\vspace{0.3 cm}

\section{Comparison of the Ray-Singer analytic torsions}

\vspace{0.2 cm}

In this section we are going to compare the Ray-Singer analytic torsion subject to the boundary condition ${\mathcal P}_{-, {\mathcal L}_{0}}$
 and ${\mathcal P}_{+, {\mathcal L}_{1}}$
with the Ray-Singer analytic torsion subject to the relative and absolute boundary conditions. For this purpose we are going to use the BFK-gluing formula and
the method of the adiabatic limit for stretching the cylinder part.
In this section we do not assume the vanishing of de Rham cohomologies. We only assume that the metric is a product one near boundary.
We recall that $(M, g^{M})$ is a compact oriented Riemannian manifold with boundary $Y$
with a collar neighbrhood $N = [0, 1) \times Y$ and $g^{M}$ is assumed to be a product metric on $N$.
We denote by $M_{1, 1} = [0, 1] \times Y$ and $M_{2} = M - N$.
To use the adiabatic limit we stretch the cylinder part $M_{1, 1}$ to the cylinder of length $r$.
We denote $M_{1, r} = [0, r] \times Y$ with the product metric and

$$
M_{r} = M_{1, r} \cup_{Y_{r}} M_{2} \quad \text{with} \hspace{0.2 cm} Y_{r} = \{ r \} \times Y.
$$

\noindent
Then we can extend the bundle $E$ and
the odd signature operator $\B$ on $M$ to $M_{r}$ in the natural way and we denote these extensions
by $E_{r}$ and $\B(r)$ ($\B = \B(1)$). We denote the restriction of $\B(r)$ to $M_{1, r}$, $M_{2}$ by $\B_{M_{1, r}}$, $\B_{M_{2}}$.
It is well known (cf. [2], [13]) that the Dirichlet boundary value problem for $\B_{q}^{2}$ on $M_{2}$ has a unique
solution, {\it i.e.} for $f + dx \wedge g \in \Omega^{q}(M_{2}, E|_{M_{2}})|_{Y_{r}}$,
there exists a unique $\psi \in \Omega^{q}(M_{2}, E|_{M_{2}})$
such that

$$
\B_{q}^{2}\psi = 0, \qquad \psi|_{Y_{r}} = f + dx \wedge g.
$$

\noindent
Let ${\frak D}$ be one of the following boundary conditions :
${\mathcal P}_{-, {\mathcal L}_{0}}$, ${\mathcal P}_{+, {\mathcal L}_{1}}$,
the absolute boundary condition, the relative boundary condition or the Dirichlet boundary condition.
We define the Neumann jump operators $Q_{q, 1, {\frak D}}(r), \hspace{0.1 cm} Q_{q, 2} $ and the Dirichlet-to-Neumann operator
$R_{q, {\frak D}}(r)$

$$
Q_{q, 1, {\frak D}}(r), \hspace{0.1 cm} Q_{q, 2}, \hspace{0.1 cm} R_{q, {\frak D}}(r)
: \Omega^{q}(Y_{r}, E|_{Y_{r}}) \oplus \Omega^{q-1}(Y_{r}, E|_{Y_{r}}) \rightarrow
\Omega^{q}(Y_{r}, E|_{Y_{r}}) \oplus \Omega^{q-1}(Y_{r}, E|_{Y_{r}})
$$

\vspace{0.2 cm}

\noindent
as follows.
For $\left( \begin{array}{clcr} f \\ g \end{array} \right) \in \Omega^{q}(Y_{r}, E|_{Y_{r}}) \oplus \Omega^{q-1}(Y_{r}, E|_{Y_{r}})$, we choose
$\phi \in \Omega^{q}(M_{1, r}, E|_{M_{1, r}})$ and $\psi \in \Omega^{q}(M_{2}, E|_{M_{2}})$ such that

\begin{equation}  \label{E:2.1}
\B_{q, M_{1, r}}^{2}\phi = 0, \qquad \B_{q, M_{2}}^{2} \psi = 0,  \qquad \phi|_{Y_{r}} = \psi|_{Y_{r}} = f + dx \wedge g,
\qquad {\frak D}(\phi|_{Y_{0}}) = 0.
\end{equation}

\noindent
Then we define
\begin{eqnarray} \label{E:2.2}
Q_{q, 1, {\frak D}}(r) (f) \hspace{0.1 cm} = \hspace{0.1 cm} (\nabla_{\partial_{x}} \phi)|_{Y_{r}}, \qquad
Q_{q, 2}(f) \hspace{0.1 cm} = \hspace{0.1 cm} - \hspace{0.1 cm} (\nabla_{\partial_{x}} \psi)|_{Y_{r}}, \qquad
R_{q, {\frak D}} (r) \hspace{0.1 cm} = \hspace{0.1 cm} Q_{q, 1, {\frak D}}(r) + Q_{q, 2},
\end{eqnarray}

\noindent
where $\partial_{x}$ is the inward unit normal vector field on $N \subset M$.

\vspace{0.2 cm}

We denote by $\B^{2}_{q, M_{1, r}, {\frak D}, D}$ ($\B^{2}_{q, M_{2}, D}$)
the restriction of $\B^{2}_{q}(r)$ to $M_{1, r}$
($M_{2}$) subject to the boundary condition ${\frak D}$ on $Y_{0}$ and the Dirichlet boundary condition on $Y_{r}$
(the Dirichlet boundary condition on $Y_{r}$).
We denote by $\B^{2}_{q, {\frak D}}(r)$ the operator $\B_{q}^{2}(r)$ on $M_{r}$ subject to the boundary condition ${\frak D}$
on $Y_{0}$.
The following lemma is well known (cf. [15]).

\begin{lemma} \label{Lemma:2.1}
(1) $R_{q, \frak D} (r)$ is a non-negative elliptic pseudodifferential operator of order $1$ and has the form of

\begin{equation} \label{E:2.3}
R_{q, \frak D}(r) \hspace{0.1 cm} = \hspace{0.1 cm} 2 \hspace{0.1 cm} | {\mathcal A} |  \hspace{0.1 cm} +
\hspace{0.1 cm} \text{ a smoothing operator}, \qquad
|{\mathcal A}| \hspace{0.1 cm} = \hspace{0.1 cm} \left( \begin{array}{clcr} \sqrt{\B_{Y, q}^{2}} & 0 \\ 0 & \sqrt{\B_{Y, q-1}^{2}} \end{array} \right).
\end{equation}

\noindent
(2) $\ker R_{q, \frak D}(r) = \{ \phi|_{Y_{r}} \mid \phi \in \Ker \B_{q, {\frak D}}^{2}(r) \}$.

\end{lemma}

\vspace{0.2 cm}

\noindent
Lemma \ref{Lemma:1.2} shows that $\Dim \Ker \B^{2}_{q, {\frak D}}(r)$ is a topological invariant.
Let $\Dim \Ker \B^{2}_{q, {\frak D}}(r) = k$ and
$\{\varphi_{1}, \cdots, \varphi_{k} \}$ be an orthonormal basis of $\Ker \B^{2}_{q, {\frak D}}(r)$.
We define a positive definite $k \times k$ Hermitian matrix $A_{q, {\frak D}}(r)$ by
$$
A_{q, {\frak D}}(r) = (a_{ij}), \qquad a_{ij} = \langle \varphi_{i}|_{Y_{0}}, \varphi_{j}|_{Y_{0}} \rangle_{Y_{0}}.
$$
Then the BFK-gluing formula ([7], [15], [16]) is described as follows.
Setting $l_{q} = \Dim \Ker \B_{Y, q}^{2}$, we have

\begin{eqnarray}  \label{E:2.333}
\log \Det_{2 \theta} \B^{2}_{q, {\frak D}}(r)
& = & \log \Det_{2 \theta} \B^{2}_{q, M_{1, r}, {\frak D}, D} + \log \Det_{2 \theta} \B^{2}_{q, M_{2}, D}
  + \log \Det_{2 \theta} R_{q, {\frak D}}(r) \nonumber  \\
&  & - \log 2 \cdot (\zeta_{\B^{2}_{Y, q}}(0) + \zeta_{\B^{2}_{Y, q-1}}(0) + l_{q} + l_{q-1})  - \log \ddet A_{q, {\frak D}}(r).
\end{eqnarray}

\vspace{0.2 cm}

\noindent
{\it Remark} : (1) Lemma \ref{Lemma:1.2} shows that $A_{q, {\mathcal P}_{-, {\mathcal L}_{0}}}(r) = A_{q, \rel}(r)$ and
$A_{q, {\mathcal P}_{+, {\mathcal L}_{1}}}(r) = A_{q, \Abs}(r)$.   \newline
(2) The BFK-gluing formula was proved originally on a closed manifold in [7]. But it can be extended
to a compact manifold with boundary with only minor modification when a cutting hypersurface does not intersect the boundary.

\vspace{0.2 cm}

\noindent
Lemma 2.3 in [12] shows that

\begin{eqnarray*}
\Spec \left( \B^{2}_{q, M_{1, r}, {\mathcal P}_{-, {\mathcal L}_{0}}, D} \right) \cup \Spec \left( \B^{2}_{q, M_{1, r}, {\mathcal P}_{+, {\mathcal L}_{1}}, D} \right)
\hspace{0.1 cm} = \hspace{0.1 cm}
\Spec \left( \B^{2}_{q, M_{1, r}, \rel, D} \right) \cup \Spec \left( \B^{2}_{q, M_{1, r}, \Abs, D} \right),
\end{eqnarray*}

\vspace{0.2 cm}

\noindent
which together with (\ref{E:2.333}) yields the following result.

\begin{corollary}  \label{Corollary:2.8}
\begin{eqnarray*}
& & \sum_{q=0}^{m} (-1)^{q+1} q \left( \log \Det \B^{2}_{q, {\widetilde {\mathcal P}}_{0}}(r) + \log \Det \B^{2}_{q, {\widetilde {\mathcal P}}_{1}}(r)
- \log \Det \B^{2}_{q, \rel}(r) - \log \Det \B^{2}_{q, \Abs}(r)  \right) \\
& = & \sum_{q=0}^{m} (-1)^{q+1} q \left( \log \Det R_{q, {\mathcal P}_{-, {\mathcal L}_{0}}}(r) + \log \Det R_{q, {\mathcal P}_{+, {\mathcal L}_{1}}}(r) -
 \log \Det R_{q, {\mathcal P}_{\rel}}(r) -  \log \Det R_{q, {\mathcal P}_{\Abs}}(r) \right).
\end{eqnarray*}
\end{corollary}

\vspace{0.2 cm}

We next discuss the Dirichlet-to-Neumann operator $R_{q, {\frak D}}(r)$ defined by
$R_{q, {\frak D}}(r) = Q_{q, 1, {\frak D}}(r) +  Q_{q, 2}$, where ${\frak D}$ is one of
${\mathcal P}_{-, {\mathcal L}_{0}}$, ${\mathcal P}_{+, {\mathcal L}_{1}}$, the absolute or the relative boundary condition.
The following lemma is straightforward (Lemma 2.8 in [12]).

\vspace{0.2 cm}

\begin{lemma}  \label{Lemma:2.8}
$$
R_{q, {\frak D}}(r) \hspace{0.1 cm} = \hspace{0.1 cm} Q_{q, 2} + | {\mathcal A} | + {\mathcal K}_{q, {\frak D}}(r),
$$
where

\begin{eqnarray*}
{\mathcal K}_{q, {\mathcal P}_{-, {\mathcal L}_{0}}}(r), \hspace{0.2 cm}  {\mathcal K}_{q, {\mathcal P}_{+, {\mathcal L}_{1}}}(r)  &  = &
 \begin{cases}
\frac{ 2 \sqrt{\B_{Y}^{2}} }{e^{2 \sqrt{\B_{Y}^{2}} r} \hspace{0.1 cm} - \hspace{0.1 cm} 1}, \hspace{0.2 cm}
 - \frac{ 2 \sqrt{\B_{Y}^{2}} }{e^{2 \sqrt{\B_{Y}^{2}} r} \hspace{0.1 cm} + \hspace{0.1 cm} 1}
\hspace{0.4 cm} \text{on} \hspace{0.2 cm}
\Omega^{q}_{-}(Y, E|_{Y}) \oplus \Omega^{q-1}_{-}(Y, E|_{Y}) \\
\hspace{1.0 cm}  \frac{1}{r} \hspace{0.5 cm}, \hspace{1.0 cm} 0 \hspace{1.4 cm} \text{on} \hspace{0.2 cm}
{\mathcal K} \cap \left( \Omega^{q}(Y, E|_{Y}) \oplus \Omega^{q-1}(Y, E|_{Y}) \right)\\
 -  \frac{ 2 \sqrt{\B_{Y}^{2}} }{e^{2 \sqrt{\B_{Y}^{2}} r} \hspace{0.1 cm} + \hspace{0.1 cm} 1}, \hspace{0.2 cm}
 \frac{ 2 \sqrt{\B_{Y}^{2}} }{e^{2 \sqrt{\B_{Y}^{2}} r} \hspace{0.1 cm} - \hspace{0.1 cm} 1}
  \hspace{0.4 cm} \text{on} \hspace{0.2 cm}
 \Omega^{q}_{+}(Y, E|_{Y}) \oplus \Omega^{q-1}_{+}(Y, E|_{Y})\\
\hspace{1.0 cm}  0 \hspace{0.5 cm}  , \hspace{1.0 cm}  \frac{1}{r}  \hspace{1.4 cm} \text{on} \hspace{0.2 cm}
\Gamma^{Y} {\mathcal K} \cap \left( \Omega^{q}(Y, E|_{Y}) \oplus \Omega^{q-1}(Y, E|_{Y}) \right) ,
\end{cases}  \\
{\mathcal K}_{q, \rel}(r), \hspace{0.2 cm} {\mathcal K}_{q, \Abs}(r) &  = &
 \begin{cases}
 \frac{ 2 \sqrt{\B_{Y}^{2}} }{e^{2 \sqrt{\B_{Y}^{2}} r} \hspace{0.1 cm} - \hspace{0.1 cm} 1}, \hspace{0.2 cm}
  - \frac{ 2 \sqrt{\B_{Y}^{2}}}{e^{2 \sqrt{\B_{Y}^{2}} r} \hspace{0.1 cm} + \hspace{0.1 cm} 1}
\hspace{0.4 cm} \text{on} \hspace{0.2 cm} \Omega^{q}_{-}(Y, E|_{Y}) \oplus \Omega^{q}_{+}(Y, E|_{Y})  \\
\hspace{1.0 cm}  \frac{1}{r} \hspace{0.5 cm}  , \hspace{1.0 cm}  0  \hspace{1.4 cm} \text{on} \hspace{0.2 cm} \Ker \B_{Y}^{2} \cap \Omega^{q}(Y, E|_{Y})\\
- \frac{ 2 \sqrt{\B_{Y}^{2}} }{e^{2 \sqrt{\B_{Y}^{2}} r} \hspace{0.1 cm} + \hspace{0.1 cm} 1}, \hspace{0.2 cm}
\frac{ 2 \sqrt{\B_{Y}^{2}}}{e^{2 \sqrt{\B_{Y}^{2}} r} \hspace{0.1 cm} - \hspace{0.1 cm} 1}
  \hspace{0.4 cm} \text{on} \hspace{0.2 cm}  \Omega^{q-1}_{-}(Y, E|_{Y}) \oplus \Omega^{q-1}_{+}(Y, E|_{Y})  \\
\hspace{1.0 cm}  0 \hspace{0.5 cm}  , \hspace{1.0 cm} \frac{1}{r}  \hspace{1.4 cm} \text{on} \hspace{0.2 cm}
\Ker \B_{Y}^{2} \cap \Omega^{q-1}(Y, E|_{Y}) .
\end{cases}
\end{eqnarray*}
\end{lemma}

\vspace{0.2 cm}

\noindent
The above lemma and (\ref{E:1.12}) lead to the following result.

\begin{corollary}  \label{Corollary:2.98}
\begin{eqnarray*}
{\mathcal K}_{q}(r) & := & R_{q, {\mathcal P}_{-, {\mathcal L}_{0}}}(r) - R_{q, \rel}(r) \hspace{0.1 cm} = \hspace{0.1 cm}
 {\mathcal K}_{q, {\mathcal P}_{-, {\mathcal L}_{0}}}(r) - {\mathcal K}_{q, \rel}(r) \hspace{0.1 cm} = \hspace{0.1 cm}
- \left(  R_{q, {\mathcal P}_{+, {\mathcal L}_{1}}}(r) - R_{q, \Abs}(r)  \right)  \\
& = &
\begin{cases}
 \frac{ - 4 \sqrt{\B_{Y}^{2}} }{e^{2 \sqrt{\B_{Y}^{2}} r} \hspace{0.1 cm} - \hspace{0.1 cm} e^{- 2 \sqrt{\B_{Y}^{2}} r}}
\hspace{0.5 cm} \text{on} \hspace{0.3 cm}  \Omega^{q}_{+}(Y, E|_{Y})  \\
%\vspace{0.01 cm} \\
\frac{ 4 \sqrt{\B_{Y}^{2}} }{e^{2 \sqrt{\B_{Y}^{2}} r} \hspace{0.1 cm} - \hspace{0.1 cm} e^{- 2 \sqrt{\B_{Y}^{2}} r}}
  \hspace{0.5 cm} \text{on} \hspace{0.3 cm}  \Omega^{q-1}_{-}(Y, E|_{Y})  \\
%\vspace{0.01 cm} \\
- \frac{1}{r} \hspace{2.8 cm} \text{on} \hspace{0.3 cm} \Gamma^{Y} {\mathcal K}  \cap \Omega^{q}(Y, E|_{Y})\\
%\vspace{0.01 cm} \\
\frac{1}{r} \hspace{3.1 cm} \text{on} \hspace{0.3 cm} {\mathcal K}  \cap \Omega^{q-1}(Y, E|_{Y})\\
%\vspace{0.01 cm} \\
0 \hspace{3.2 cm} \text{otherwise} .
\end{cases}
\end{eqnarray*}
In particular, if $\hspace{0.1 cm} \nabla \phi = \Gamma \nabla \Gamma \phi = 0 \hspace{0.1 cm}$ for
$\hspace{0.1 cm} \phi \in \Omega^{q}(M, E)$, then $\hspace{0.1 cm} {\mathcal K}_{q}(r)(\phi|_{Y}) = 0$.
\end{corollary}

\vspace{0.2 cm}

We next discuss the limit of
$\hspace{0.1 cm}\left(\log \Det R_{q, {\mathcal P}_{-, {\mathcal L}_{0}}/{\mathcal P}_{+, {\mathcal L}_{1}}}(r) - \log \Det R_{q, \rel/\Abs}(r) \right)
\hspace{0.1 cm}$ for $r \rightarrow \infty$.
We note that

\begin{eqnarray*}
\lim_{r \rightarrow \infty} R_{q, {\mathcal P}_{-, {\mathcal L}_{0}}/{\mathcal P}_{+, {\mathcal L}_{1}}}(r) \hspace{0.1 cm} = \hspace{0.1 cm}
\lim_{r \rightarrow \infty} R_{q, \rel/\Abs}(r) \hspace{0.1 cm} = \hspace{0.1 cm}
 Q_{q, 2} \hspace{0.1 cm} + \hspace{0.1 cm} | {\mathcal A} | .
\end{eqnarray*}

\vspace{0.2 cm}

\noindent
The kernel of $\hspace{0.1 cm} Q_{q, 2} \hspace{0.1 cm} + \hspace{0.1 cm} | {\mathcal A} | \hspace{0.1 cm}$ is described as follows.
For $f \in \Omega^{q}(M_{2}, E)|_{Y_{1}}$, choose $\psi \in \Omega^{q}(M_{2}, E)$ such that
$ \B^{2} \psi = 0$ and $\psi|_{Y_{1}} = f$.
Then,

\begin{eqnarray}  \label{E:2.55}
0 \hspace{0.1 cm} = \hspace{0.1 cm} \langle \B^{2} \psi, \hspace{0.1 cm} \psi \rangle & = &
\langle \B \psi, \hspace{0.1 cm} \B \psi \rangle \hspace{0.1 cm} + \hspace{0.1 cm}
\langle (\B \psi)|_{Y_{1}}, \hspace{0.1 cm} ({\mathcal \gamma} \psi)|_{Y_{1}} \rangle_{Y_{1}}  \nonumber \\
& = & \langle \B \psi, \hspace{0.1 cm} \B \psi \rangle \hspace{0.1 cm} + \hspace{0.1 cm}
\langle (\nabla_{\partial_{x}} \psi + {\mathcal A} \psi)|_{Y_{1}}, \hspace{0.1 cm} f \rangle_{Y_{1}}  \nonumber \\
& = & \parallel \B \psi \parallel^{2} \hspace{0.1 cm} - \hspace{0.1 cm}
\langle Q_{q, 2} f, \hspace{0.1 cm} f \rangle_{Y_{1}} \hspace{0.1 cm} + \hspace{0.1 cm}
\langle {\mathcal A} f, \hspace{0.1 cm} f \rangle_{Y_{1}} ,
\end{eqnarray}

\noindent
which leads to

\begin{eqnarray}  \label{E:2.56}
\langle (Q_{q, 2} + |{\mathcal A}|)f, \hspace{0.1 cm} f \rangle_{Y_{1}} & = &
\parallel \B \psi \parallel^{2} \hspace{0.1 cm} + \hspace{0.1 cm}
\langle (|{\mathcal A}| + {\mathcal A})f, \hspace{0.1 cm} f \rangle_{Y_{1}}.
\end{eqnarray}

\vspace{0.2 cm}

\noindent
Hence, $f \in \Ker (Q_{q, 2} + |{\mathcal A}|)$ if and only if $\B \psi = 0$ and $(|{\mathcal A}| + {\mathcal A})f = 0$,
which shows that $\psi$ is expressed, on a collar neighborhood of $Y_{1}$, by

\begin{equation}  \label{E:2.77}
\psi = \sum_{\lambda_{j} \in \Spec ({\mathcal A}) \atop \lambda_{j} \leq 0} a_{j} e^{- \lambda_{j} x} \phi_{j}, \qquad \text{where}
\quad {\mathcal A} \phi_{j} = \lambda_{j} \phi_{j}.
\end{equation}

\noindent
Let $M_{\infty} := \left( (- \infty, 0] \times Y \right) \cup_{Y} M_{2}$.
We can extend $E$ and $\B$ canonically to $M_{\infty}$, which we denote by
$E_{\infty}$ and $\B_{\infty}$. Then $\psi$ in (\ref{E:2.77}) can be extended to $M_{\infty}$ as an $L^{2}$ or extended $L^{2}$-solution of $\B_{\infty}$
(for definitions of $L^{2}$ and extended $L^{2}$-solutions we refer to [1] or [3]).
Hence,

$$
\Ker (Q_{q, 2} + |{\mathcal A}|) = \{ \psi|_{Y_{1}} \mid \psi \hspace{0.1 cm} \text{is} \hspace{0.1 cm} \text{an} \hspace{0.1 cm}
L^{2} \hspace{0.1 cm} \text{or} \hspace{0.1 cm} \text{extended}\hspace{0.1 cm} L^{2}\text{-solution} \hspace{0.1 cm}
\text{of} \hspace{0.1 cm} \B_{\infty} \hspace{0.1 cm} \text{in}
\hspace{0.1 cm} \Omega^{q}(M_{\infty}, E_{\infty}) \hspace{0.1 cm} \}.
$$

\vspace{0.2 cm}

\noindent
This fact together with Lemma \ref{Lemma:2.1}, Corollary \ref{Corollary:2.98} and (\ref{E:1.12}) leads to the following result.

\begin{corollary} \label{Corollary:2.99}
Let $f \in \Ker (Q_{q, 2} + |{\mathcal A}|)$ or $ f \in \Ker R_{q, {\frak D}}(r)$,
where ${\frak D}$ is one of ${\widetilde {\mathcal P}}_{0}$, ${\widetilde {\mathcal P}}_{1}$, the absolute or the relative boundary condition.
Then,
$\hspace{0.1 cm} {\mathcal K}_{q}(r) f \hspace{0.1 cm} = \hspace{0.1 cm} 0$.

\end{corollary}

\vspace{0.2 cm}

Since $\Ker R_{q, {\mathcal P}_{-, {\mathcal L}_{0}}}(r) = \Ker R_{q, \rel}(r)$ (Lemma \ref{Lemma:1.2}), we have

\begin{eqnarray}  \label{E:2.110}
& & \log \Det R_{q, {\mathcal P}_{-, {\mathcal L}_{0}}}(r) - \log \Det R_{q, \rel}(r) \nonumber \\
& = &
\log \Det \left( R_{q, {\mathcal P}_{-, {\mathcal L}_{0}}}(r) + \pr_{\Ker R_{q, {\mathcal P}_{-, {\mathcal L}_{0}}}(r)} \right)
- \log \Det \left( R_{q, \rel}(r) + \pr_{\Ker R_{q, \rel}(r)} \right)  \nonumber \\
& = & \int_{0}^{1} \frac{d}{ds} \log \Det \left( R_{q, \rel}(r) + \pr_{\Ker R_{q, \rel}(r)} + s \left( R_{q, {\mathcal P}_{-, {\mathcal L}_{0}}}(r) -
R_{q, \rel} (r) \right) \right) ds  \nonumber \\
& = & \int_{0}^{1} \Tr \left( \left( R_{q, \rel}(r) + \pr_{\Ker R_{q, \rel}(r)} + s \hspace{0.1 cm} {\mathcal K}_{q}(r) \right)^{-1}
{\mathcal K}_{q}(r) \right) ds ,
\end{eqnarray}

\vspace{0.2 cm}

\noindent
where $\hspace{0.1 cm} \pr_{\Ker R_{q, {\mathcal P}_{-, {\mathcal L}_{0}}}(r)} \hspace{0.1 cm}$ is the orthogonal projection onto
$\hspace{0.1 cm} \Ker R_{q, {\mathcal P}_{-, {\mathcal L}_{0}}}(r)$.
We denote

\begin{eqnarray*}
X(r) & := & R_{q, \rel}(r) + \pr_{\Ker R_{q, \rel}(r)} + s {\mathcal K}_{q}(r)  \hspace{0.1 cm} = \hspace{0.1 cm}
 Q_{q, 2} + |{\mathcal A}| + {\mathcal K}_{q, \rel}(r) + \pr_{\Ker R_{q, \rel}(r)} + s {\mathcal K}_{q}(r),   \\
{\mathcal M} & := & \{ \phi \mid {\mathcal K}_{q}(r) (\phi) \neq 0 \} \hspace{0.1 cm} = \hspace{0.1 cm}
\Omega^{q}_{+}(Y, E|_{Y}) \oplus \Omega^{q-1}_{-}(Y, E|_{Y}) \oplus \Gamma^{Y} {\mathcal K}  \cap \Omega^{q}(Y, E|_{Y})
\oplus {\mathcal K}  \cap \Omega^{q-1}(Y, E|_{Y}) , \\
{\mathcal W}(r) & := & X(r)^{-1} \left( {\mathcal M}  \right) .
\end{eqnarray*}

\vspace{0.2 cm}

\noindent
Then, we have

\begin{eqnarray}  \label{E:2.220}
\Tr \left( X(r)^{-1} {\mathcal K}_{q}(r) \right) \hspace{0.1 cm} = \hspace{0.1 cm}
\Tr \left( X(r)^{-1} {\mathcal K}_{q}(r) : {\mathcal M} \rightarrow {\mathcal W}(r) \right).
\end{eqnarray}

\vspace{0.2 cm}

\begin{lemma}  \label{Lemma:100}
$$
{\mathcal W}(r) \hspace{0.1 cm} \cap \hspace{0.1 cm} \Ker \left( Q_{q, 2} + |{\mathcal A}| \right)
\hspace{0.1 cm} = \hspace{0.1 cm} \{ 0 \}.
$$
\end{lemma}

\begin{proof}
Let $\phi \in \Ker ( Q_{q, 2} + |{\mathcal A}|)$. Corollary \ref{Corollary:2.99} shows that $ X(r) \phi \hspace{0.1 cm} = \hspace{0.1 cm}
{\mathcal K}_{q, \rel}(r) \phi + \pr_{\Ker R_{q, \rel}(r)} \phi$.
Corollary \ref{Corollary:2.99} again shows $\phi$, $\pr_{\Ker R_{q, \rel}(r)} \phi \in {\mathcal M}^{\perp}$.
From $\phi \in {\mathcal M}^{\perp}$, we have ${\mathcal K}_{q, \rel}(r) \phi \in {\mathcal M}^{\perp}$, which shows that $X(r) \phi \in {\mathcal M}^{\perp}$.
Since $X(r)$ is invertible, this completes the proof of the lemma.
\end{proof}

\vspace{0.2 cm}

\noindent
We denote by ${\mathcal P}(r) : {\mathcal W}(r) \rightarrow \Ker \left( Q_{q, 2} + |{\mathcal A}| \right)$
the orthogonal projection from ${\mathcal W}(r)$ into $\Ker \left( Q_{q, 2} + |{\mathcal A}| \right)$.
We let
\begin{eqnarray*}
{\mathcal W}_{0}(r) := \Ker {\mathcal P}(r), \qquad {\mathcal W}_{1}(r) := {\mathcal W}(r) \ominus {\mathcal W}_{0}(r).
\end{eqnarray*}

\noindent
Since $\Ker \left( Q_{q, 2} + |{\mathcal A}| \right)$ is finite dimensional, so is ${\mathcal W}_{1}(r)$.
Let $\{ \phi_{1}, \cdots, \phi_{k} \}$ be an orthonormal basis for ${\mathcal W}_{1}(r)$.
For each $1 \leq i \leq k$, $\hspace{0.1 cm} \phi_{i}$ is expressed by

\begin{eqnarray*}
\phi_{i} \hspace{0.1 cm} = \hspace{0.1 cm} \psi_{i} \hspace{0.1 cm} + \hspace{0.1 cm} \varphi_{i},
\end{eqnarray*}

\noindent
where $0 \neq \psi_{i} \in \left( \Ker \left( Q_{q, 2} + |{\mathcal A}| \right) \right)^{\perp}$ and
$0 \neq \varphi_{i} \in \Ker \left( Q_{q, 2} + |{\mathcal A}| \right)$.
We put

\begin{eqnarray*}
c_{0} & := & \Min \{ \hspace{0.1 cm} \parallel \psi_{i} \parallel \hspace{0.1 cm} \mid \hspace{0.1 cm} 1 \leq i \leq k \} \hspace{0.1 cm} > \hspace{0.1 cm} 0,
\end{eqnarray*}

\vspace{0.2 cm}

\noindent
which leads to the following result.

\begin{lemma}  \label{Lemma:101}
For any $\phi \in {\mathcal W}(r)$, $\phi$ is expressed by
$\phi = \psi + \varphi$, where $\psi \in \left( \Ker \left( Q_{q, 2} + |{\mathcal A}| \right) \right)^{\perp}$ and
$\varphi \in \Ker \left( Q_{q, 2} + |{\mathcal A}| \right)$.
Then $\parallel \psi \parallel \geq c_{0} \parallel \phi \parallel$.
\end{lemma}

\vspace{0.2 cm}

\begin{lemma} \label{Lemma:102}
Let $\lambda_{1} > 0$ be the first nonzero eigenvalue of $Q_{q, 2} + |{\mathcal A}| $.
Then there exists $R_{0} > 0$ such that
for $r > R_{0}$ and $f \in {\mathcal M}$,
$$
\parallel X(r)^{-1} f \parallel \leq \frac{2}{c_{0} \lambda_{1}} \parallel f \parallel.
$$
Hence, for $r > R_{0}$ we have
$| \Tr \left(  X(r)^{-1} {\mathcal K}_{q}(r) \right) | \hspace{0.1 cm} \leq \hspace{0.1 cm} \frac{2}{c_{0} \lambda_{1}} | \Tr {\mathcal K}_{q}(r) | $ ,
which shows that
$$
\lim_{r \rightarrow \infty } | \Tr \left(  X(r)^{-1} {\mathcal K}_{q}(r) \right) |
\hspace{0.1 cm} = \hspace{0.1 cm} 0.
$$
\end{lemma}

\begin{proof}
It's enough to prove that $\parallel X(r) \phi \parallel \geq \frac{c_{0} \lambda_{1}}{2} \parallel \phi \parallel$
for $\phi \in {\mathcal W}(r)$ and $r$ large enough.
As in Lemma \ref{Lemma:101}, we write $\phi = \psi + \varphi$, where
$\psi \in \left( \Ker \left( Q_{q, 2} + |{\mathcal A}| \right) \right)^{\perp}$,
$\varphi \in \Ker \left( Q_{q, 2} + |{\mathcal A}| \right)$, and $\parallel \psi \parallel \geq c_{0} \parallel \phi \parallel$.
Corollary \ref{Corollary:2.99} shows that

\begin{eqnarray*}
X(r) \phi & = & \left( \left( Q_{q, 2} + |{\mathcal A}| \right) \psi + \pr_{\Ker R_{q, \rel}(r)} \varphi \right) +
\left( {\mathcal K}_{q, \rel}(r) \psi + \pr_{\Ker R_{q, \rel}(r)} \psi + s {\mathcal K}_{q}(r) \psi + {\mathcal K}_{q, \rel}(r) \varphi \right).
\end{eqnarray*}

\noindent

\noindent
We note that

\begin{eqnarray*}
| \langle \left( Q_{q, 2} + |{\mathcal A}| \right) \psi , \hspace{0.2 cm} \pr_{\Ker R_{q, \rel}(r)} \varphi \rangle |
& = & | \langle \psi , \hspace{0.2 cm} \left( Q_{q, 2} + |{\mathcal A}| \right) \pr_{\Ker R_{q, \rel}(r)} \varphi \rangle | \\
& = & | \langle \psi , \hspace{0.2 cm} - {\mathcal K}_{q, \rel}(r) \pr_{\Ker R_{q, \rel}(r)} \varphi \rangle |,
\end{eqnarray*}

\noindent
which tends to $0$ as $r \rightarrow \infty$.
Similarly,

\begin{eqnarray*}
\parallel \pr_{\Ker R_{q, \rel}(r)} \psi \parallel^{2}
& = & \langle \pr_{\Ker R_{q, \rel}(r)} \psi, \hspace{0.2 cm} \pr_{\Ker R_{q, \rel}(r)} \psi \rangle  \hspace{0.1 cm} = \hspace{0.1 cm}
\langle \pr_{\Ker R_{q, \rel}(r)} \psi, \hspace{0.2 cm}  \psi \rangle  \\
& = & \langle \left( Q_{q, 2} + |{\mathcal A}| \right) \pr_{\Ker R_{q, \rel}(r)} \psi, \hspace{0.2 cm}
\left( Q_{q, 2} + |{\mathcal A}| + \pr_{\Ker \left( Q_{q, 2} + |{\mathcal A}| \right)} \right)^{-1} \psi \rangle  \\
& = &
\langle - {\mathcal K}_{q, \rel}(r) \pr_{\Ker R_{q, \rel}(r)} \psi, \hspace{0.2 cm}
\left( Q_{q, 2} + |{\mathcal A}| + \pr_{\Ker \left( Q_{q, 2} + |{\mathcal A}| \right)} \right)^{-1} \psi \rangle ,
\end{eqnarray*}

\noindent
which tends to $0$ as $r \rightarrow \infty$.
Hence, we have

\begin{eqnarray}  \label{E:2.221}
\parallel X(r) \phi \parallel & \geq &
\left( \parallel \left( Q_{q, 2} + |{\mathcal A}| \right) \psi + \pr_{\Ker R_{q, \rel}(r)} \varphi \parallel \right)  \nonumber \\
& & - \left( \parallel {\mathcal K}_{q, \rel}(r) \psi \parallel + \parallel \pr_{\Ker R_{q, \rel}(r)} \psi \parallel +
\parallel s {\mathcal K}_{q}(r) \psi \parallel + \parallel {\mathcal K}_{q, \rel}(r) \varphi \parallel \right) \nonumber \\
& \geq & \left( \parallel \left( Q_{q, 2} + |{\mathcal A}| \right) \psi \parallel^{2} + \parallel \pr_{\Ker R_{q, \rel}(r)} \varphi \parallel^{2}
-  2
| \langle \left( Q_{q, 2} + |{\mathcal A}| \right) \psi , \hspace{0.1 cm} \pr_{\Ker R_{q, \rel}(r)} \varphi \rangle |
\right)^{\frac{1}{2}} + o(r)  \nonumber\\
& \geq & \left( \parallel \left( Q_{q, 2} + |{\mathcal A}| \right) \psi \parallel^{2}
- \hspace{0.1 cm} 2 \hspace{0.1 cm}
| \langle \left( Q_{q, 2} + |{\mathcal A}| \right) \psi , \hspace{0.2 cm} \pr_{\Ker R_{q, \rel}(r)} \varphi \rangle |
\right)^{\frac{1}{2}}   + \hspace{0.1 cm} o(r) ,
\end{eqnarray}

\noindent
which completes the proof of the lemma.
\end{proof}

\vspace{0.2 cm}

\noindent
The above lemma with (\ref{E:2.110}) leads to the following result.

\vspace{0.2 cm}

\begin{eqnarray}  \label{E:2.222}
\lim_{r \rightarrow \infty} \left( \log \Det R_{q, {\mathcal P}_{-, {\mathcal L}_{0}}}(r) - \log \Det R_{q, \rel}(r) \right)
& = & 0.
\end{eqnarray}

\noindent
By the same method, we have

\begin{eqnarray}  \label{E:2.223}
\lim_{r \rightarrow \infty} \left( \log \Det R_{q, {\mathcal P}_{+, {\mathcal L}_{1}}}(r) - \log \Det R_{q, \Abs}(r) \right)
& = & 0.
\end{eqnarray}

\noindent
Corollary \ref{Corollary:2.8} with (\ref{E:2.222}) and (\ref{E:2.223}) yields

\begin{eqnarray}
\lim_{r \rightarrow \infty} \sum_{q=0}^{m} (-1)^{q+1} q \left( \log \Det \B^{2}_{q, {\widetilde {\mathcal P}}_{0}}(r) + \log \Det \B^{2}_{q, {\widetilde {\mathcal P}}_{1}}(r) - \log \Det \B^{2}_{q, \rel}(r) - \log \Det \B^{2}_{q, \Abs}(r)  \right)  =  0.
\end{eqnarray}

\vspace{0.2 cm}

The proof of the following lemma is a verbatim repetition of the proof of Theorem 7.6 in [17] (cf. Theorem 2.1 in [19]).

\begin{lemma}  \label{Lemma:2.11}
Let $M$ be a compact manifold with boundary $Y$ and $N$ be a collar neighborhood of $Y$.
We suppose that $\{ g_{t}^{M} \mid - \delta_{0} < t < \delta_{0} \}$ is a family of metrics such that each $g^{M}_{t}$ is a product metric and
does not vary on $N$.
Let ${\frak D}$ be one of
${\widetilde {\mathcal P}}_{0}$, ${\widetilde {\mathcal P}}_{1}$, the absolute or the relative boundary condition.
We denote by $\B^{2}_{q, {\frak D}}(t)$ the square of the odd signature operator acting on $q$-forms subject to ${\frak D}$ and with respect to the metric $g^{M}_{t}$. Then we have

$$
\frac{d}{dt} \left( \sum_{q=0}^{m} (-1)^{q+1} \cdot q \cdot \log \Det \B^{2}_{q, {\frak D}}(t) \right) \hspace{0.1 cm} = \hspace{0.1 cm}
\sum_{q=0}^{m} (-1)^{q+1} \cdot q \cdot \Tr \left( \pr_{{\mathcal H}^{q}_{{\frak D}}(t)} \ast_{t} (\frac{d}{dt} \ast_{t}) \right),
$$

\noindent
where $\pr_{{\mathcal H}^{q}_{{\frak D}}(t)}$ is the orthogonal projection onto the kernel of $\B^{2}_{q, {\frak D}}(t)$ and $\ast_{t}$ is the Hodge
star operator with respect to the metric $g^{M}_{t}$.
\end{lemma}

\vspace{0.2 cm}

\noindent
Lemma \ref{Lemma:2.11} and Lemma \ref{Lemma:1.2} lead to the following corollary.

\begin{corollary}  \label{Corollary:2.12}
We assume the same assumptions as in Lemma \ref{Lemma:2.11}. Then,
$$
\frac{d}{dt} \sum_{q=0}^{m} (-1)^{q+1} \cdot q \cdot \left(
\log \Det \B^{2}_{q, {\widetilde {\mathcal P}}_{0}}(t) + \log \Det \B^{2}_{q, {\widetilde {\mathcal P}}_{1}}(t) -
\log \Det \B^{2}_{q, \rel}(t) - \log \Det \B^{2}_{q, \Abs}(t) \right) \hspace{0.1 cm} = \hspace{0.1 cm} 0.
$$
\end{corollary}

\vspace{0.3 cm}

We fix $\delta_{0} > 0$ sufficiently small and
choose a smooth function $\hspace{0.1 cm} f(r, u) : [1, \infty) \times [0, 1] \rightarrow [0, \infty), \hspace{0.1 cm}$ ($r \geq 1$) such that
for each $r$
$$
\supp_{u}f(r, u) \subset [\delta_{0}, 1 - \delta_{0}], \quad \int_{0}^{1} f(r, u) du = r - 1 , \quad
\text{and} \quad f(1, u) \equiv 0.
$$
Setting $F(r, u) = u + \int_{0}^{u} f(r, t) dt$,
$\quad F_{r} := F(r, \cdot) : [0, 1] \rightarrow [0, r] \hspace{0.1 cm}$ is a diffeomorphism satisfying

\begin{eqnarray*}
F_{r}(u) = \begin{cases} u  \hspace{1.2 cm} \text{for} \hspace{0.2 cm}  0 \leq u \leq \delta_{0}  \\
 u + r -1  \hspace{0.6 cm} \text{for} \hspace{0.2 cm}  1 - \delta_{0} \leq u \leq 1.
\end{cases}
\end{eqnarray*}

\noindent
Let $g^{M}_{r}$ be a metric on $M_{r} : = \left( [0, r] \times Y \right) \cup_{\{ r \} \times Y} M_{2}$
which is the extension of $g^{M}$ by a product one on $[0, r] \times Y$.
Then $F_{r}^{\ast}g^{M}_{r}$ is a metric on $M$, which is
$\left( \begin{array}{clcr} F_{r}^{\prime}(u)^{2} & 0 \\ 0 & g_{Y} \end{array} \right)$ on $[0, 1] \times Y$.
Hence, $F_{r}^{\ast}g^{M}_{r}$ is a metric on $M$ which is a product one near $Y$. Furthermore,
$(M, F_{r}^{\ast}g^{M}_{r})$ and $(M_{r}, g^{M}_{r})$ are isometric.
Let ${\tilde \B}(r)$ and $\B(r)$ be the odd signature operators defined on $M$ and $M_{r}$ associated to the metrics $F_{r}^{\ast}g^{M}_{r}$
and $g^{M}_{r}$, respectively.
Then Corollary \ref{Corollary:2.12} leads to the following equalities.

\begin{eqnarray*}
& & \sum_{q=0}^{m} (-1)^{q+1} q \cdot
\left( \log \Det_{2 \theta} \B_{q, {\widetilde {\mathcal P}}_{0}}^{2} + \log \Det_{2 \theta} \B_{q, {\widetilde {\mathcal P}}_{1}}^{2}
- \log \Det_{2 \theta} \B_{q, \rel}^{2} - \log \Det_{2 \theta} \B_{q, \Abs}^{2} \right)  \\
& = & \sum_{q=0}^{m} (-1)^{q+1}  q \cdot
\left( \log \Det_{2 \theta} {\tilde \B(r)^{2}_{q, {\widetilde {\mathcal P}}_{0}}}
+ \log \Det_{2 \theta} {\tilde \B(r)^{2}_{q, {\widetilde {\mathcal P}}_{1}}} -
\log \Det_{2 \theta} {\tilde \B(r)^{2}_{q, \rel}} - \log \Det_{2 \theta} {\tilde \B(r)^{2}_{q, \Abs}} \right)  \\
& = & \sum_{q=0}^{m} (-1)^{q+1}  q \cdot
\left( \log \Det_{2 \theta} \B^{2}_{q, {\widetilde {\mathcal P}}_{0}}(r)
+ \log \Det_{2 \theta} \B^{2}_{q, {\widetilde {\mathcal P}}_{1}}(r)
- \log \Det_{2 \theta} \B^{2}_{q, \rel}(r) - \log \Det_{2 \theta} \B^{2}_{q, \Abs}(r)  \right)  \\
& = & \lim_{r \rightarrow \infty}
\sum_{q=0}^{m} (-1)^{q+1}  q \cdot
\left( \log \Det_{2 \theta} \B^{2}_{q, {\widetilde {\mathcal P}}_{0}}(r)
+ \log \Det_{2 \theta} \B^{2}_{q, {\widetilde {\mathcal P}}_{1}}(r)
- \log \Det_{2 \theta} \B^{2}_{q, \rel}(r) - \log \Det_{2 \theta} \B^{2}_{q, \Abs}(r)  \right)  \\
& = & 0 ,
\end{eqnarray*}

\noindent
which yields the following result. This is the main result of this section and is also interesting independently.
\vspace{0.2 cm}

\begin{theorem}  \label{Theorem:2.11}
Let $(M, g^{M})$ be a compact Riemannian manifold with boundary $Y$ and $g^{M}$ be a product metric near $Y$. Then :

\begin{eqnarray*}
\sum_{q=0}^{m} (-1)^{q+1} q \cdot \left( \log \Det \B^{2}_{q, {\widetilde {\mathcal P}}_{0}} + \log \Det \B^{2}_{q, {\widetilde {\mathcal P}}_{1}} \right)
\hspace{0.1 cm} = \hspace{0.1 cm}
\sum_{q=0}^{m} (-1)^{q+1} q \cdot \left( \log \Det \B^{2}_{q, \rel} + \log \Det \B^{2}_{q, \Abs}  \right) .
\end{eqnarray*}
\end{theorem}

\vspace{0.2 cm}

\noindent
{\it Remark} : This result improves Theorem 2.12 in [12], in which the same result was obtained under the additional assumption of
$H^{q}(M ; E) = H^{q}(M, Y ; E) = \{ 0 \}$ for each $0 \leq q \leq m$.

\vspace{0.2 cm}

For later use we include some result of [12] for eta invariants.
We denote by $\left(\Omega^{\even}(M, E)|_{Y} \right)^{\ast}$ the orthogonal complement of
$\left( \begin{array}{clcr} {\mathcal H}^{\even}(Y, E|_{Y}) \\ {\mathcal H}^{\odd}(Y, E|_{Y}) \end{array} \right)$ in
$\left( \Omega^{\even}(M, E)|_{Y} \right)$, {\it i.e.}

$$
\Omega^{\even}(M, E)|_{Y} \hspace{0.1 cm} = \hspace{0.1 cm}
\left( \Omega^{\even}(M, E)|_{Y} \right)^{\ast} \oplus
\left( \begin{array}{clcr} {\mathcal H}^{\even}(Y, E|_{Y}) \\ {\mathcal H}^{\odd}(Y, E|_{Y}) \end{array} \right),
$$
and denote by ${\mathcal P}_{\ast}$ the orthogonal projection onto $\left(\Omega^{\even}(M, E)|_{Y} \right)^{\ast}$.
We define one parameter families of orthogonal projections ${\widetilde {\frak P}}_{-}(\theta)$,
${\widetilde {\frak P}}_{+}(\theta) : \Omega^{\even}(M, E)|_{Y} \rightarrow \Omega^{\even}(M, E)|_{Y}$ by

\begin{eqnarray*}
{\widetilde {\frak P}}_{-}(\theta) & = & \Pi_{>} \cos \theta + {\mathcal P}_{-} \sin \theta + \frac{1}{2} ( 1 - \cos \theta - \sin \theta ) {\mathcal P}_{\ast} + {\mathcal P}_{{\mathcal L}_{0}} , \\
{\widetilde {\frak P}}_{+}(\theta) & = & \Pi_{>} \cos \theta + {\mathcal P}_{+} \sin \theta + \frac{1}{2} ( 1 - \cos \theta - \sin \theta ) {\mathcal P}_{\ast} + {\mathcal P}_{{\mathcal L}_{1}} ,
\qquad (0 \leq \theta \leq \frac{\pi}{2}),
\end{eqnarray*}

\vspace{0.2 cm}

\noindent
where $\Pi_{>}$ is the orthogonal projection onto the eigenspace
generated by positive eigenforms of ${\mathcal A}$ and ${\mathcal P}_{{\mathcal L}_{i}}$ is the orthogonal projection onto ${\mathcal L}_{i}$ ($i = 1, 2$).
${\widetilde {\frak P}}_{-}(\theta)$ (${\widetilde {\frak P}}_{+}(\theta)$)
is a smooth curve of orthogonal projections connecting ${\mathcal P}_{-, {\mathcal L}_{0}}$ (${\mathcal P}_{+, {\mathcal L}_{1}}$) and
$\Pi_{>, {\mathcal L}_{0}}$ ($\Pi_{>, {\mathcal L}_{1}}$).
We denote the Calder\'on projector for $\B$ by ${\mathcal C}_{M}$.
We also denote the spectral flow for $(\B_{{\widetilde {\frak P}}_{\pm}(\theta)})_{\theta \in [0, \frac{\pi}{2}]}$ and Maslov index for
$( {\widetilde {\frak P}}_{\pm}(\theta), \hspace{0.1 cm} {\mathcal C}_{M})_{\theta \in [0, \frac{\pi}{2}]} $
by $\SF (\B_{{\widetilde {\frak P}}_{\pm}(\theta)})_{\theta \in [0, \frac{\pi}{2}]}$ and
$ \Mas ( {\widetilde {\frak P}}_{\pm}(\theta), \hspace{0.1 cm} {\mathcal C}_{M})_{\theta \in [0, \frac{\pi}{2}]}$.
We refer to [3], [14] and [18] for the definitions of the Calder\'on projector, the spectral flow and Maslov index.
We refer to Theorem 3.12 in [12] for the proof of the following result.

\begin{theorem}   \label{Theorem:2.4}
Let $(M, g^{M})$ be a compact Riemannian manifold with boundary $Y$ and $g^{M}$ be a product metric near $Y$.
Then :

\vspace{0.2 cm}

\noindent
(1) $\hspace{0.2 cm} \eta( \B_{{\mathcal P}_{-, {\mathcal L}_{0}}} ) - \eta ( \B_{\Pi_{>, {\mathcal L}_{0}}} )
\hspace{0.1 cm} = \hspace{0.1 cm} \SF (\B_{{\widetilde {\frak P}}_{-}(\theta)})_{\theta \in [0, \frac{\pi}{2}]}
\hspace{0.1 cm} = \hspace{0.1 cm}  \Mas ( {\widetilde {\frak P}}_{-}(\theta), \hspace{0.1 cm} {\mathcal C}_{M})_{\theta \in [0, \frac{\pi}{2}]} $.  \newline
(2) $\hspace{0.2 cm} \eta( \B_{{\mathcal P}_{+, {\mathcal L}_{1}}} ) - \eta ( \B_{\Pi_{>, {\mathcal L}_{1}}} )
\hspace{0.1 cm} = \hspace{0.1 cm}  \SF (\B_{{\widetilde {\frak P}}_{+}(\theta)})_{\theta \in [0, \frac{\pi}{2}]}
\hspace{0.1 cm} = \hspace{0.1 cm} \Mas ( {\widetilde {\frak P}}_{+}(\theta), \hspace{0.1 cm} {\mathcal C}_{M})_{\theta \in [0, \frac{\pi}{2}]} $. \newline

\noindent
In particular, if for each $0 \leq q \leq m$, $H^{q}(M ; E) = H^{q}(M, Y ; E) = \{ 0 \}$, then
$\eta( \B_{{\mathcal P}_{-}} ) - \eta( \B_{{\mathcal P}_{+}} ) \in {\Bbb Z}$.
\end{theorem}

\vspace{0.2 cm}

In the remaining part of this paper we assume that
$H^{q}(M ; E) = H^{q}(M, Y ; E) = \{ 0 \}$  for each $0 \leq q \leq m$ so that ${\mathcal L}_{0} = {\mathcal L}_{1} = \{ 0 \}$.

\vspace{0.3 cm}

\section{A family of odd signature operators on manifolds with boundary}\label{S:foso}

\vspace{0.2 cm}

In this section we construct a one parameter family of odd signature operators on manifolds with boundary connecting
$\left( \begin{array}{clcr} \B_{\mathcal{P}_{-}} & 0 \\ 0 & -\B_{\mathcal{P}_{+}} \end{array} \right)$
and the odd signature operator considered in [22], (cf. (\ref{E:3.011}) below) by using the ideas in [6] and Section 11 in [5].
For the motivation of this work we review briefly the Vertman's construction of the refined analytic torsion discussed in [22].
We first consider a direct sum of two de Rham complexes with the chirality operator ${\widetilde \Gamma}$, de Rham operator
${\widetilde \nabla}$ and odd signature operator ${\widetilde \B}$ defined as follows.

\vspace{0.2 cm}

\begin{eqnarray}   \label{E:3.a1}
L^{2}\Omega^{\bullet}(M, E) \oplus L^{2}\Omega^{\bullet}(M, E), \qquad
{\widetilde \Gamma} = \left( \begin{array}{clcr} 0 & \Gamma \\ \Gamma & 0 \end{array} \right), \quad
{\widetilde \nabla} = \left( \begin{array}{clcr} \nabla & 0 \\ 0 & \nabla \end{array} \right), \quad
{\widetilde \B}  =  {\widetilde \Gamma} {\widetilde \nabla} + {\widetilde \nabla} {\widetilde \Gamma}.
\end{eqnarray}

\vspace{0.2 cm}

\noindent
We denote by $\nabla_{\Min}$ and $\nabla_{\Max}$ the minimal and maximal closed extensions of $\nabla$ defined on smooth forms
having compact supports in the interior of $M$.
We refer to [22] for definitions of $\nabla_{\Min}$ and $\nabla_{\Max}$.
The following equalities are well known facts (cf. p.1996 in [22]).

%\vspace{0.2 cm}

\begin{eqnarray}   \label{E:3.a2}
\Dom ((\nabla_{\Min})^{\ast}) = \Dom ( \Gamma \nabla_{\Max} \Gamma), \qquad
\Dom ((\nabla_{\Max})^{\ast}) = \Dom ( \Gamma \nabla_{\Min} \Gamma).
\end{eqnarray}

\vspace{0.2 cm}

\noindent
We define

\begin{eqnarray}  \label{E:3.a3}
& \Omega^{q}_{\B, \Min}(M, E) & = \hspace{0.1 cm} \Dom (\nabla_{\Min}) \hspace{0.1 cm} \cap \hspace{0.1 cm} \Dom ((\nabla_{\Min})^{\ast})
\hspace{0.1 cm} \cap L^{2} \hspace{0.1 cm} \Omega^{q}(M, E),  \nonumber \\
& \Omega^{q}_{\B, \Max}(M, E) & = \hspace{0.1 cm} \Dom (\nabla_{\Max}) \hspace{0.1 cm} \cap \hspace{0.1 cm} \Dom ((\nabla_{\Max})^{\ast})
\hspace{0.1 cm} \cap L^{2} \hspace{0.1 cm} \Omega^{q}(M, E),  \nonumber \\
& \Omega^{q}_{\B^{2}, \Min}(M, E) & = \hspace{0.1 cm} \{ \omega \in \Omega^{q}_{\Min}(M, E) \mid \nabla_{\Min} \omega \in \Dom (\nabla_{\Min}^{\ast}), \quad
\nabla_{\Min}^{\ast} \omega \in \Dom (\nabla_{\Min}) \},  \nonumber \\
& \Omega^{q}_{\B^{2}, \Max}(M, E) & = \hspace{0.1 cm} \{ \omega \in \Omega^{q}_{\Max}(M, E) \mid \nabla_{\Max} \omega \in \Dom (\nabla_{\Max}^{\ast}), \quad
\nabla_{\Max}^{\ast} \omega \in \Dom (\nabla_{\Max}) \} ,
\end{eqnarray}

\noindent
and put

\begin{eqnarray}  \label{E:3.011}
{\widetilde \nabla}_{(\m)} & = & \left( \begin{array}{clcr} \nabla_{\Min} & \hspace{0.3 cm} 0 \\ 0 & \nabla_{\Max} \end{array} \right), \nonumber \\
{\widetilde \B}_{(\m)} & = & {\widetilde \Gamma} {\widetilde \nabla}_{(\m)} + {\widetilde \nabla}_{(\m)}  {\widetilde \Gamma}
\hspace{0.1 cm} = \hspace{0.1 cm}
\left( \begin{array}{clcr}  0 & \Gamma \nabla_{\Max} + \nabla_{\Min} \Gamma \\  \Gamma \nabla_{\Min} + \nabla_{\Max} \Gamma & \hspace{1.0 cm} 0 \end{array} \right),
\end{eqnarray}

\vspace{0.2 cm}

\noindent
where the subscript $(\m)$ in ${\widetilde \nabla}_{(\m)}$ and ${\widetilde \B}_{(\m)}$ stands for min/max.
We denote by ${\widetilde \B}_{(\m)}^{\even}$ and ${\widetilde \B}_{(\m)}^{2, q}$ the restriction of ${\widetilde \B}_{(\m)}$ and
$({\widetilde \B}_{(\m)})^{2}$ to even and $q$-forms, respectively.
Then the domains of ${\widetilde \B}_{(\m)}^{\even}$ and ${\widetilde \B}_{(\m)}^{2, q}$  are given as follows.

\begin{eqnarray}  \label{E:3.a4}
\Dom ({\widetilde \B}_{(\m)}^{\even})  =  \left( \begin{array}{clcr}
\Omega^{\even}_{\B, \Min}(M, E) \\ \Omega^{\even}_{\B, \Max}(M, E) \end{array} \right), \quad
\Dom ({\widetilde \B}^{2, q}_{(\m)})  =  \left( \begin{array}{clcr} \Omega^{q}_{\B^{2}, \Min}(M, E) \\
\Omega^{q}_{\B^{2}, \Max}(M, E) \end{array} \right) \hspace{0.1 cm} :=  \hspace{0.1 cm}
{\widetilde \Omega}^{q}_{{\widetilde \B}_{(\m)}^{2}}(M, E \oplus E).
\end{eqnarray}

\vspace{0.2 cm}

\noindent
In [22] Vertman considered the following complex

\begin{eqnarray}  \label{E:3.a5}
0 \longrightarrow
\cdots  \stackrel{\widetilde{\nabla}}{\longrightarrow}  {\widetilde \Omega}^{q-1}_{{\widetilde \B}_{(\m)}^{2}}(M, E \oplus E)
\stackrel{\widetilde{\nabla}}{\longrightarrow}
{\widetilde \Omega}^{q}_{{\widetilde \B}_{(\m)}^{2}}(M, E \oplus E)
\stackrel{\widetilde{\nabla}}{\longrightarrow} \cdots \longrightarrow  0.
\end{eqnarray}

\vspace{0.2 cm}

\noindent
We define ${\widetilde \B}^{\even, \trivial}_{(\m)}$ by the same way as ${\widetilde \B}^{\even}_{(\m)}$
when $\nabla$ is the trivial connection acting on the trivial line bundle $M \times {\Bbb C}$.
For simplicity we assume that $H^{\bullet}(M ; E) = H^{\bullet}(M, Y ; E) = 0$.
In this case the Vertman's construction of the refined analytic torsion $\rho_{\an, (\m)}(g^{M}, {\widetilde \nabla})$ is given as follows.

\begin{eqnarray}    \label{E:3.115}
 \log {\widetilde \rho}_{\an, (\m)}(g^{M}, {\widetilde \nabla})  & = &
\log \Det_{\gr, \theta} \widetilde{\B}^{\even}_{(\m)}  + i \pi \rk(E) \cdot
\eta_{{\widetilde{\B}^{\even, \trivial}_{(\m)}}}(0)  \nonumber \\
& = & \frac{1}{2} \sum_{q=0}^{m} (-1)^{q+1} q \cdot \log \Det_{2 \theta}{\widetilde{\B}}^{2, q}_{(\m)}
+ \hspace{0.1 cm}  \frac{i \pi}{2} \sum_{q=0}^{m} (-1)^{q+1} q \cdot \zeta_{\widetilde{\B}^{2, q}_{(\m)}}(0) \nonumber  \\
& - & i \pi \left( \eta \left({\widetilde{\B}^{\even}_{(\m)}} \right) -
\frac{1}{2} \rk(E) \cdot \eta_{{\widetilde{\B}^{\even, \trivial}_{(\m)}}}(0) \right).
\end{eqnarray}

\vspace{0.2 cm}

We denote by $\Omega^{\bullet}(M, E)$ the space of all smooth $E$-valued forms on $M$. We define

\begin{eqnarray}     \label{E:3.117}
\Omega^{\bullet}_{\rel}(M,E) & := & \{ \hspace{0.1 cm} \omega \in \Omega^{\bullet}(M, E) \mid \hspace{0.1 cm} \mathcal{P}_{\rel}(\omega|_Y)
\hspace{0.1 cm} := \hspace{0.1 cm}  dx \lrcorner \left( dx \wedge (\omega|_Y) \right) \hspace{0.1 cm} = \hspace{0.1 cm} 0 \hspace{0.1 cm},
\hspace{0.1 cm}
\mathcal{P}_{\rel}((\Gamma \nabla \Gamma \omega)|_Y) \hspace{0.1 cm} = \hspace{0.1 cm} 0 \hspace{0.1 cm} \}, \nonumber   \\
\Omega^{\bullet}_{\Abs}(M,E) & := & \{ \hspace{0.1 cm} \omega \in \Omega^{\bullet}(M, E) \mid \hspace{0.1 cm} \mathcal{P}_{\Abs}(\omega|_Y)
\hspace{0.1 cm} := \hspace{0.1 cm} dx \lrcorner (\omega|_Y) \hspace{0.1 cm} = \hspace{0.1 cm} 0 \hspace{0.1 cm},
\hspace{0.1 cm} \mathcal{P}_{\Abs}((\nabla \omega)|_Y) \hspace{0.1 cm} = \hspace{0.1 cm} 0 \hspace{0.1 cm} \},
\end{eqnarray}

\vspace{0.2 cm}

\noindent
and denote by $\B^{2}_{q, \rel}$ and $\B^{2}_{q, \Abs}$ the restriction of $\B^{2}$ to $\Omega^{q}_{\rel}(M,E)$
and $\Omega^{q}_{\Abs}(M,E)$, respectively.
It is a well known facts (cf. Theorem 3.2 in [22], Section 2.7 in [10]) that

\begin{eqnarray}     \label{E:3.118}
& & \Omega^{q}_{\rel}(M,E)  \hspace{0.2 cm}  \subset \hspace{0.2 cm}  \Omega^{q}_{\B^{2}, \Min}(M, E)
\hspace{0.2 cm}  \subset \hspace{0.2 cm}  \Omega^{q}_{\B, \Min}(M, E), \nonumber  \\
& & \Omega^{q}_{\Abs}(M,E) \hspace{0.2 cm}  \subset \hspace{0.2 cm}  \Omega^{q}_{\B^{2}, \Max}(M, E)
\hspace{0.2 cm}  \subset \hspace{0.2 cm}  \Omega^{q}_{\B, \Max}(M, E),   \nonumber  \\
& & \Spec \left( \B^{2}_{q, \rel} \right) \hspace{0.2 cm} = \hspace{0.2 cm} \Spec \left( \B^{2}|_{\Omega^{q}_{\B^{2}, \Min} (M, E)} \right), \quad
\Spec \left( \B^{2}_{q, \Abs} \right)  \hspace{0.1 cm} = \hspace{0.1 cm}
\Spec \left( \B^{2}|_{\Omega^{q}_{\B^{2}, \Max} (M, E)} \right) ,
\end{eqnarray}

\noindent
which leads to

\begin{eqnarray}    \label{E:3.119}
\log \Det_{2 \theta} \left( {\widetilde \B}^{2, q}_{(\m)} \right) & = & \log \Det_{2 \theta} \left( \B^{2, q}_{\rel} \right) +
\log \Det_{2 \theta} \left( \B^{2, q}_{\Abs} \right), \nonumber  \\
\zeta_{{\widetilde \B}^{2, q}_{(\m)}}(0) & = & \zeta_{\B^{2, q}_{\rel}}(0) + \zeta_{\B^{2, q}_{\Abs}}(0) .
\end{eqnarray}

\vspace{0.2 cm}

\noindent
The proof of the following lemma is similar to the proof of Lemma 3.4 in [11].

\begin{lemma}  \label{Lemma:3.0}
Let $(M, g^{M})$ be a compact oriented Riemannian manifold with boundary $Y$ and $g^{M}$ be a product metric near $Y$. We assume that
$H^{q}(M, Y ; E) = H^{q}(M ; E) = 0$ for each $0 \leq q \leq m$. Then,

$$
\zeta_{\B^{2, q}_{\rel}}(0) \hspace{0.1 cm} + \hspace{0.1 cm} \zeta_{\B^{2, q}_{\Abs}}(0)
\hspace{0.1 cm} = \hspace{0.1 cm} \zeta_{\B^{2, q}_{{\mathcal P}_{-, {\mathcal L}_{0}}}}(0) \hspace{0.1 cm} + \hspace{0.1 cm}
\zeta_{\B^{2, q}_{{\mathcal P}_{+, {\mathcal L}_{1}}}}(0) \hspace{0.1 cm} = \hspace{0.1 cm} 0 .
$$
\end{lemma}

\begin{proof} :
We denote by ${\mathcal E}^{q}_{\double}$, ${\mathcal E}^{q}_{\cyl, \rel}$ the heat kernels of $e^{-t \B^{2, q}_{\double}}$
and $e^{-t \B^{2, q}_{\cyl, \rel}}$, where $\B^{2,q}_{\double}$ and $\B^{2, q}_{\cyl, \rel}$
are Laplacians acting on $q$-forms on the closed double $M \cup_{Y} M$
and the half-infinite cylinder $Y \times [0, \infty)$ with the relative boundary condition at $Y \times \{ 0 \}$, respectively.
Let $\rho(a, b)$ be a smooth increasing function of real variable such that

$$
\rho(a, b) (u) = \left\{ \begin{array}{ll} 0 & \mbox{for $u \leq a$} \\
1 & \mbox{for $u \geq b$} \hspace{0.1 cm}.
\end{array} \right.
$$

\noindent
We put
\begin{eqnarray}   \label{E:5.55}
\phi_{1} := 1 - \rho(\frac{5}{7}, \frac{6}{7}), &\quad & \phi_{2} := \rho(\frac{1}{7}, \frac{2}{7})  \nonumber \\
\psi_{1} := 1 - \rho(\frac{3}{7}, \frac{4}{7}), &\quad & \psi_{2} := \rho(\frac{3}{7}, \frac{4}{7}) .
\end{eqnarray}
Then a parametrix $Q(t, (w, x), (w^{\prime}, y))$ of the kernel of $e^{- t \B^{2, q}_{\rel}}$ is given as follows.

\begin{eqnarray*}
Q(t,  (w, x), (w^{\prime}, y)) & = & \phi_{1}(x) {\mathcal E}^{q}_{\cyl, \rel} (t,  (w, x), (w^{\prime}, y)) \psi_{1}(y) +
\phi_{2}(x) {\mathcal E}^{q}_{\double} (t,  (w, x), (w^{\prime}, y)) \psi_{2}(y).
\end{eqnarray*}

\noindent
It is a well known fact that

\begin{eqnarray}
{\mathcal E}^{q}_{\double} (t,  (w, x), (w, x)) & \sim &  \sum_{j=0}^{\infty} a_{m-j}(w, x) t^{- \frac{m-j}{2}} \qquad \text{with} \quad
a_{0}(w, x) = 0 ,     \\
{\mathcal E}^{q}_{\cyl} (t, (w, x), (w, x)) & = & \frac{1}{\sqrt{4 \pi t}} \left( e^{- \frac{(x - y)^{2}}{4 t}} - e^{- \frac{(x + y)^{2}}{4 t}} \right)
e^{- t \B^{2, q}_{Y}}
+ \frac{1}{\sqrt{4 \pi t}} \left( e^{- \frac{(x - y)^{2}}{4 t}} + e^{- \frac{(x + y)^{2}}{4 t}} \right)
e^{- t \B^{2, q-1}_{Y}} . \nonumber
\end{eqnarray}

\noindent
Since $\Tr \left( e^{-t \B^{2, q}_{Y}}  \right) \sim \sum_{j=0}^{\infty} a_{\frac{m-1}{2} - j}(Y, q) \hspace{0.1 cm} t^{- \frac{m-1}{2} + j}$, we have

\begin{eqnarray*}
\zeta_{\B^{2, q}_{\rel}} (0) \hspace{0.1 cm} = \hspace{0.1 cm} \frac{1}{4} \left( - a_{0}(Y, q) + a_{0}(Y, q-1) \right).
\end{eqnarray*}

\noindent
A similar method shows that

\begin{eqnarray*}
\zeta_{\B^{2, q}_{\Abs}} (0) \hspace{0.1 cm} = \hspace{0.1 cm} \frac{1}{4} \left( a_{0}(Y, q) - a_{0}(Y, q-1) \right),
\end{eqnarray*}

\noindent
which completes the proof of the first equality. The second equality is proven in Lemma 3.4 in [11].
\end{proof}

\noindent
{\it Remark} : More generally, one can show that if $g^{M}$ is a product metric near boundary, then
$$
\zeta_{\B^{2, q}_{\rel}}(0) \hspace{0.1 cm} + \hspace{0.1 cm}
\zeta_{\B^{2, q}_{\Abs}}(0) \hspace{0.1 cm} = \hspace{0.1 cm} \zeta_{\B^{2, q}_{{\mathcal P}_{-, {\mathcal L}_{0}}}}(0) \hspace{0.1 cm} + \hspace{0.1 cm}
\zeta_{\B^{2, q}_{{\mathcal P}_{+, {\mathcal L}_{1}}}}(0)
\hspace{0.1 cm} = \hspace{0.1 cm} - \left( \Dim H^{q}(M, Y ; E) + \Dim H^{q}(M ; E) \right).
$$

\vspace{0.2 cm}

We recall that

\begin{eqnarray}  \label{E:4.1313}
{\widetilde \B}^{\even}_{(\m)} & : &  \Omega^{\even}_{\B, \Min}(M, E) \oplus \Omega^{\even}_{\B, \Max}(M, E)  \hspace{0.1 cm}  \rightarrow  \hspace{0.1 cm}
L^{2}\Omega^{\even}(M, E) \oplus L^{2}\Omega^{\even}(M, E)     \\
{\widetilde \B}^{\even}_{(\R/\A)} & := & \left( \begin{array}{clcr} 0 & \B^{\even}_{\Abs} \\ \B^{\even}_{\rel} & \hspace{0.2 cm} 0 \end{array} \right)  :
\Omega^{\even}_{\rel}(M, E) \oplus \Omega^{\even}_{\Abs}(M, E) \hspace{0.1 cm}  \rightarrow  \hspace{0.1 cm}
\Omega^{\even}(M, E) \oplus \Omega^{\even}(M, E),   \nonumber
\end{eqnarray}

\vspace{0.2 cm}

\noindent
where the subscript (r/a) in ${\widetilde \B}^{\even}_{(\R/\A)}$ stands for rel/abs.
By the same reason as in (\ref{E:3.118}), we have

\begin{eqnarray}   \label{E:3.120}
\Spec \left( {\widetilde \B}^{\even}_{(\m)} \right) \hspace{0.1 cm}  =  \hspace{0.1 cm}
\Spec \left( {\widetilde \B}^{\even}_{(\R/\A)}  \right).
\end{eqnarray}

\vspace{0.2 cm}

\noindent
By (\ref{E:3.119}), (\ref{E:3.120}) and Lemma \ref{Lemma:3.0} we can rewrite (\ref{E:3.115}) as follows.

\begin{eqnarray}    \label{E:3.121}
\log {\widetilde \rho}_{\an, (\m)}(g^{M}, {\widetilde \nabla})  & = &
\frac{1}{2} \sum_{q=0}^{m} (-1)^{q+1} q \cdot \left( \log \Det_{2 \theta} {\B}^{2, q}_{\rel} + \log \Det_{2 \theta} {\B}^{2, q}_{\Abs} \right)
\nonumber  \\
& & - \hspace{0.1 cm} i \pi \left( \eta \left( {\widetilde \B}^{\even}_{(\R/\A)} \right)  -
\frac{1}{2} \rk(E) \cdot \eta_{{\widetilde \B}^{\even, \trivial}_{(\R/\A)}}(0) \right) .
\end{eqnarray}

\vspace{0.3 cm}

We next consider another complex, which is similar to (\ref{E:3.a5}).
We put (cf. (\ref{E:2.17}))

\begin{eqnarray*}
{\widetilde \Omega}^{q}_{({\widetilde {\mathcal P}}_{0}/{\widetilde {\mathcal P}}_{1})}(M, E \oplus E) \hspace{0.1 cm} := \hspace{0.1 cm}
\Omega^{q}_{{\widetilde {\mathcal P}}_{0}}(M, E) \oplus \Omega^{q}_{{\widetilde {\mathcal P}}_{1}}(M, E),
\end{eqnarray*}

\noindent
and consider the following complex

\begin{eqnarray}    \label{E:3.116}
\cdots  \stackrel{\widetilde{\nabla}}{\longrightarrow}
{\widetilde \Omega}^{q-1}_{({\widetilde {\mathcal P}}_{0}/{\widetilde {\mathcal P}}_{1})} (M, E \oplus E)  & \stackrel{\widetilde{\nabla}}{\longrightarrow}
{\widetilde \Omega}^{q}_{({\widetilde {\mathcal P}}_{0}/{\widetilde {\mathcal P}}_{1})} (M, E \oplus E)
\stackrel{\widetilde{\nabla}}{\longrightarrow} \cdots
\end{eqnarray}

\vspace{0.2 cm}

\noindent
with the following operators

\begin{eqnarray}  \label{E:4.1616}
{\widetilde \Gamma}_{(\pm)} & = & \left( \begin{array}{clcr} \Gamma & 0 \\ 0 & -\Gamma \end{array} \right),  \quad
{\widetilde \nabla}_{({\widetilde {\mathcal P}}_{0}/{\widetilde {\mathcal P}}_{1})} \hspace{0.1 cm} = \hspace{0.1 cm}
\left( \begin{array}{clcr} \nabla_{{\widetilde {\mathcal P}}_{0}} & 0 \\ 0 & \nabla_{{\widetilde {\mathcal P}}_{1}} \end{array} \right),  \\
{\widetilde \B}^{\even}_{({\widetilde {\mathcal P}}_{0}/{\widetilde {\mathcal P}}_{1})} & := & {\widetilde \Gamma}_{(\pm)}
{\widetilde \nabla}_{({\widetilde {\mathcal P}}_{0}/{\widetilde {\mathcal P}}_{1})}
+ {\widetilde \nabla}_{({\widetilde {\mathcal P}}_{0}/{\widetilde {\mathcal P}}_{1})} {\widetilde \Gamma}_{(\pm)} \hspace{0.1 cm} = \hspace{0.1 cm}
\left( \begin{array}{clcr} \B^{\even}_{{\mathcal P}_{-}} & 0 \\ 0 & - \B^{\even}_{{\mathcal P}_{+}} \end{array} \right),  \quad
{\widetilde \B}^{2, q}_{({\widetilde {\mathcal P}}_{0}/{\widetilde {\mathcal P}}_{1})} :=
\left( \begin{array}{clcr} \B^{2, q}_{{\widetilde {\mathcal P}}_{0}} & 0 \\ 0 & \B^{2, q}_{{\widetilde {\mathcal P}}_{1}} \end{array} \right). \nonumber
\end{eqnarray}

\vspace{0.2 cm}

\noindent
When $H^{\bullet}(M ; E)= H^{\bullet}(M, Y ; E) = 0$,
we define the refined analytic torsion ${\widetilde \rho}_{\an, ({\mathcal P}_{-}/{\mathcal P}_{+})}(g^{M}, {\widetilde \nabla})$
with respect to this complex and the boundary conditions ${\mathcal P}_{-}$, ${\mathcal P}_{+}$ as follows (cf. (\ref{E:1.17})).

\begin{eqnarray}
& & \log {\widetilde \rho}_{\an, ({\mathcal P}_{-}/{\mathcal P}_{+})}(g^{M}, {\widetilde \nabla})  \nonumber \\
& := &  \log \Det_{\gr, \theta} \left( \B^{\even}_{{\mathcal P}_{-}} \right) + \log \Det_{\gr, \theta} \left( - \B^{\even}_{{\mathcal P}_{+}} \right)
+ \frac{\pi i}{2} \rk(E)
\left( \eta_{\B^{\even, \trivial}_{{\mathcal P}_{-}}}(0) - \eta_{\B^{\even, \trivial}_{{\mathcal P}_{+}}}(0) \right)
\nonumber   \\
& = & \frac{1}{2} \sum_{q=0}^{m} (-1)^{q+1} \cdot q \cdot \left( \log \Det_{2 \theta} \B^{2, q}_{{\widetilde {\mathcal P}}_{0}}
 +  \log \Det_{2 \theta} \B^{2, q}_{{\widetilde {\mathcal P}}_{1}} \right)  \nonumber   \\
& & - i \pi \left( \eta(\B^{\even}_{{\mathcal P}_{-}}) - \eta(\B^{\even}_{{\mathcal P}_{+}}) \right)
+ \frac{\pi i}{2} \rk(E) \left( \eta_{\B^{\even, \trivial}_{{\mathcal P}_{-}}}(0) - \eta_{\B^{\even, \trivial}_{{\mathcal P}_{+}}}(0) \right) \nonumber  \\
& = &  \log \rho_{\an, {\mathcal P}_{-}}(g^{M}, \nabla) + \log {\overline{\rho_{\an, {\mathcal P}_{+}}(g^{M}, \nabla)}}  ,
\end{eqnarray}

\vspace{0.2 cm}

\noindent
where ${\overline{\rho_{\an, {\mathcal P}_{+}}(g^{M}, \nabla)}}$ is the complex conjugation of $\rho_{\an, {\mathcal P}_{+}}(g^{M}, \nabla)$.
Theorem \ref{Theorem:2.11} and Lemma \ref{Lemma:3.0} lead to the following result.

\begin{corollary}  \label{Corollary:3.3}
Let $(M, g^{M})$ be a compact oriented Riemannian manifold with boundary $Y$. We assume that $g^{M}$ is a product metric near $Y$ and
$H^{\bullet}(M ; E)= H^{\bullet}(M, Y ; E) = 0$. Then,
\begin{eqnarray*}
& & \log {\widetilde \rho}_{\an, (\m)}(g^{M}, {\widetilde \nabla}) - \log {\widetilde \rho}_{\an, ({\mathcal P}_{-}/{\mathcal P}_{+})}(g^{M}, {\widetilde \nabla}) \hspace{0.1 cm} = \hspace{0.1 cm}
- \hspace{0.1 cm} i \pi \left( \eta \left( {\widetilde \B}^{\even}_{(\R/\A)} \right)  - \left( \eta(\B^{\even}_{{\mathcal P}_{-}})
- \eta(\B^{\even}_{{\mathcal P}_{+}}) \right) \right)  \\
& & \hspace{3.0 cm} + \hspace{0.1 cm} \frac{\pi i}{2} \rk(E) \left( \eta_{{\widetilde \B}^{\even, \trivial}_{(\R/\A)}}(0) -
\left( \eta_{\B^{\even, \trivial}_{{\mathcal P}_{-}}}(0) - \eta_{\B^{\even, \trivial}_{{\mathcal P}_{+}}}(0) \right) \right).
\end{eqnarray*}
\end{corollary}

\vspace{0.3 cm}

The purpose of this paper is to compare ${\widetilde \rho}_{\an, (\m)}(g^{M}, {\widetilde \nabla})$ with
${\widetilde \rho}_{\an, ({\mathcal P}_{-}/{\mathcal P}_{+})}(g^{M}, {\widetilde \nabla}) $.
By Corollary \ref{Corollary:3.3}
it's enough to compare $\eta \left( {\widetilde \B}^{\even}_{(\R/\A)} \right)$ with
$\eta(\B^{\even}_{{\mathcal P}_{-}}) - \eta(\B^{\even}_{{\mathcal P}_{+}})$.
For this purpose we are going to use a deformation of odd signature operators and boundary conditions simultaneously.
We here note that $\eta(\B^{\even}_{{\mathcal P}_{-}}) - \eta(\B^{\even}_{{\mathcal P}_{+}})$
and $\eta_{\B^{\even, \trivial}_{{\mathcal P}_{-}}}(0) - \eta_{\B^{\even, \trivial}_{{\mathcal P}_{+}}}(0)$  are integers by Theorem \ref{Theorem:2.4}.
We are next going to construct a one parameter family of operators connecting ${\widetilde \B}^{\even}_{(\R/\A)}$ and
${\widetilde \B}^{\even}_{({\mathcal P}_{-}/{\mathcal P}_{+})}$.
We begin with the following de Rham complex

$$
\Omega^\bullet(M,E) \oplus \Omega^\bullet(M,E) \hspace{0.1 cm} \equiv \hspace{0.1 cm} \Omega^\bullet(M,E \oplus E).
$$

\noindent
We define the de Rham operator ${\widetilde \nabla}$ and chirality operator $\widetilde{\Gamma}(\theta)$ (cf. 11.8, 11.9 in [5]) by

\begin{equation}    \label{E:3.1}
\widetilde{\nabla} \hspace{0.1 cm} =  \hspace{0.1 cm} \left( \begin{array}{clcr} \nabla & 0 \\ 0 & \nabla \end{array} \right), \quad
\widetilde{\Gamma}(\theta) \hspace{0.1 cm} =  \hspace{0.1 cm}
\left( \begin{array}{clcr} \Gamma \sin \theta & \Gamma \cos \theta \\ \Gamma \cos \theta & -\Gamma \sin \theta \end{array} \right)
\hspace{0.1 cm} = \hspace{0.1 cm}
\Gamma \circ   \left( \begin{array}{clcr} \sin \theta & \cos \theta \\ \cos \theta & -\sin \theta \end{array} \right) ,
\quad \theta \in [0,\frac{\pi}{2}].
\end{equation}

\noindent
For $0 \leq \theta \leq \frac{\pi}{2}$, we define a one parameter family of odd signature operators $\widetilde{\B}(\theta)$ by

\begin{eqnarray}    \label{E:3.2}
& & \widetilde{\B}(\theta) : \hspace{0.1 cm} \Omega^\bullet(M,E \oplus E)  \to \Omega^\bullet(M,E \oplus E),  \nonumber  \\
& & \widetilde{\B}(\theta) \hspace{0.1 cm} := \hspace{0.1 cm}  \widetilde{\Gamma}(\theta)\widetilde{\nabla} + \widetilde{\nabla} \widetilde{\Gamma}(\theta)
\hspace{0.1 cm} = \hspace{0.1 cm} \B \left( \begin{array}{clcr} \sin \theta & \cos \theta \\ \cos \theta & -\sin \theta \end{array} \right).
\end{eqnarray}

\noindent
Then we have

\begin{equation}    \label{E:3.3}
\widetilde{\B}(0) = \left( \begin{array}{clcr} 0 & \B \\ \B & 0 \end{array} \right), \quad
\widetilde{\B}(\frac{\pi}{2}) = \left( \begin{array}{clcr} \B & 0 \\ 0 & - \B \end{array} \right), \quad \text{and} \quad
\widetilde{\B}(\theta)^2 = \left( \begin{array}{clcr} \B^{2} & 0 \\ 0 & \B^{2} \end{array} \right).
\end{equation}

\noindent
On a collar neighborhood $N$ of the boundary, the odd signature operator $\widetilde{\B}(\theta)$ is expressed by

\begin{equation}   \label{E:3.4}
\widetilde{\B}(\theta) \hspace{0.1 cm} = \hspace{0.1 cm} {\widetilde {\mathcal \gamma}}(\theta)(\nabla_{\partial_{x}}+ \widetilde{\mathcal{A}}),
\end{equation}

\noindent
where

\begin{equation}   \label{E:3.5}
{\widetilde {\mathcal \gamma}}(\theta) \hspace{0.1 cm} = \hspace{0.1 cm}
\mathcal{\gamma} \circ \left( \begin{array}{clcr} \sin \theta & \cos \theta \\ \cos \theta & - \sin \theta \end{array} \right), \qquad
\widetilde{\mathcal{A}} \hspace{0.1 cm} = \hspace{0.1 cm}
\left( \begin{array}{clcr} \mathcal{A} & 0  \\ 0 & \mathcal{A} \end{array} \right).
\end{equation}

\noindent
Simple computation shows that

\begin{equation}  \label{E:3.6}
{\widetilde {\mathcal \gamma}}(\theta)^{2} \hspace{0.1 cm} = \hspace{0.1 cm} - \Id, \qquad
{\widetilde {\mathcal \gamma}}(\theta) \hspace{0.1 cm} \widetilde{\mathcal{A}} \hspace{0.1 cm} = \hspace{0.1 cm} - \hspace{0.1 cm} \widetilde{\mathcal{A}} \hspace{0.1 cm} {\widetilde {\mathcal \gamma}}(\theta).
\end{equation}

\vspace{0.2 cm}

\noindent
We denote
the $(\pm i)$-eigenspaces of ${\widetilde {\mathcal \gamma}}(\theta)$ in
$\left( \Omega^{\even}(M, E \oplus E)|_{Y} \right)$
by $\left( \Omega^{\even}(M, E \oplus E)|_{Y} \right)_{\theta, \pm i}$, {\it i.e.}

\begin{equation}  \label{E:3.0009}
\left( \Omega^{\even}(M, E \oplus E)|_{Y} \right)_{\theta, \pm i} \hspace{0.1 cm} := \hspace{0.1 cm} \frac{I \mp i {\widetilde {\mathcal \gamma}}(\theta)}{2}
\left( \Omega^{\even}(M, E \oplus E)|_{Y} \right).
\end{equation}

\noindent
Then we have

\begin{equation}  \label{E:3.009}
\Omega^{\even}(M, E \oplus E)|_{Y} \hspace{0.1 cm} = \hspace{0.1 cm} \left( \Omega^{\even}(M, E \oplus E)|_{Y} \right)_{\theta, i} \hspace{0.1 cm} \oplus
\hspace{0.1 cm} \left( \Omega^{\even}(M, E \oplus E)|_{Y} \right)_{\theta, -i}.
\end{equation}

\vspace{0.2 cm}

\noindent
For each $\theta$,
$\left( \Omega^{\even}(M, E \oplus E)|_{Y}, {\widetilde {\mathcal \gamma}}(\theta), \langle \hspace{0.1 cm}, \hspace{0.1 cm} \rangle \right)$
is a symplectic vector space and each Lagrangian subspace is expressed by the graph of a unitary operator from
$\left( \Omega^{\even}(M, E \oplus E)|_{Y} \right)_{\theta, +i}$ to $\left( \Omega^{\even}(M, E \oplus E)|_{Y} \right)_{\theta, -i}$.
$\Imm {\mathcal P}_{\rel} \oplus \Imm {\mathcal P}_{\Abs}$ (cf. (\ref{E:3.117}))
 and $\Imm {\mathcal P}_{-} \oplus \Imm {\mathcal P}_{+}$ are Lagrangian
subspaces of
$\left( \Omega^{\even}(M, E \oplus E)|_{Y}, {\widetilde {\mathcal \gamma}}(\theta), \langle \hspace{0.1 cm}, \hspace{0.1 cm} \rangle \right)$
for $\theta = 0$ and $\frac{\pi}{2}$, respectively.

Using the decomposition (\ref{E:1.3}),
we define two maps $U_{{\mathcal P}_{-} \oplus {\mathcal P}_{+}}$ and $U_{{\mathcal P}_{\rel} \oplus {\mathcal P}_{\Abs}}$
as follows (cf. \ref{E:1.9}).

$$
U_{{\mathcal P}_{-} \oplus {\mathcal P}_{+}}, \quad U_{{\mathcal P}_{\rel} \oplus {\mathcal P}_{\Abs}} : \hspace{0.1 cm}
\Omega^{\even}(M, E \oplus E)|_{Y}  \rightarrow  \Omega^{\even}(M, E \oplus E)|_{Y}
$$

\begin{eqnarray}  \label{E:3.9}
U_{{\mathcal P}_{-} \oplus {\mathcal P}_{+}} & = &
(\B_{Y}^{2})^{- 1} \left( (\B_{Y}^{2})^{-} - (\B_{Y}^{2})^{+} \right) \left( \begin{array}{clcr}  1 & 0 & 0 & 0 \\ 0 & 1 & 0 & 0 \\ 0 & 0 & -1 & 0 \\0 & 0 & 0 &-1 \end{array} \right) ,  \nonumber  \\
U_{{\mathcal P}_{\rel} \oplus {\mathcal P}_{\Abs}} & = & i \hspace{0.1 cm} (- i \beta \Gamma^{Y}) \hspace{0.1 cm}
\left( \begin{array}{clcr} 0 & 0 & -1 & 0 \\ 0 & 0 & 0 & 1 \\ 1 & 0 & 0 & 0 \\ 0 & -1 & 0 & 0 \end{array} \right) ,
\end{eqnarray}

\vspace{0.2 cm}

\noindent
where $\hspace{0.1 cm} \left( \B_{Y}^{2} \right)^{-} := \nabla^{Y} \Gamma^{Y} \nabla^{Y} \Gamma^{Y} :
\Omega^{\bullet}_{-} (Y, E|_{Y}) \rightarrow \Omega^{\bullet}_{-} (Y, E|_{Y}) \hspace{0.1 cm}$ and
$\hspace{0.1 cm} \left( \B_{Y}^{2} \right)^{+} :=  \Gamma^{Y} \nabla^{Y} \Gamma^{Y} \nabla^{Y} :
\Omega^{\bullet}_{+} (Y, E|_{Y}) \rightarrow \Omega^{\bullet}_{+} (Y, E|_{Y}) \hspace{0.1 cm}$.
The following lemma is straightforward.

\vspace{0.2 cm}

\begin{lemma}  \label{Lemma:3.1}
(1) $\hspace{0.1 cm} U_{{\mathcal P}_{-} \oplus {\mathcal P}_{+}} \hspace{0.1 cm} $ and
$\hspace{0.1 cm} U_{{\mathcal P}_{\rel} \oplus {\mathcal P}_{\Abs}} \hspace{0.1 cm}$ are unitary operators
with $\hspace{0.1 cm} \left( U_{{\mathcal P}_{-} \oplus {\mathcal P}_{+}} \right)^{\ast} = U_{{\mathcal P}_{-} \oplus {\mathcal P}_{+}} \hspace{0.1 cm}$
and $\hspace{0.1 cm} \left( U_{{\mathcal P}_{\rel} \oplus {\mathcal P}_{\Abs}} \right)^{\ast} =
- \hspace{0.1 cm} U_{{\mathcal P}_{\rel} \oplus {\mathcal P}_{\Abs}}$.  \newline
(2) $\hspace{0.1 cm} \left( U_{{\mathcal P}_{\rel} \oplus {\mathcal P}_{\Abs}} \right) {\widetilde {\mathcal \gamma}}(0)
\hspace{0.1 cm} = \hspace{0.1 cm} - \hspace{0.1 cm} {\widetilde {\mathcal \gamma}}(0)
\left( U_{{\mathcal P}_{\rel} \oplus {\mathcal P}_{\Abs}} \right) \hspace{0.1 cm}$ and
$\hspace{0.1 cm} \left( U_{{\mathcal P}_{-} \oplus {\mathcal P}_{+}} \right) {\widetilde {\mathcal \gamma}}(\frac{\pi}{2})
\hspace{0.1 cm} = \hspace{0.1 cm} - \hspace{0.1 cm} {\widetilde {\mathcal \gamma}}(\frac{\pi}{2})
\left( U_{{\mathcal P}_{-} \oplus {\mathcal P}_{+}} \right) \hspace{0.1 cm}$. \newline
(3)
$\Imm {\mathcal P}_{\rel} \oplus \Imm {\mathcal P}_{\Abs}$ is the graph of
$\hspace{0.1 cm} U_{{\mathcal P}_{\rel} \oplus {\mathcal P}_{\Abs}} : \left( \Omega^{\even}(M, E \oplus E)|_{Y} \right)_{0, +i} \rightarrow
\left( \Omega^{\even}(M, E \oplus E)|_{Y} \right)_{0, -i}$ and
$\hspace{0.1 cm} \Imm {\mathcal P}_{-} \oplus \Imm {\mathcal P}_{+}$ is the graph of
$\hspace{0.1 cm} U_{{\mathcal P}_{-} \oplus {\mathcal P}_{+}} : \left( \Omega^{\even}(M, E \oplus E)|_{Y} \right)_{\frac{\pi}{2}, +i}
\rightarrow \left( \Omega^{\even}(M, E \oplus E)|_{Y} \right)_{\frac{\pi}{2}, -i}$.
\end{lemma}

\vspace{0.2 cm}

We next define $\hspace{0.1 cm} P(\theta) :
\hspace{0.1 cm} \Omega^{\even}(M, E \oplus E)|_{Y} \hspace{0.1 cm} \rightarrow \hspace{0.1 cm} \Omega^{\even}(M, E \oplus E)|_{Y}  \hspace{0.1 cm}$,
($0 \leq \theta \leq \frac{\pi}{2} \hspace{0.1 cm}$) by

\begin{eqnarray}  \label{E:3.9}
P(\theta) & = & (\B_{Y}^{2})^{- 1} \left( (\B_{Y}^{2})^{-} - (\B_{Y}^{2})^{+} \right) \left( \begin{array}{clcr}
\sin \theta \Id & \cos \theta \Id \\ \cos \theta \Id & - \sin \theta \Id \end{array} \right) \sin \theta +
U_{{\mathcal P}_{\rel} \oplus {\mathcal P}_{\Abs}} \cos \theta,  \nonumber  \\
& = & {\frak A}(\theta) \sin \theta + {\frak B} \cos \theta,
\end{eqnarray}

\noindent
where
\begin{equation}  \label{E:3.99}
{\frak A}(\theta) \hspace{0.1 cm} := \hspace{0.1 cm} (\B_{Y}^{2})^{- 1} \left( (\B_{Y}^{2})^{-} - (\B_{Y}^{2})^{+} \right) \left( \begin{array}{clcr}
\sin \theta \Id & \cos \theta \Id \\ \cos \theta \Id & - \sin \theta \Id \end{array} \right), \quad
{\frak B} \hspace{0.1 cm} := \hspace{0.1 cm}  U_{{\mathcal P}_{\rel} \oplus {\mathcal P}_{\Abs}} .
\end{equation}

\noindent
Then $P(\theta)$ is a smooth path connecting $U_{{\mathcal P}_{\rel} \oplus {\mathcal P}_{\Abs}}$ and $U_{{\mathcal P}_{-} \oplus {\mathcal P}_{+}}$.
The following lemma is straightforward.

\vspace{0.2 cm}

\begin{lemma}  \label{Lemma:3.2}
(1) ${\frak A}(\theta)$ and ${\frak B}$ are unitary operators satisfying
${\frak A}(\theta)^{2} = \Id$, ${\frak B}^{2} = - \Id$, ${\frak A}(\theta)^{\ast} = {\frak A}(\theta)$, and ${\frak B}^{\ast} = - {\frak B}$.
\newline
(2) ${\frak A}(\theta) \hspace{0.1 cm} {\widetilde {\mathcal \gamma}}(\theta) \hspace{0.1 cm} =
\hspace{0.1 cm} - \hspace{0.1 cm}  {\widetilde {\mathcal \gamma}}(\theta) \hspace{0.1 cm} {\frak A}(\theta)$,
${\frak B} \hspace{0.1 cm} {\widetilde {\mathcal \gamma}}(\theta) \hspace{0.1 cm} =
\hspace{0.1 cm} - \hspace{0.1 cm}  {\widetilde {\mathcal \gamma}}(\theta) \hspace{0.1 cm} {\frak B}$ and
${\frak A}^{\prime}(\theta) \hspace{0.1 cm} {\widetilde {\mathcal \gamma}}(\theta) \hspace{0.1 cm} =
\hspace{0.1 cm} {\widetilde {\mathcal \gamma}}(\theta) \hspace{0.1 cm} {\frak A}^{\prime}(\theta)$.  \newline
(3) ${\frak A}^{\prime}(\theta) {\frak A}(\theta) = - {\frak A}(\theta) {\frak A}^{\prime}(\theta) =
\left( \begin{array}{clcr} 0 & \Id \\ - \Id & 0 \end{array} \right)$.  \newline
(4) ${\frak A}(\theta) \hspace{0.1 cm} {\frak B} \hspace{0.1 cm} =
\hspace{0.1 cm} {\frak B} \hspace{0.1 cm} {\frak A}(\theta)$ and ${\frak A}^{\prime}(\theta) \hspace{0.1 cm} {\frak B} \hspace{0.1 cm} =
\hspace{0.1 cm} {\frak B} \hspace{0.1 cm} {\frak A}^{\prime}(\theta)$.  \newline
(5) $P(\theta)$ is a unitary operator with $P(\theta)^{\ast} = {\frak A}(\theta) \sin \theta - {\frak B} \cos \theta$ and
$P(\theta) \hspace{0.1 cm} {\widetilde {\mathcal \gamma}}(\theta) \hspace{0.1 cm} =
\hspace{0.1 cm} - \hspace{0.1 cm}  {\widetilde {\mathcal \gamma}}(\theta) \hspace{0.1 cm} P(\theta)$.  \newline
(6) ${\frak A}(\theta) \hspace{0.1 cm} {\widetilde {\mathcal A}} \hspace{0.1 cm} = \hspace{0.1 cm}
- \hspace{0.1 cm} {\widetilde {\mathcal A}} \hspace{0.1 cm} {\frak A}(\theta) \hspace{0.1 cm}$ and
$\hspace{0.1 cm} {\frak B} \hspace{0.1 cm} {\widetilde {\mathcal A}} \hspace{0.1 cm} =
\hspace{0.1 cm} {\widetilde {\mathcal A}} \hspace{0.1 cm} {\frak B} \hspace{0.1 cm}$ and hence
$\hspace{0.1 cm} P(\theta)^{\ast} \hspace{0.1 cm} {\widetilde {\mathcal A}} \hspace{0.1 cm} = \hspace{0.1 cm} -
\hspace{0.1 cm} {\widetilde {\mathcal A}} \hspace{0.1 cm} P(\theta)$.
\end{lemma}

\vspace{0.2 cm}

\noindent
We note that the orthogonal projections ${\mathcal P}_{\rel} \oplus {\mathcal P}_{\Abs}$ and ${\mathcal P}_{-} \oplus {\mathcal P}_{+}$
are described as follows.

\begin{eqnarray*}
{\mathcal P}_{\rel} \oplus {\mathcal P}_{\Abs}, \hspace{0.2 cm} {\mathcal P}_{-} \oplus {\mathcal P}_{+} :
\oplus_{k=0}^{1} \left( \Omega^{\even}(M, E \oplus E)|_{Y} \right)_{\theta, (-1)^{k}i}  \hspace{0.1 cm} \rightarrow
\hspace{0.1 cm}   \oplus_{k=0}^{1} \left( \Omega^{\even}(M, E \oplus E)|_{Y} \right)_{\theta, (-1)^{k}i}
\end{eqnarray*}

\begin{eqnarray}  \label{E:3.1415}
{\mathcal P}_{\rel} \oplus {\mathcal P}_{\Abs} \hspace{0.1 cm} =
\hspace{0.1 cm} \frac{1}{2} \left[ \begin{array}{clcr} \Id & {\frak B}^{\ast} \\ {\frak B} & \Id \end{array} \right] , \qquad
{\mathcal P}_{-} \oplus {\mathcal P}_{+} \hspace{0.1 cm} =
\hspace{0.1 cm} \frac{1}{2} \left[ \begin{array}{clcr} \Id & {\frak A}(\frac{\pi}{2})^{\ast} \\ {\frak A}(\frac{\pi}{2}) & \Id \end{array} \right] .
\end{eqnarray}

\noindent
We define a smooth path $\hspace{0.1 cm} {\widetilde P}(\theta) \hspace{0.1 cm}$
of orthogonal projections connecting $\hspace{0.1 cm} {\mathcal P}_{\rel} \oplus {\mathcal P}_{\Abs} \hspace{0.1 cm} $ and $\hspace{0.1 cm} {\mathcal P}_{-} \oplus {\mathcal P}_{+} \hspace{0.1 cm} $ by

\begin{eqnarray}  \label{E:3.10}
{\widetilde P}(\theta) \hspace{0.1 cm} =
\hspace{0.1 cm} \frac{1}{2} \left[ \begin{array}{clcr} \Id & P(\theta)^{\ast} \\ P(\theta) & \Id \end{array} \right] , \qquad (0 \leq \theta \leq \frac{\pi}{2}).
\end{eqnarray}

\noindent
Under the decomposition (\ref{E:3.009}) we have

\begin{eqnarray} \label{E:4.3232}
\hspace{0.1 cm} \widetilde{\mathcal \gamma}(\theta)= \left[ \begin{array}{clcr} i & 0 \\ 0 & -i  \end{array} \right],   \qquad
\hspace{0.1 cm} \widetilde{\B}_{Y}^{2} := \widetilde{\mathcal A}^{2}  = \B_{Y}^{2} \left[\begin{array}{clcr} \Id & 0 \\ 0 & \Id  \end{array} \right] ,   \qquad
\hspace{0.1 cm} \widetilde{\mathcal A} =
\left[ \begin{array}{clcr} 0 & \widetilde{\mathcal A}_{-} \\ \widetilde{\mathcal A}_{+} & 0 \end{array} \right],
\end{eqnarray}

\noindent
where $\hspace{0.1 cm} \widetilde{\mathcal A}_{\pm} = \widetilde{\mathcal A}|_{\left(\Omega^{\even}(M, E \oplus E)|_{Y} \right)_{\theta, \pm i}}$.
Then ${\widetilde P}(\theta)$ satisfies the following properties, whose proofs are straightforward.

\vspace{0.2 cm}

\begin{lemma} \label{Lemma:3.3}
(1) $\widetilde{\mathcal \gamma}(\theta) \hspace{0.1 cm} {\widetilde P}(\theta) \hspace{0.1 cm} = \hspace{0.1 cm} (I - {\widetilde P}(\theta))
\hspace{0.1 cm} \widetilde{\mathcal \gamma}(\theta), \hspace{0.2 cm}$ and
 $\hspace{0.2 cm} {\widetilde P}(\theta) \hspace{0.1 cm} \widetilde{\B}_{Y}^{2} \hspace{0.1 cm} = \hspace{0.1 cm} \widetilde{\B}_{Y}^{2} {\widetilde P}(\theta)$.       \newline
(2) ${\widetilde P}(\theta) \hspace{0.1 cm} \widetilde{\mathcal A} \hspace{0.1 cm} {\widetilde P}(\theta)
\hspace{0.1 cm} = 0, \hspace{0.2 cm}$ \text{and} $\hspace{0.2 cm} (I -{\widetilde P}(\theta)) \hspace{0.1 cm} \widetilde{\mathcal A} \hspace{0.1 cm}
(I - {\widetilde P}(\theta))  \hspace{0.1 cm} = 0.$   \newline
\end{lemma}

\noindent
{\it Remark} : In this paper we are using two types of decompositions. One comes from the decomposition (\ref{E:1.3}) and
the other one comes from the decomposition (\ref{E:3.009}). When we write a matrix form of an operator, we are going to use the notation
$\left( \hspace{0.2 cm} \right)$ for the decomposition (\ref{E:1.3}) like (\ref{E:3.3}) and use $\left[ \hspace{0.2 cm} \right]$ for the decomposition (\ref{E:3.009}) like (\ref{E:3.1415}).

\vspace{0.3 cm}

\noindent
Let $\widetilde{\B}(\theta)_{{\widetilde P}(\theta)}$ $(0 \leq \theta \leq \frac{\pi}{2}) \hspace{0.1 cm}$ be the realization of $\widetilde{\B}(\theta)$ with respect to ${\widetilde P}(\theta)$, {\it i.e.}
$$
\Dom \left( \widetilde{\B}(\theta)_{{\widetilde P}(\theta)} \right) \hspace{0.1 cm} = \{ \phi \in H^{1} \left( \Omega^{\even}(M, E \oplus E) \right) \mid
{\widetilde P}(\theta) (\phi|_{Y}) = 0 \}.
$$
Then $\widetilde{\B}(\theta)^{\even}_{{\widetilde P}(\theta)} \hspace{0.1 cm}$  is a smooth path of operators connecting
$\hspace{0.1 cm} {\widetilde \B}^{\even}_{(\R/\A)} \hspace{0.1 cm}$ and
$ \hspace{0.1 cm}{\widetilde \B}^{\even}_{({\widetilde {\mathcal P}_{-}}/{\widetilde {\mathcal P}_{+}})} \hspace{0.1 cm}$
(cf. (\ref{E:4.1313}), (\ref{E:4.1616})).

\vspace{0.2 cm}

\begin{lemma} \label{Lemma:3.4}
$\widetilde{\B}(\theta)_{{\widetilde P}(\theta)}$ is essentially self-adjoint.
\end{lemma}

\begin{proof}
The Green formula for $\widetilde{\B}(\theta)$ can be written as follows (cf. (3) in Lemma \ref{Lemma:1.1}).
For ${\widetilde \phi} = \left( \begin{array}{clcr} \phi_{1} \\ \phi_{2}  \end{array} \right)$,
${\widetilde \psi} = \left( \begin{array}{clcr} \psi_{1} \\ \psi_{2}  \end{array} \right) \in \Omega^{\even}(M, E \oplus E)$,
$$
\langle \widetilde{\B}(\theta) {\widetilde \phi}, \hspace{0.1 cm} {\widetilde \psi} \rangle_{M} \hspace{0.1 cm} = \hspace{0.1 cm}
\langle {\widetilde \phi}, \hspace{0.1 cm} \widetilde{\B}(\theta) {\widetilde \psi} \rangle_{M} \hspace{0.1 cm} + \hspace{0.1 cm}
\langle {\widetilde \phi}|_{Y}, \hspace{0.1 cm} {\widetilde {\mathcal \gamma}}(\theta) ({\widetilde \psi}|_{Y}) \rangle_{Y}.
$$
The remaining part is a verbatim repetition of the proof of Lemma 3.3 in [12].

\end{proof}

\vspace{0.2 cm}

We next define a unitary operator $U(\theta)$ ($0 \leq \theta \leq \frac{\pi}{2}$) by

$$
U(\theta) : \hspace{0.1 cm} \oplus_{k=0}^{1} \left( \Omega^{\even}(M, E \oplus E)|_{Y} \right)_{\theta, (-1)^{k}i} \hspace{0.1 cm}
\rightarrow   \hspace{0.1 cm}  \oplus_{k=0}^{1} \left( \Omega^{\even}(M, E \oplus E)|_{Y} \right)_{\theta, (-1)^{k}i}
$$

\begin{eqnarray} \label{E:3.11}
U(\theta) \hspace{0.1 cm} = \hspace{0.1 cm}
\left[ \begin{array}{clcr}  \cos \theta - {\frak B}^{\ast} \hspace{0.1 cm} {\frak A}(\theta) \sin \theta & 0 \\ 0 & \Id  \end{array} \right] .
\end{eqnarray}

\noindent
Then $U(\theta)$ satisfies the following equality.

\begin{eqnarray} \label{E:3.12}
U(\theta) \hspace{0.1 cm} {\widetilde P}(0) \hspace{0.1 cm} U(\theta)^{\ast} \hspace{0.1 cm} = \hspace{0.1 cm} {\widetilde P}(\theta).
\end{eqnarray}

\noindent
Setting

\begin{eqnarray} \label{E:3.13}
T(\theta) \hspace{0.1 cm} = \hspace{0.1 cm} - i \theta
\left[ \begin{array}{clcr} - {\frak B}^{\ast} \hspace{0.1 cm} {\frak A}(\theta) & 0 \\ 0 & 0  \end{array} \right]
\hspace{0.1 cm} = \hspace{0.1 cm}  i \theta \hspace{0.1 cm}
 {\frak B}^{\ast} \hspace{0.1 cm} {\frak A}(\theta) \hspace{0.1 cm} \frac{ I - i {\widetilde {\mathcal \gamma}}(\theta)}{2} ,
\end{eqnarray}

\noindent
Lemma \ref{Lemma:3.2} shows that

\begin{eqnarray} \label{E:3.133}
e^{i T(\theta)} \hspace{0.1 cm} = \hspace{0.1 cm} U(\theta).
\end{eqnarray}

\noindent
$T(\theta)$ satisfies the following properties.

\begin{lemma} \label{Lemma:3.5}
 $T(\theta) \hspace{0.1 cm} {\widetilde {\mathcal \gamma}}(\theta) \hspace{0.1 cm} = \hspace{0.1 cm}
\widetilde{\mathcal \gamma}(\theta) \hspace{0.1 cm} T(\theta) \hspace{0.1 cm}$ and
$\hspace{0.1 cm} T(\theta) \hspace{0.1 cm} {\widetilde \B_{Y}^{2}} \hspace{0.1 cm} = \hspace{0.1 cm}
{\widetilde \B_{Y}^{2}} \hspace{0.1 cm} T(\theta)$.
\end{lemma}

\vspace{0.2 cm}

\noindent
{\it Remark} : Contrary to the case of [6], $T(\theta)$ does not anticommute with ${\widetilde {\mathcal A}}$.

\vspace{0.2 cm}

Let $\phi : [0, 1] \rightarrow [0, 1]$ be a decreasing smooth function such that $\phi = 1$ on a small neighborhood of $0$ and
$\phi = 0$ on a small neighborhood of $1$. We use this cut-off function to extend $T(\theta)$ defined on $\Omega^{\bullet}(M, E \oplus E)|_{Y}$
to a unitary operator defined on $\Omega^{\bullet}(M, E \oplus E)$.
We define
$\Psi_{\theta} : \Omega^{\bullet}(M, E \oplus E) \rightarrow \Omega^{\bullet}(M, E \oplus E)$ by

\begin{equation} \label{E:3.14}
\Psi_{\theta} (\omega) (x) = e^{i \phi(x) T(\theta)} \omega (x),
\end{equation}

\noindent
where the support of $\phi(x) T(\theta)$ is contained in $N := [0, 1) \times Y$, a collar neighborhood of $Y$.

\vspace{0.2 cm}

\begin{lemma} \label{Lemma:3.6}
$\Psi_{\theta}$ is a unitary operator mapping from $\Dom \left(\widetilde{\B}(\theta)_{{\widetilde P}(0)} \right)$ onto $\Dom \left(\widetilde{\B}(\theta)_{{\widetilde P}(\theta)} \right)$.
\end{lemma}

\begin{proof} Clearly $\Psi_{\theta}$ is a unitary operator.
Let ${\widetilde P}(0) \omega(0) = 0$. Then
\begin{eqnarray*}
{\widetilde P}(\theta) (\Psi_{\theta} \omega) (0) & = & U(\theta) {\widetilde P}(0) U(\theta)^{\ast} \left( e^{i \phi(x) T(\theta)} \omega \right)|_{x=0}  \\
& = & U(\theta) {\widetilde P}(0) e^{- i T(\theta)} \left( e^{i \phi(x) T(\theta)} \omega \right)|_{x=0} =
U(\theta) {\widetilde P}(0) \omega(0) = 0,
\end{eqnarray*}
which completes the proof of the lemma.
\end{proof}

\vspace{0.2 cm}

\noindent
Note that
$\Dom \left(\widetilde{\B}(0)_{{\widetilde P}(0)} \right)\equiv \Dom \left(\widetilde{\B}(\theta)_{{\widetilde P}(0)} \right)
\subset \Omega^{\even}(M, E \oplus E) \hspace{0.1 cm}$ and consider the following diagram.

\vspace{0.3 cm}

$$\begin{CD}
  \Dom \left( \widetilde{\B}(0)_{{\widetilde P}(0)} \right)     & @> \widehat{\B}(\theta) >> &  \Omega^{\even}(M, E \oplus E)       \\
  @V \Psi_{\theta} VV      &        & @V{\Psi_{\theta}}VV    \\
  \Dom \left( \widetilde{\B}(\theta)_{{\widetilde P}(\theta)} \right)     & @> \widetilde{\B}(\theta)  >> & \Omega^{\even}(M, E \oplus E)
\end{CD}
$$

\vspace{0.3 cm}

\noindent
Setting $\widehat{\B}(\theta) := \Psi_{\theta}^{\ast} \hspace{0.1 cm} \widetilde{\B}(\theta) \hspace{0.1 cm}
\Psi_{\theta}|_{\Dom \left(\widetilde{\B}(0)_{{\widetilde P}(0)} \right)}$,
$$
\widehat{\B}(\theta) : \Dom \left(\widetilde{\B}(0)_{{\widetilde P}(0)} \right) \rightarrow \Omega^{\even}(M, E \oplus E)
$$
is an elliptic $\Psi$DO of order $1$
with a fixed domain $\Dom \left(\widetilde{\B}(0)_{{\widetilde P}(0)} \right)$ and has the same spectrum as
$\widetilde{\B}(\theta)_{{\widetilde P}(\theta)}$,
which is a smooth path of operators
connecting $\hspace{0.1 cm} {\widetilde \B}^{\even}_{(\R/\A)} \hspace{0.1 cm}$ at $\theta = 0$ and
$ \hspace{0.1 cm}{\widetilde \B}^{\even}_{({\widetilde {\mathcal P}_{-}}/{\widetilde {\mathcal P}_{+}})} \hspace{0.1 cm}$
at $\theta = \frac{\pi}{2}$.

\vspace{0.3 cm}

\section{Comparison of the eta invariants}

\vspace{0.2 cm}

In this section we discuss the variation of eta functions for $\widehat{\B}(\theta)$ to compare $\eta \left( {\widetilde \B}^{\even}_{(\R/\A)} \right)$
with $\eta \left( {\widetilde \B}^{\even}_{({\widetilde {\mathcal P}_{-}}/{\widetilde {\mathcal P}_{+}})} \right)$.
For this purpose we are going to use the deformation method in [6].
In [12] we used similar method to compare the eta invariants subject to the boundary conditions ${\mathcal P}_{-}$ and ${\mathcal P}_{+}$ with the eta invariants subject to the APS boundary condition.
We now begin with the one parameter family of the eta functions $\eta_{\widehat{\B}(\theta)}(s)$ defined by

\vspace{0.2 cm}

\begin{equation} \label{E:4.1}
\eta_{\widehat{\B}(\theta)}(s) = \frac{1}{\Gamma(\frac{s+1}{2})} \int_{0}^{\infty} t^{\frac{s-1}{2}} \Tr \left( \widehat{\B}(\theta) e^{-t \widehat{\B}(\theta)^{2}} \right) dt.
\end{equation}

\vspace{0.2 cm}

\noindent
If $\eta_{\widehat{\B}(\theta)}(s)$ has a regular value at $s=0$, we define the eta invariant $\eta (\widehat{\B}(\theta))$ by

\begin{equation}  \label{E:4.2}
\eta (\widehat{\B}(\theta)) = \frac{1}{2} \left( \eta_{\widehat{\B}(\theta)}(0) + \Dim \Ker \widehat{\B}(\theta) \right).
\end{equation}

\vspace{0.2 cm}

\noindent
For a fixed $\theta_{0}$ in $[0, \frac{\pi}{2}]$, there exist $c(\theta_{0}) > 0$ and $\delta > 0$ such that
$c(\theta_{0}) \notin \Spec \left( \widehat{\B}(\theta) \right)$ for $\theta_{0} - \delta < \theta < \theta_{0} + \delta$.
We denote by $Q(\theta)$ the orthogonal projection onto the space spanned by eigensections of $\widehat{\B}(\theta)$ whose eigenvalues belong to
$\hspace{0.1 cm} (-c(\theta_{0}), \hspace{0.1 cm} c(\theta_{0})) \hspace{0.1 cm}$ for $\theta_{0} - \delta < \theta < \theta_{0} + \delta$.
We define

\begin{eqnarray*}
\eta_{\widehat{\B}(\theta)} \left( s \hspace{0.1 cm} ; \hspace{0.1 cm} c(\theta_{0}) \right) \hspace{0.1 cm} =
\sum_{| \lambda | > c(\theta_{0})} sign (\lambda) \vert \lambda \vert^{-s}
\hspace{0.1 cm} = \hspace{0.1 cm} \frac{1}{\Gamma(\frac{s+1}{2})} \int^{\infty}_{0} t^{\frac{s-1}{2}}
\Tr \left\{ \left( I - Q(\theta) \right) \widehat{\B}(\theta) e^{- t \widehat{\B}(\theta)^{2}} \right\} dt.
\end{eqnarray*}

\noindent
Then $\eta_{\widehat{\B}(\theta)} (s) - \eta_{\widehat{\B}(\theta)} \left( s \hspace{0.1 cm} ; c(\theta_{0}) \right)$ is an entire function and

\vspace{0.2 cm}

\begin{equation}   \label{E:4.3}
\left\{ \frac{1}{2} \left( \eta_{\widehat{\B}(\theta)} (s) + \Dim \Ker \widehat{\B}(\theta) \right) - \frac{1}{2} \eta_{\widehat{\B}(\theta)} \left( s \hspace{0.1 cm} ; c(\theta_{0}) \right) \right\}_{s=0}
\end{equation}

\vspace{0.2 cm}

\noindent
does not depend on the $\theta$ for $\theta_{0} - \delta < \theta < \theta_{0} + \delta$ up to $\Mod {\Bbb Z}$.
Simple computation shows that

\begin{eqnarray}  \label{E:4.4}
& & \frac{d}{d \theta} \eta_{\widehat{\B}(\theta)} \left( s \hspace{0.1 cm} ; c(\theta_{0}) \right)  \\
& = & \frac{1}{\Gamma(\frac{s+1}{2})} \int^{\infty}_{0} t^{\frac{s-1}{2}} \Tr \left( - {\dot Q}(\theta) \widehat{\B}(\theta) e^{-t \widehat{\B}(\theta)^{2}}  +
 \left( I - Q(\theta) \right) \frac{d}{d \theta} \left( \widehat{\B}(\theta) e^{- t \widehat{\B}(\theta)^{2}} \right) \right) dt  \nonumber \\
& = & \frac{1}{\Gamma(\frac{s+1}{2})} \int^{\infty}_{0} t^{\frac{s-1}{2}} \Tr \left( - {\dot Q}(\theta) \widehat{\B}(\theta) e^{-t \widehat{\B}(\theta)^{2}} \right) dt -
\frac{s}{\Gamma(\frac{s+1}{2})} \int^{\infty}_{0} t^{\frac{s-1}{2}}
\Tr \left\{ \left( I - Q(\theta) \right) \left( \dot{\widehat{\B}}(\theta) e^{- t \widehat{\B}(\theta)^{2}} \right) \right\} dt \nonumber,
\end{eqnarray}

\vspace{0.2 cm}

\noindent
where ${\dot Q(\theta)}$ and $\dot{\widehat{\B}(\theta)}$ mean the derivative of $Q(\theta)$ and $\widehat{\B}(\theta)$ with respect to $\theta$.
Furthermore, we have (cf. Section 4.2 in [11])
$$
\Tr \left( - {\dot Q}(\theta) \widehat{\B}(\theta) e^{-t \widehat{\B}(\theta)^{2}} \right) = 0, \qquad
\left\{ \frac{s}{\Gamma(\frac{s+1}{2})} \int^{\infty}_{0} t^{\frac{s-1}{2}}
\Tr \left( Q(\theta)  \dot{\widehat{\B}}(\theta) e^{- t \widehat{\B}(\theta)^{2}} \right) dt \right\}_{s=0} = 0.
$$

\noindent
These equalities imply that

\vspace{0.2 cm}

\begin{eqnarray}  \label{E:4.5}
 \frac{d}{d \theta} \eta_{\widehat{\B}(\theta)} \left( s \hspace{0.1 cm} ; c(\theta_{0}) \right) & = &
-  \hspace{0.1 cm} \frac{s}{\Gamma(\frac{s+1}{2})} \int^{\infty}_{0} t^{\frac{s-1}{2}}
\Tr \left( \dot{\widehat{\B}}(\theta) e^{- t \widehat{\B}(\theta)^{2}} \right) dt + F(s)
\end{eqnarray}

\vspace{0.2 cm}

\noindent
where $F(s)$ is an analytic function at least for $\operatorname{Re} s > -1$ with $F(0) = 0$.
The equality (\ref{E:4.5}) leads to the following lemma.

\vspace{0.2 cm}

\begin{lemma}  \label{Lemma:4.1}
(1) The derivative of the residue of $\eta_{\widehat{\B}(\theta)} (s)$ at $s=0$ is given by
$$
\frac{d}{d \theta} \Res_{s=0} \eta_{\widehat{\B}(\theta)} (s) \hspace{0.1 cm} = \hspace{0.1 cm}
\Res_{s=0} \left( \frac{d}{d \theta} \eta_{\widehat{\B}(\theta)} (s) \right) \hspace{0.1 cm} = \hspace{0.1 cm}
\frac{4}{\sqrt{\pi}} \hspace{0.1 cm} a_{- \frac{1}{2}, 1} \hspace{0.1 cm} (\widehat{\B}(\theta), \dot{\widehat{\B}}(\theta)).
$$
(2) If $\eta_{\widehat{\B}(\theta)} (s)$ is regular at $s=0$ for each $\theta$, the derivative of $\eta_{\widehat{\B}(\theta)} (0)$, up to $\Mod {\Bbb Z}$, is given by
$$
\left( \frac{d}{d \theta}\eta_{\widehat{\B}(\theta)} \right) (0)  \hspace{0.1 cm} = \hspace{0.1 cm}
- \frac{2}{\sqrt{\pi}} a_{- \frac{1}{2}, 0} \hspace{0.1 cm} (\widehat{\B}(\theta), \dot{\widehat{\B}}(\theta)) \qquad (\Mod {\Bbb Z}).
$$

\noindent
Here $a_{- \frac{1}{2}, 1} \hspace{0.1 cm} (\widehat{\B}(\theta), \dot{\widehat{\B}}(\theta))$ and
$a_{- \frac{1}{2}, 0} \hspace{0.1 cm} (\widehat{\B}(\theta), \dot{\widehat{\B}}(\theta))$
are the coefficients of $t^{- \frac{1}{2}} \log t$ and $t^{- \frac{1}{2}}$
in the asymptotic expansion of $\Tr \left( \dot{\widehat{\B}}(\theta) e^{- t \widehat{\B}(\theta)^{2}} \right)$ for $t \rightarrow 0^{+}$,
respectively.
\end{lemma}

\vspace{0.2 cm}

\begin{theorem}  \label{Theorem:4.2}
We recall that $\hspace{0.1 cm} \widehat{\B}(\theta) = \Psi_{\theta}^{\ast} \hspace{0.1 cm} \widetilde{\B}(\theta) \hspace{0.1 cm}
\Psi_{\theta} \hspace{0.1 cm} : \hspace{0.1 cm} {\Dom \left(\widetilde{\B}(0)_{{\widetilde P}(0)} \right)}
 \rightarrow \Omega^{\even}(M, E \oplus E)$. Then :

$$
\Tr \left( \dot{\widehat{\B}}(\theta) e^{- t \widehat{\B}(\theta)^{2}} \right) \hspace{0.1 cm} \sim \hspace{0.1 cm} 0 \quad
(\text{up to} \quad e^{- \frac{c}{t}}), \quad
\text{for} \quad t \rightarrow 0^{+}.
$$

\end{theorem}

\vspace{0.2 cm}

\noindent
Lemma \ref{Lemma:4.1} and Theorem \ref{Theorem:4.2} imply that for each $\theta$, $\eta_{\widehat{\B}(\theta)} (s)$ has a regular value at $s=0$.
Moreover, $\eta_{\widehat{\B}(\theta)} (0)$ and $\eta (\widehat{\B}(\theta))$ do not depend on $\theta$ up to $\Mod {\Bbb Z}$,
which yields the following result.

\vspace{0.2 cm}

\begin{corollary}  \label{Corollary:4.3}
\begin{eqnarray*}
\eta \left( {\widetilde \B}^{\even}_{(\R/\A)} \right) \hspace{0.1 cm} = \hspace{0.1 cm} \eta \left( \widehat{\B}(0) \right)
 \hspace{0.1 cm} = \hspace{0.1 cm} \eta \left( \widehat{\B}(\frac{\pi}{2}) \right) \hspace{0.1 cm} = \hspace{0.1 cm}
\eta \left( {\widetilde \B}^{\even}_{({\widetilde {\mathcal P}_{-}}/{\widetilde {\mathcal P}_{+}})} \right) \hspace{0.1 cm} = \hspace{0.1 cm}
\eta \left( \B_{{\mathcal P}_{-}} \right) - \eta \left( \B_{{\mathcal P}_{+}} \right)  \qquad (\Mod {\Bbb Z}).
\end{eqnarray*}
\end{corollary}

\vspace{0.2 cm}

\noindent
Corollary \ref{Corollary:3.3}, Corollary \ref{Corollary:4.3}, Theorem \ref{Theorem:2.11} and Theorem \ref{Theorem:2.4} lead to the following result,
which is the main result of this paper.

\vspace{0.2 cm}

\begin{theorem}   \label{Theorem:4.4}
Let $(M, g^{M})$ be an odd dimensional compact Riemannian manifold with boundary $Y$ and $g^{M}$ be a product metric near $Y$.
We assume that the connection $\nabla$ is a Hermitian connection and for each $0 \leq q \leq m$, $H^{q}(M ; E) = H^{q}(M, Y ; E) = \{ 0 \}$. Then : \newline

\vspace{0.05 cm}

\noindent
(1) $\hspace{0.2 cm} {\widetilde \rho}_{\an, (\m)}(g^{M}, {\widetilde \nabla}) \hspace{0.1 cm} = \hspace{0.1 cm}  \pm \hspace{0.1 cm}
 {\widetilde \rho}_{\an, ({\mathcal P}_{-}/{\mathcal P}_{+})}(g^{M}, {\widetilde \nabla})  = \hspace{0.1 cm}  \pm \hspace{0.1 cm}
\rho_{\an, {\mathcal P}_{-}}(g^{M}, \nabla) \cdot \overline{\rho_{\an, {\mathcal P}_{+}}(g^{M}, \nabla)}  \hspace{0.1 cm}
\in \hspace{0.1 cm} {\Bbb R}. $    \newline

\vspace{0.05 cm}

\noindent
(2) $\hspace{0.2 cm} \eta \left( {\widetilde \B}^{\even}_{(\m)} \right) \hspace{0.1 cm}  =  \hspace{0.1 cm}
\eta \left( {\widetilde \B}^{\even}_{(\R/\A)} \right) \hspace{0.1 cm} \equiv \hspace{0.1 cm}   0  \qquad (\Mod {\Bbb Z}) . $
\end{theorem}

\vspace{0.3 cm}

\section{Proof of Theorem \ref{Theorem:4.2}}

\vspace{0.2 cm}

Recall that
$$
\widehat{\B}(\theta) \hspace{0.1 cm} = \hspace{0.1 cm} \Psi_{\theta}^{\ast} \hspace{0.1 cm} \widetilde{\B}(\theta) \hspace{0.1 cm} \Psi_{\theta}
\hspace{0.1 cm} = \hspace{0.1 cm} e^{- i \phi(x) T(\theta)} \hspace{0.1 cm} \widetilde{\B}(\theta) \hspace{0.1 cm} e^{i \phi(x) T(\theta)}.
$$

\noindent
Since $T^{\prime}(\theta)$ does not commute with $T(\theta)$, we should be careful in computing $\dot{\widehat{\B}}(\theta)$.
We note that

\begin{eqnarray}  \label{E:5.1}
\dot{\widehat{\B}}(\theta) & = &
\left( \frac{d}{d \theta} e^{- i \phi(x) T(\theta)} \right) \widetilde{\B}(\theta) e^{i \phi(x) T(\theta)}
\hspace{0.1 cm} + \hspace{0.1 cm}
e^{- i \phi(x) T(\theta)} \left( \frac{d}{d \theta} \widetilde{\B}(\theta) \right) e^{i \phi(x) T(\theta)}  \nonumber  \\
& & \hspace{0.1 cm} + \hspace{0.1 cm}
e^{- i \phi(x) T(\theta)} \widetilde{\B}(\theta) \left( \frac{d}{d \theta} e^{i \phi(x) T(\theta)} \right),
\end{eqnarray}

\vspace{0.2 cm}

\noindent
which leads to

\begin{eqnarray}  \label{E:5.2}
& & \Tr \left( \dot{\widehat{\B}}(\theta) e^{- t {\widehat{\B}}(\theta)^{2}} \right)  \hspace{0.1 cm} = \hspace{0.1 cm}
\Tr \left(  \left( \frac{d}{d \theta} e^{- i \phi(x) T(\theta)} \right) \widetilde{\B}(\theta)
e^{- t \widetilde{\B}(\theta)_{{\widetilde P}(\theta)}^{2}} e^{i \phi(x) T(\theta)} \right)   \nonumber  \\
& + &
\Tr \left(  \left( \frac{d}{d \theta} \widetilde{\B}(\theta) \right)
e^{- t \widetilde{\B}(\theta)_{{\widetilde P}(\theta)}^{2}} \right)   \hspace{0.1 cm} + \hspace{0.1 cm}
\Tr \left( \widetilde{\B}(\theta) \left( \frac{d}{d \theta} e^{i \phi(x) T(\theta)} \right) e^{- i \phi(x) T(\theta)}
e^{- t \widetilde{\B}(\theta)_{{\widetilde P}(\theta)}^{2}} \right).
\end{eqnarray}

\vspace{0.2 cm}

\noindent
Simple computation leads to the following result.

\begin{eqnarray}   \label{E:5.3}
\frac{d}{d \theta} e^{- i \phi(x) T(\theta)} & = & - e^{- i \phi(x) T(\theta)} \int_{1}^{x} e^{i \phi(u) T(\theta)}
\left( i \phi^{\prime}(u) T^{\prime}(\theta) \right) e^{- i \phi(u) T(\theta)} du
\hspace{0.1 cm} = \hspace{0.1 cm} - e^{- i \phi(x) T(\theta)} {\mathcal Q}(x) \nonumber  \\
\frac{d}{d \theta} e^{i \phi(x) T(\theta)} & = & \int_{1}^{x} e^{i \phi(u) T(\theta)}
\left( i \phi^{\prime}(u) T^{\prime}(\theta) \right) e^{- i \phi(u) T(\theta)} du \hspace{0.1 cm} e^{i \phi(x) T(\theta)}
\hspace{0.1 cm} = \hspace{0.1 cm} {\mathcal Q}(x)   e^{i \phi(x) T(\theta)} ,
\end{eqnarray}

\noindent
where

\begin{equation}   \label{E:5.4}
{\mathcal Q}(x) :=  \int_{1}^{x} e^{i \phi(u) T(\theta)} \left( i \phi^{\prime}(u) T^{\prime}(\theta) \right) e^{- i \phi(u) T(\theta)} du.
\end{equation}

\vspace{0.2 cm}

\noindent
Since $\phi(u)$ is supported in $[0, 1]$, the support of ${\mathcal Q}(x)$ belongs to $[0, 1] \times Y$.
Equations (\ref{E:5.2}) and (\ref{E:5.3}) yield the following result.

\vspace{0.2 cm}

\begin{lemma}  \label{Lemma:5.1}
\begin{eqnarray*}
\Tr \left( \dot{\widehat{\B}}(\theta) e^{- t {\widehat{\B}}(\theta)^{2}} \right)
& = & \Tr \left(  \left( \frac{d}{d \theta} \widetilde{\B}(\theta) \right) e^{- t \widetilde{\B}_{{\widetilde P}(\theta)}^{2}} \right)
\hspace{0.1 cm} + \hspace{0.1 cm}
\Tr \left\{ \left( \widetilde{\B}(\theta) {\mathcal Q}(x) - {\mathcal Q}(x) \widetilde{\B}(\theta) \right)
e^{- t \widetilde{\B}(\theta)_{{\widetilde P}(\theta)}^{2}} \right\}.
\end{eqnarray*}
\end{lemma}

\vspace{0.3 cm}

Let $\B^{\cyl}$ be the odd signature operator defined as in (\ref{E:1.8}) on $[0, \infty) \times Y$ and

$$\widetilde{\B}(\theta)^{\cyl} := \B^{\cyl} \cdot \left( \begin{array}{clcr} \sin \theta \cdot \Id & \cos \theta \cdot \Id  \\
\cos \theta \cdot \Id & -\sin \theta \cdot \Id   \end{array} \right).
$$

\noindent
The heat kernel of $\left( \widetilde{\B}(\theta)^{\cyl}_{{\widetilde P}(\theta)} \right)^{2}$ was computed in [6] as follows.

\vspace{0.2 cm}

\begin{eqnarray*}
e^{-t \left( \widetilde{\B}(\theta)^{\cyl}_{{\widetilde P}(\theta)} \right)^{2}} (x, y)  & = & (4 \pi t)^{- \frac{1}{2}} \left( e^{- \frac{(x-y)^{2}}{4t}} +
(I - 2 {\widetilde P} (\theta)) e^{- \frac{(x+y)^{2}}{4t}} \right) e^{-t {\widetilde {\mathcal A}}^{2}}  \nonumber \\
& & + (\pi t)^{- \frac{1}{2}} \left( I - {\widetilde P} (\theta) \right) \int_{0}^{\infty} e^{- \frac{(x+y+z)^{2}}{4t}} {\widetilde {\mathcal A}}(\theta)
\hspace{0.1 cm} e^{{\widetilde {\mathcal A}}(\theta)z - t {\widetilde {\mathcal A}}^{2}} dz,
\end{eqnarray*}

\vspace{0.2 cm}

\noindent
where
${\widetilde {\mathcal A}}(\theta) := (I - {\widetilde P}(\theta)) \widetilde{\mathcal A} (I - {\widetilde P } (\theta))$.
Moreover, Lemma \ref{Lemma:3.3} shows that

\begin{eqnarray} \label{E:5.5}
e^{-t \left( \widetilde{\B}(\theta)^{\cyl}_{{\widetilde P}(\theta)} \right)^{2}} (x, y)  & = & (4 \pi t)^{- \frac{1}{2}} \left( e^{- \frac{(x-y)^{2}}{4t}} +
(I - 2 {\widetilde P} (\theta)) e^{- \frac{(x+y)^{2}}{4t}} \right) e^{-t {\widetilde {\mathcal A}}^{2}}.
\end{eqnarray}

\vspace{0.3 cm}

\subsection{Asymptotic expansion of $\Tr \left(  \left( \frac{d}{d \theta} \widetilde{\B}(\theta) \right)
e^{- t \widetilde{\B}(\theta)_{{\widetilde P}(\theta)}^{2}} \right)$}

Recall that $N = [0, 1) \times Y$ is a collar neighborhood of $Y$.
We define cut off functions $\phi_{1}$, $\phi_{2}$, $\psi_{1}$ and $\psi_{2}$ as in (\ref{E:5.55}).
We put

\begin{eqnarray} \label{E:5.6}
& & \hspace{0.5 cm} {\mathcal R}_{\even}(t, (w, x), (w^{\prime}, y)) (\theta)  \\
& = &   \phi_{1}(x) \hspace{0.1 cm}  \dot{{\widetilde \B}}(\theta)^{\cyl}_{\even} \hspace{0.1 cm}  e^{-t \left( \widetilde{\B}(\theta)^{\cyl}_{{\widetilde P}(\theta)} \right)^{2}} (x, y)
 \hspace{0.1 cm}  \psi_{1}(y) \hspace{0.1 cm}  + \hspace{0.1 cm}
\phi_{2}(x) \hspace{0.1 cm}  \dot{{\widetilde \B}}(\theta)_{\even} \hspace{0.1 cm}  {\widetilde \E}^{\doub}_{\even}(t, (w, x), (w^{\prime}, y)) \hspace{0.1 cm}  \psi_{2}(y),  \nonumber
\end{eqnarray}

\vspace{0.2 cm}

\vspace{0.2 cm}

\noindent
where $\dot{{\widetilde \B}}(\theta) := \frac{d}{d \theta} \widetilde{\B}(\theta)$ and
${\widetilde \E}^{\doub}_{\even}(t, (w, x), (w^{\prime}, y)) $ is the heat kernel of $e^{-t {\widetilde \B}(\theta)^{2}}$
on $M^{\doub} := M \cup_{Y} M$, the closed double of $M$.
Then ${\mathcal R}_{\even}(t, (w, x), (w^{\prime}, y)) (\theta)$ is a parametrix for
$\dot{\widetilde \B}(\theta)_{\even} {\widetilde \E}_{\even}(t, (w, x), (w^{\prime}, y)) (\theta)$, the kernel of
${\widetilde \B}(\theta)_{\even} e^{-t {\widetilde \B}(\theta)^{2}_{\even, {\widetilde P}(\theta)}}$ on $M$ and
the standard computation shows that for $0 < t \leq 1$,

\begin{equation} \label{E:5.7}
| {\widetilde \B}(\theta)_{\even} {\widetilde \E}_{\even}(t, (w, x), (w, x)) (\theta) - {\mathcal R}_{\even}(t, (w, x), (w, x)) (\theta) |
\leq c_{3} e^{-\frac{c_{4}}{t}}
\end{equation}

\noindent
for some positive constants $c_{3}$ and $c_{4}$,
which implies that

\begin{eqnarray}  \label{E:5.8}
\Tr \left(  \dot{\widetilde{\B}}(\theta) e^{- t \widetilde{\B}(\theta)_{{\widetilde P}(\theta)}^{2}} \right)
& \sim &
\Tr \left( {\mathcal R}_{\even}(t, (w, x), (w, x)) (\theta) \right)  \nonumber \\
& = &
\Tr \left(  \dot{\widetilde{\B}}(\theta)^{\cyl}
e^{- t (\widetilde{\B}(\theta)_{{\widetilde P}(\theta)}^{\cyl})^{2}} \psi_{1}(x) \right) \hspace{0.1 cm} + \hspace{0.1 cm}
\Tr \left(  \dot{\widetilde{\B}}(\theta) e^{- t (\widetilde{\B}(\theta))^{2}} \psi_{2}(x) \right) .
\end{eqnarray}

\vspace{0.2 cm}

\noindent
We note that

\begin{eqnarray*}
e^{- t (\widetilde{\B}(\theta))^{2}} \hspace{0.1 cm} = \hspace{0.1 cm} \left( \begin{array}{clcr} e^{-t \B^{2}} & 0 \\ 0 & e^{-t \B^{2}}
\end{array} \right) ,  \qquad
\dot{\widetilde{\B}}(\theta) \hspace{0.1 cm} = \hspace{0.1 cm}  \B \left(
\begin{array}{clcr} \cos \theta \Id & - \sin \theta \Id \\ - \sin \theta \Id & - \cos \theta \Id \end{array} \right) ,
\end{eqnarray*}

\noindent
which shows that

\begin{equation}  \label{E:5.9}
\Tr \left(  \dot{\widetilde{\B}}(\theta) e^{- t (\widetilde{\B}(\theta))^{2}} \psi_{2}(u) \right) \hspace{0.1 cm} = \hspace{0.1 cm} 0.
\end{equation}

Using (\ref{E:5.5}) and the decomposition (\ref{E:1.3}), we have

\begin{eqnarray}  \label{E:5.10}
& & \Tr \left(  \dot{\widetilde{\B}}(\theta)^{\cyl}
e^{- t (\widetilde{\B}(\theta)_{{\widetilde P}(\theta)}^{\cyl})^{2}} \psi_{1}(x) \right)  \nonumber \\
& = & \Tr \left\{ \left( \begin{array}{clcr} \cos \theta  & - \sin \theta   \\ - \sin \theta  & - \cos \theta  \end{array} \right)
{\mathcal \gamma} \left( \nabla_{\partial_{x}} + {\mathcal A} \right) \left\{ \frac{1}{\sqrt{4 \pi t}} \left( e^{- \frac{(x - y)^{2}}{4 t}} +
\left( I - 2 {\widetilde P}(\theta) \right) e^{- \frac{(x + y)^{2}}{4t}} \right) e^{- t {\widetilde {\mathcal A}}^{2}} \right\} \psi_{1}(x) \right\}
\nonumber  \\
& = & \left( \frac{1}{\sqrt{4 \pi t}} \int_{0}^{\infty} - \frac{x}{t} e^{- \frac{x^{2}}{t}} \psi_{1}(x) dx \right) \hspace{0.1 cm}
\Tr \left\{ \left( \begin{array}{clcr} \cos \theta  & - \sin \theta   \\ - \sin \theta  & - \cos \theta  \end{array} \right)
{\mathcal \gamma}
\left( I - 2 {\widetilde P}(\theta) \right) e^{- t {\widetilde {\mathcal A}}^{2}} \right\}   \nonumber \\
&  & + \hspace{0.1 cm} \left( \frac{1}{\sqrt{4 \pi t}} \int_{0}^{\infty} \psi_{1}(x) dx \right) \hspace{0.1 cm}
\Tr \left\{ \left( \begin{array}{clcr} \cos \theta  & - \sin \theta   \\ - \sin \theta  & - \cos \theta  \end{array} \right)
{\mathcal \gamma}
{\mathcal A} e^{- t {\widetilde {\mathcal A}}^{2}} \right\}   \nonumber \\
&  &  + \hspace{0.1 cm}  \left( \frac{1}{\sqrt{4 \pi t}} \int_{0}^{\infty} \psi_{1}(x) e^{- \frac{x^{2}}{t}} dx \right) \hspace{0.1 cm}
\Tr \left\{ \left( \begin{array}{clcr} \cos \theta  & - \sin \theta   \\ - \sin \theta  & - \cos \theta  \end{array} \right)
{\mathcal \gamma}
{\mathcal A} \left( I - 2 {\widetilde P}(\theta) \right)  e^{- t {\widetilde {\mathcal A}}^{2}} \right\} \nonumber \\
& = : & (\I) + (\II) + (\III).
\end{eqnarray}

\vspace{0.2 cm}

\noindent
Since ${\mathcal \gamma} {\mathcal A} = - {\mathcal A} {\mathcal \gamma}$, we have

\begin{equation}  \label{E:5.11}
(\II) \hspace{0.1 cm} = \hspace{0.1 cm} 0.
\end{equation}

\vspace{0.2 cm}

\noindent
We note that

\begin{eqnarray}  \label{E:5.12}
& & \Tr \left\{ \left( \begin{array}{clcr} \cos \theta  & - \sin \theta   \\ - \sin \theta  & - \cos \theta  \end{array} \right)
{\mathcal \gamma}
{\mathcal A} \left( I - 2 {\widetilde P}(\theta) \right)  e^{- t {\widetilde {\mathcal A}}^{2}} \right\}  \nonumber \\
& = & \Tr \left\{ \left( \begin{array}{clcr} \cos \theta  & - \sin \theta   \\ - \sin \theta  & - \cos \theta  \end{array} \right)
\left( \begin{array}{clcr} \sin \theta  & \cos \theta   \\ \cos \theta  & - \sin \theta  \end{array} \right)
{\widetilde {\mathcal \gamma}}(\theta)
{\mathcal A} \left( I - 2 {\widetilde P}(\theta) \right)  e^{- t {\widetilde {\mathcal A}}^{2}} \right\}.
\end{eqnarray}

\noindent
Since
$\hspace{0.1 cm} \left( \begin{array}{clcr} \cos \theta  & - \sin \theta   \\ - \sin \theta & - \cos \theta  \end{array} \right)$ and
$\left( \begin{array}{clcr} \sin \theta  & \cos \theta   \\ \cos \theta  & - \sin \theta  \end{array} \right), \hspace{0.1 cm}$
$\hspace{0.1 cm} {\widetilde {\mathcal \gamma}}(\theta)$ and $\hspace{0.1 cm} {\mathcal A} \hspace{0.1 cm}$,
$\hspace{0.1 cm} {\widetilde {\mathcal \gamma}}(\theta) $ and
$\left( I - 2 {\widetilde P}(\theta) \right) \hspace{0.1 cm}$ anticommute with each other (Lemma \ref{Lemma:3.3}), we have

\begin{eqnarray}  \label{E:5.13}
(\ref{E:5.12})
& = & - \Tr \left\{ {\widetilde {\mathcal \gamma}}(\theta) \left(
\begin{array}{clcr} \sin \theta  & \cos \theta   \\ \cos \theta  & - \sin \theta  \end{array} \right)
\left( \begin{array}{clcr} \cos \theta  & - \sin \theta   \\ - \sin \theta  & - \cos \theta  \end{array} \right)
{\mathcal A} \left( I - 2 {\widetilde P}(\theta) \right)  e^{- t {\widetilde {\mathcal A}}^{2}} \right\}  \nonumber  \\
& = & - \Tr \left\{ {\mathcal \gamma}
\left( \begin{array}{clcr} \cos \theta  & - \sin \theta   \\ - \sin \theta  & - \cos \theta  \end{array} \right)
{\mathcal A} \left( I - 2 {\widetilde P}(\theta) \right)  e^{- t {\widetilde {\mathcal A}}^{2}} \right\} \nonumber  \\
& = & - \Tr \left\{
\left( \begin{array}{clcr} \cos \theta  & - \sin \theta   \\ - \sin \theta  & - \cos \theta  \end{array} \right)
{\mathcal \gamma} {\mathcal A} \left( I - 2 {\widetilde P}(\theta) \right)  e^{- t {\widetilde {\mathcal A}}^{2}} \right\}
\hspace{0.1 cm} = \hspace{0.1 cm}  0 ,
\end{eqnarray}

\vspace{0.2 cm}

\noindent
which shows that

\begin{equation}   \label{E:5.133}
(\III) = 0.
\end{equation}

\noindent
For $(\I)$ we note that

\begin{eqnarray*}
I - 2 {\widetilde P}(\theta) & = & \left[ \begin{array}{clcr} 0 & - P(\theta)^{\ast} \\ - P(\theta) & 0 \end{array} \right]
\hspace{0.1 cm} = \hspace{0.1 cm}
- \hspace{0.1 cm} P(\theta)^{\ast} \hspace{0.1 cm} \frac{I + i {\widetilde {\mathcal \gamma}}(\theta)}{2} \hspace{0.1 cm}
- \hspace{0.1 cm} P(\theta) \hspace{0.1 cm} \frac{I - i {\widetilde {\mathcal \gamma}}(\theta)}{2}  \\
& = & - \hspace{0.1 cm} {\frak A}(\theta) \sin \theta \hspace{0.1 cm} + \hspace{0.1 cm} i \hspace{0.1 cm} {\frak B} \hspace{0.1 cm}
{\widetilde {\mathcal \gamma}}(\theta) \hspace{0.1 cm} \cos \theta,
\end{eqnarray*}

\noindent
which leads to

\begin{eqnarray*}
& & \Tr \left\{ \left( \begin{array}{clcr} \cos \theta  & - \sin \theta   \\ - \sin \theta  & - \cos \theta  \end{array} \right)
{\mathcal \gamma}
\left( I - 2 {\widetilde P}(\theta) \right) e^{- t {\widetilde {\mathcal A}}^{2}} \right\}  \\
& = & \Tr \left\{ \left( \begin{array}{clcr} \cos \theta  & - \sin \theta   \\ - \sin \theta  & - \cos \theta \end{array} \right) {\mathcal \gamma}
\left( - \hspace{0.1 cm} {\frak A}(\theta) \sin \theta \hspace{0.1 cm} + \hspace{0.1 cm} i \hspace{0.1 cm} {\frak B} \hspace{0.1 cm}
{\widetilde {\mathcal \gamma}}(\theta) \hspace{0.1 cm} \cos \theta \right) e^{- t {\widetilde {\mathcal A}}^{2}} \right\}  \\
& = & - \sin \theta \Tr \left\{ \left( \begin{array}{clcr} \cos \theta & - \sin \theta  \\ - \sin \theta & - \cos \theta \end{array} \right)
{\mathcal \gamma} \hspace{0.1 cm} {\frak A}(\theta) e^{- t {\widetilde {\mathcal A}}^{2}} \right\} + i \cos \theta
\Tr \left\{ \left( \begin{array}{clcr} \cos \theta & - \sin \theta  \\ - \sin \theta & - \cos \theta \end{array} \right)
{\mathcal \gamma}
{\frak B} \hspace{0.1 cm} {\widetilde {\mathcal \gamma}}(\theta) e^{- t {\widetilde {\mathcal A}}^{2}} \right\}.
\end{eqnarray*}

\vspace{0.2 cm}
\noindent
Let $\K = \left( \begin{array}{clcr} -1 & 0 \\ 0 & 1 \end{array} \right)$. Simple computation shows that

\begin{eqnarray}  \label{E:5.14}
& & \Tr \left\{ \left( \begin{array}{clcr} \cos \theta & - \sin \theta  \\ - \sin \theta & - \cos \theta \end{array} \right)
 {\mathcal \gamma} \hspace{0.1 cm} {\frak A}(\theta) e^{- t {\widetilde {\mathcal A}}^{2}} \right\}
  =
\Tr \left\{  {\mathcal \gamma} (\B_{Y}^{2})^{-1} \left( (\B_{Y}^{2})^{-} - (\B_{Y}^{2})^{+} \right) e^{- t \B_{Y}^{2}}
\left( \begin{array}{clcr} 0 & \Id  \\ - \Id & 0 \end{array} \right) \right\}  \nonumber \\
& & \hspace{1.0 cm} = \hspace{0.1 cm} 0,  \\
\label{E:5.15}
& & \Tr \left\{ \left( \begin{array}{clcr} \cos \theta & - \sin \theta  \\ - \sin \theta & - \cos \theta \end{array} \right)
{\mathcal \gamma}
{\frak B} \hspace{0.1 cm} {\widetilde {\mathcal \gamma}}(\theta) e^{- t {\widetilde {\mathcal A}}^{2}} \right\}
\hspace{0.1 cm} = \hspace{0.1 cm}
\Tr \left( \left( \begin{array}{clcr} 0 & \Id \\ - \Id & 0 \end{array} \right) {\frak B} e^{- t {\widetilde{\mathcal A}}^{2}} \right)
\nonumber \\
& & \hspace{1.0 cm}  =  \hspace{0.1 cm}
 \Tr \left( \beta \Gamma^{Y} \left( \begin{array}{clcr} - \K & 0 \\ 0 & - \K \end{array} \right) e^{- t \B_{Y}^{2}} \right)
\hspace{0.1 cm} = \hspace{0.1 cm}
- 2 \Tr \left( \beta \Gamma^{Y} \left( \begin{array}{clcr} - 1 & 0 \\ 0 & 1 \end{array} \right) e^{- t \B_{Y}^{2}} \right)
 \nonumber \\
& & \hspace{1.0 cm} =  \hspace{0.1 cm}
 2 \Tr \left( \Gamma^{Y} e^{- t \B_{Y}^{2}}|_{\Omega^{\even}(Y, E|_{Y})} + \Gamma^{Y} e^{- t \B_{Y}^{2}}|_{\Omega^{\odd}(Y, E|_{Y})} \right)
\hspace{0.1 cm} = \hspace{0.1 cm} 0 .
\end{eqnarray}

\vspace{0.2 cm}

\noindent
In the last equality we used the fact that
$\hspace{0.1 cm} \Tr \left(\ast_{Y} e^{- t \B_{Y}^{2}}|_{\Omega^{\even}(Y, E|_{Y})} \right)
\hspace{0.1 cm} = \hspace{0.1 cm}
\Tr \left(\ast_{Y} e^{- t \B_{Y}^{2}}|_{\Omega^{\odd}(Y, E|_{Y})} \right)
\hspace{0.1 cm} = 0 \hspace{0.1 cm}$
since $H^{\bullet}(Y, E|_{Y}) = 0$.
Hence $(\I) = 0$.
Equations from (\ref{E:5.9}) to (\ref{E:5.15}) show that

\begin{equation}  \label{E:5.16}
\Tr \left(  \left( \frac{d}{d \theta} \widetilde{\B}(\theta) \right)
e^{- t \widetilde{\B}(\theta)_{{\widetilde P}(\theta)}^{2}} \right) \hspace{0.1 cm} = \hspace{0.1 cm} 0  .
\end{equation}

\vspace{0.3 cm}

\subsection{Asymptotic expansion of $ \Tr \left\{ \left( \widetilde{\B}(\theta) {\mathcal Q}(x) -
{\mathcal Q}(x) \widetilde{\B}(\theta) \right) e^{- t \widetilde{\B}(\theta)_{{\widetilde P}(\theta)}^{2}} \right\} $}

Since ${\mathcal Q}(x)$ is supported in $[0, 1] \times Y$, the standard theory for heat kernel ([1], [3]) implies
that the asymptotic expansions of $\Tr \left( \left( \widetilde{\B}(\theta) {\mathcal Q}(x) -
{\mathcal Q}(x) \widetilde{\B}(\theta) \right) e^{-t \widetilde{\B}(\theta)_{{\widetilde P}(\theta)}^{2}} \right)$ and
$\Tr \left( \left( \widetilde{\B}(\theta)^{\cyl} {\mathcal Q}(x) -
{\mathcal Q}(x) \widetilde{\B}(\theta)^{\cyl} \right)
e^{-t \left( \widetilde{\B}(\theta)^{\cyl}_{{\widetilde P}(\theta)} \right)^{2}} \right)$
are equal up to $\left( e^{- \frac{c}{t}} \right)$ for some $c > 0$.
With a little abuse of notation we write $\widetilde{\B}(\theta)^{\cyl}$ by $\widetilde{\B}(\theta)$ again.
Simple computation using (\ref{E:3.4}) shows that

\begin{eqnarray}   \label{E:5.17}
& & \Tr \left\{ \left( \widetilde{\B}(\theta) {\mathcal Q}(x) -
{\mathcal Q}(x) \widetilde{\B}(\theta) \right) e^{- t \widetilde{\B}(\theta)_{{\widetilde P}(\theta)}^{2}} \right\}
\hspace{0.1 cm} = \hspace{0.1 cm}
 \Tr \left\{ {\widetilde {\mathcal \gamma}} (\theta) e^{i \phi(x) T(\theta)} \left( i \phi^{\prime}(x) T^{\prime}(\theta) \right) e^{- i \phi(x) T(\theta)}
e^{- t \widetilde{\B}(\theta)_{{\widetilde P}(\theta)}^{2}} \right\}  \nonumber \\
& + &
\Tr \left\{ [ {\widetilde {\mathcal \gamma}} (\theta)\hspace{0.1 cm},  {\mathcal Q}(x) ]
\nabla_{\partial_{x}} e^{- t \widetilde{\B}(\theta)_{{\widetilde P}(\theta)}^{2}} \right\} \hspace{0.1 cm} + \hspace{0.1 cm}
\Tr \left\{ [ {\widetilde {\mathcal \gamma}} (\theta) {\widetilde {\mathcal A}}\hspace{0.1 cm},  {\mathcal Q}(x) ]
e^{- t \widetilde{\B}(\theta)_{{\widetilde P}(\theta)}^{2}} \right\} \hspace{0.1 cm} =: \hspace{0.1 cm} (\I) + (\II) + (\III).
\end{eqnarray}

\vspace{0.2 cm}

\noindent
The following lemma is straightforward by using (\ref{E:3.13}), Lemma \ref{Lemma:3.2}, Lemma \ref{Lemma:3.5} and (\ref{E:5.4}).

\vspace{0.3 cm}

\begin{lemma} \label{Lemma:5.2}
\begin{eqnarray*}
& (1) & T^{\prime}(\theta) = i {\frak B}^{\ast} {\frak A}(\theta) \frac{I - i {\widetilde {\mathcal \gamma}}(\theta)}{2} +
\frac{i \theta}{2} {\frak B}^{\ast} {\frak A}^{\prime}(\theta) .\\
& (2) &  {\mathcal Q}(x) = - \phi(x) \left( {\frak B}^{\ast} {\frak A}(\theta) \frac{I - i {\widetilde {\mathcal \gamma}}(\theta)}{2} \right)
- \frac{\theta}{2} {\widetilde {\mathcal Q}}(x), \qquad
{\widetilde {\mathcal Q}}(x) := \int^{x}_{1} \phi^{\prime}(u) e^{i \phi(u) T(\theta) } {\frak B}^{\ast}
{\frak A}^{\prime}(\theta) e^{- i \phi(u) T(\theta) } du . \\
& (3) & {\widetilde {\mathcal \gamma}}(\theta) \hspace{0.1 cm} {\widetilde {\mathcal Q}}(x) \hspace{0.1 cm} = \hspace{0.1 cm}
- \hspace{0.1 cm} {\widetilde {\mathcal Q}}(x) \hspace{0.1 cm} {\widetilde {\mathcal \gamma}}(\theta),
\qquad
{\widetilde {\mathcal \gamma}}(\theta) \left( {\frak B}^{\ast} {\frak A}(\theta) \frac{I - i {\widetilde {\mathcal \gamma}}(\theta)}{2} \right)
\hspace{0.1 cm} = \hspace{0.1 cm}
\left( {\frak B}^{\ast} {\frak A}(\theta) \frac{I - i {\widetilde {\mathcal \gamma}}(\theta)}{2} \right)  {\widetilde {\mathcal \gamma}}(\theta) .
\end{eqnarray*}
\end{lemma}

\vspace{0.2 cm}

\noindent
Using (\ref{E:5.5}) and Lemma \ref{Lemma:5.2} we have

\begin{eqnarray}   \label{E:5.18}
(\II) & = & \Tr \left\{ [ {\widetilde {\mathcal \gamma}} (\theta)\hspace{0.1 cm},  {\mathcal Q}(x) ]
\nabla_{\partial_{x}} e^{- t \widetilde{\B}(\theta)_{{\widetilde P}(\theta)}^{2}} \right\}
\hspace{0.1 cm} = \hspace{0.1 cm} - \theta
\Tr \left\{  {\widetilde {\mathcal \gamma}} (\theta) \hspace{0.1 cm}  {\widetilde {\mathcal Q}}(x) \hspace{0.1 cm}
\nabla_{\partial_{x}} \hspace{0.1 cm} e^{- t \widetilde{\B}(\theta)_{{\widetilde P}(\theta)}^{2}} \right\} \nonumber \\
& = & \frac{- \theta}{\sqrt{4 \pi t}} \int_{0}^{\infty}
\Tr \left\{ {\widetilde {\mathcal \gamma}} (\theta)\hspace{0.1 cm}  {\widetilde {\mathcal Q}}(x)
\left( - \frac{x - y}{2 t} e^{- \frac{(x-y)^{2}}{4t}} + \left( I - 2 {\widetilde P}(\theta) \right) \left( - \frac{x+y}{2t} \right)
e^{- \frac{(x+y)^{2}}{4t}} \right)_{x=y} e^{- t {\widetilde {\mathcal A}}^{2}} \right\} dx  \nonumber \\
& = & \frac{\theta}{\sqrt{4 \pi t}} \left\{ \int_{0}^{\infty} \frac{x}{t} e^{- \frac{x^{2}}{t}}  \Tr \left(
{\widetilde {\mathcal \gamma}} (\theta)\hspace{0.1 cm} {\widetilde {\mathcal Q}}(x) \left( I - 2 {\widetilde P}(\theta) \right)
e^{- t {\widetilde {\mathcal A}}^{2}} \right) dx \right\} \nonumber \\
& = & \frac{-2 \theta}{\sqrt{4 \pi t}} \left\{ \int_{0}^{\infty} \frac{x}{t} e^{- \frac{x^{2}}{t}}  \Tr \left(
{\widetilde {\mathcal \gamma}} (\theta)\hspace{0.1 cm} {\widetilde {\mathcal Q}}(x) {\widetilde P}(\theta)
e^{- t {\widetilde {\mathcal A}}^{2}} \right) dx \right\} .
\end{eqnarray}

\vspace{0.2 cm}

\noindent
Using the decomposition (\ref{E:3.009}), (\ref{E:3.13}) and Lemma \ref{Lemma:3.2}, we have

\begin{eqnarray}   \label{E:5.20}
{\widetilde {\mathcal \gamma}} (\theta)\hspace{0.1 cm} {\widetilde {\mathcal Q}}(x) {\widetilde P}(\theta) & = &
\int^{x}_{1} \phi^{\prime}(u) \hspace{0.1 cm} e^{i \phi(u) T(\theta) } \hspace{0.1 cm} {\widetilde {\mathcal \gamma}} (\theta)
 \hspace{0.1 cm} {\frak B}^{\ast} \hspace{0.1 cm} {\frak A}^{\prime}(\theta) \hspace{0.1 cm}
e^{- i \phi(u) T(\theta) } \hspace{0.1 cm} {\widetilde P}(\theta) \hspace{0.1 cm} du ,  \nonumber   \\
e^{\pm i \phi(u) T(\theta) } & = &
\left[ \begin{array}{clcr} \cos ( \phi(u) \theta ) \mp \sin ( \phi(u) \theta )
\hspace{0.1 cm} {\frak B}^{\ast} \hspace{0.1 cm} {\frak A}(\theta) & 0 \\ 0 & \Id \end{array} \right] .
\end{eqnarray}

\noindent
Since $\hspace{0.1 cm} {\widetilde {\mathcal \gamma}} (\theta) = \left[ \begin{array}{clcr} i & 0 \\ 0 & -i \end{array} \right] \hspace{0.1 cm}$ and
$\hspace{0.1 cm} {\frak B}^{\ast} {\frak A}^{\prime}(\theta) = \left[ \begin{array}{clcr} 0 & {\frak B}^{\ast} {\frak A}^{\prime}(\theta) \\
{\frak B}^{\ast} {\frak A}^{\prime}(\theta) & 0 \end{array} \right] \hspace{0.1 cm}$
with respect to the decomposition (\ref{E:3.009}), simple computation shows that

\begin{eqnarray}  \label{E:5.21}
& & e^{i \phi(u) T(\theta) } \hspace{0.1 cm} {\widetilde {\mathcal \gamma}} (\theta)
 \hspace{0.1 cm} {\frak B}^{\ast} \hspace{0.1 cm} {\frak A}^{\prime}(\theta) \hspace{0.1 cm}
e^{- i \phi(u) T(\theta) } \hspace{0.1 cm} {\widetilde P}(\theta) \nonumber  \\
& = &  \frac{i}{2} \left\{ \cos ( \phi(u) \theta ) {\frak B}^{\ast} {\frak A}^{\prime}(\theta) +
\sin ( \phi(u) \theta ) {\frak A}(\theta) {\frak A}^{\prime}(\theta) \right\}
\left[ \begin{array}{clcr} P(\theta) & \Id  \\
- \Id & - P(\theta)^{\ast} \end{array} \right]  \nonumber  \\
& =  & \frac{i}{2} \left\{ \cos ( \phi(u) \theta ) {\frak B}^{\ast} {\frak A}^{\prime}(\theta) +
\sin ( \phi(u) \theta ) {\frak A}(\theta) {\frak A}^{\prime}(\theta) \right\} \nonumber  \\
& & \times \left\{
P(\theta) \frac{I - i {\widetilde {\mathcal \gamma}}(\theta)}{2} - P(\theta)^{\ast} \frac{I + i {\widetilde {\mathcal \gamma}}(\theta)}{2}
+ \frac{I + i {\widetilde {\mathcal \gamma}}(\theta)}{2} - \frac{I - i {\widetilde {\mathcal \gamma}}(\theta)}{2} \right\} \nonumber  \\
& =  & \frac{i}{2} \left\{ \cos ( \phi(u) \theta ) {\frak B}^{\ast} {\frak A}^{\prime}(\theta) +
\sin ( \phi(u) \theta ) {\frak A}(\theta) {\frak A}^{\prime}(\theta) \right\} \left\{
{\frak B} \cos \theta - i {\frak A}(\theta) {\widetilde {\mathcal \gamma}} (\theta) \sin \theta + i {\widetilde {\mathcal \gamma}} (\theta) \right\}.
\end{eqnarray}

\vspace{0.2 cm}

\noindent
Since ${\widetilde {\mathcal \gamma}} (\theta)$ anticommutes with ${\frak B}^{\ast} {\frak A}^{\prime}(\theta)$ and
${\frak A}(\theta) {\frak A}^{\prime}(\theta)$ (cf. Lemma \ref{Lemma:3.2}), we have

\begin{eqnarray}  \label{E:5.22}
\Tr \left\{ \left( \cos ( \phi(u) \theta ) {\frak B}^{\ast} {\frak A}^{\prime}(\theta) +
\sin ( \phi(u) \theta ) {\frak A}(\theta) {\frak A}^{\prime}(\theta) \right) \left( i {\widetilde {\mathcal \gamma}} (\theta)
e^{-t {\widetilde {\mathcal A}}^{2}} \right)
\right\}
= 0.
\end{eqnarray}

\noindent
Using Lemma \ref{Lemma:3.2} again, we have

\begin{eqnarray}  \label{E:5.23}
& & \Tr \left\{ \left( \cos ( \phi(u) \theta ) {\frak B}^{\ast} {\frak A}^{\prime}(\theta) +
\sin ( \phi(u) \theta ) {\frak A}(\theta) {\frak A}^{\prime}(\theta) \right) \left(
{\frak B} \cos \theta - i {\frak A}(\theta) {\widetilde {\mathcal \gamma}} (\theta) \sin \theta \right)
e^{-t {\widetilde {\mathcal A}}^{2}} \right\}    \\
& = & \cos (\phi(u) \theta) \cos \theta \hspace{0.1 cm} \Tr \left( {\frak A}^{\prime} (\theta)
e^{-t {\widetilde {\mathcal A}}^{2}} \right)  \hspace{0.1 cm} - \hspace{0.1 cm}
\sin (\phi(u) \theta) \cos \theta \hspace{0.1 cm} \Tr \left( \left( \begin{array}{clcr} 0 & \Id \\ - \Id & 0 \end{array} \right)
{\frak B} e^{-t {\widetilde {\mathcal A}}^{2}} \right)  \nonumber  \\
& & - \hspace{0.1 cm} i \hspace{0.1 cm} \cos (\phi(u) \theta) \sin \theta \hspace{0.1 cm}
\Tr \left( {\frak B}^{\ast} \left( \begin{array}{clcr} 0 & \Id \\ - \Id & 0 \end{array} \right)
{\widetilde {\mathcal \gamma}} (\theta) e^{-t {\widetilde {\mathcal A}}^{2}} \right) \hspace{0.1 cm} + \hspace{0.1 cm}
i \hspace{0.1 cm} \sin (\phi(u) \theta) \sin \theta \hspace{0.1 cm} \Tr \left( {\frak A}^{\prime} (\theta)
{\widetilde {\mathcal \gamma}} (\theta) e^{-t {\widetilde {\mathcal A}}^{2}} \right) .  \nonumber
\end{eqnarray}

\vspace{0.2 cm}

\noindent
Simple computation shows that for $\K = \left( \begin{array}{clcr} -1 & 0 \\ 0 & 1 \end{array} \right)$

\begin{eqnarray}   \label{E:5.24}
& & \Tr \left( {\frak A}^{\prime} (\theta) e^{-t {\widetilde {\mathcal A}}^{2}} \right) \hspace{0.1 cm} = \hspace{0.1 cm}
\Tr \left( {\frak B}^{\ast} \left( \begin{array}{clcr} 0 & \Id \\ - \Id & 0 \end{array} \right)
{\widetilde {\mathcal \gamma}} (\theta) e^{-t {\widetilde {\mathcal A}}^{2}} \right)  \hspace{0.1 cm} = \hspace{0.1 cm}
\Tr \left( {\frak A}^{\prime} (\theta) {\widetilde {\mathcal \gamma}} (\theta) e^{-t {\widetilde {\mathcal A}}^{2}} \right)
\hspace{0.1 cm} = \hspace{0.1 cm} 0 ,   \\
& & \Tr \left( \left( \begin{array}{clcr} 0 & \Id \\ - \Id & 0 \end{array} \right) {\frak B} e^{-t {\widetilde {\mathcal A}}^{2}} \right)
\hspace{0.1 cm} = \hspace{0.1 cm}
\Tr \left( \beta \Gamma^{Y} e^{-t {\widetilde {\mathcal A}}^{2}} \left( \begin{array}{clcr} - \K & 0 \\ 0 & - \K \end{array} \right) \right)
\hspace{0.1 cm} = \hspace{0.1 cm} - 2 \Tr \left( \beta \Gamma^{Y} e^{-t \B_{Y}^{2}} \K \right)  \nonumber  \\
&  & \hspace{0.5 cm} = \hspace{0.1 cm} 2 \Tr \left( \Gamma^{Y} e^{-t \B_{Y}^{2}}|_{\Omega^{\even}(Y, E|_{Y})}
+ \Gamma^{Y} e^{-t \B_{Y}^{2}}|_{\Omega^{\odd}(Y, E|_{Y})} \right) \hspace{0.1 cm} = \hspace{0.1 cm} 0.
\end{eqnarray}

\noindent
Hence, equations from (\ref{E:5.18}) to (\ref{E:5.24}) show that

\begin{eqnarray}  \label{E:5.26}
(\II) \hspace{0.1 cm} = \hspace{0.1 cm} \Tr \left\{ [ {\widetilde {\mathcal \gamma}} (\theta)\hspace{0.1 cm},  {\mathcal Q}(x) ]
\nabla_{\partial_{x}} e^{- t \widetilde{\B}(\theta)_{{\widetilde P}(\theta)}^{2}} \right\} \hspace{0.1 cm} = \hspace{0.1 cm} 0 .
\end{eqnarray}

\vspace{0.3 cm}

Using (\ref{E:5.5}), we have

\begin{eqnarray}  \label{E:5.27}
(\I) & = & \Tr \left\{ {\widetilde {\mathcal \gamma}} (\theta) e^{i \phi(x) T(\theta)} \left( i \phi^{\prime}(x) T^{\prime}(\theta) \right) e^{- i \phi(x) T(\theta)} e^{- t \widetilde{\B}(\theta)_{{\widetilde P}(\theta)}^{2}} \right\}  \nonumber   \\
& = & \int_{0}^{\infty} \Tr \left\{ {\widetilde {\mathcal \gamma}} (\theta) e^{i \phi(x) T(\theta)} \left( i \phi^{\prime}(x) T^{\prime}(\theta) \right) e^{- i \phi(x) T(\theta)}
\left\{ \frac{1}{\sqrt{4 \pi t}} \left( I + \left( I - 2 {\widetilde P}(\theta) \right) e^{- \frac{x^{2}}{t}} \right)
e^{- t {\widetilde {\mathcal A}}^{2}} \right\}
\right\} dx    \nonumber \\
& = & \frac{i}{\sqrt{4 \pi t}} \left( \int_{0}^{\infty} \phi^{\prime}(x) dx \right)
\Tr \left( {\widetilde {\mathcal \gamma}} (\theta) T^{\prime}(\theta) e^{- t {\widetilde {\mathcal A}}^{2}} \right)   \nonumber \\
& & + \hspace{0.1 cm} \frac{i}{\sqrt{4 \pi t}} \left\{ \int_{0}^{\infty} \phi^{\prime}(x) e^{- \frac{x^{2}}{t}}
\Tr \left( {\widetilde {\mathcal \gamma}} (\theta) e^{i \phi(x) T(\theta)} T^{\prime}(\theta) e^{- i \phi(x) T(\theta)}
\left( I - 2 {\widetilde P}(\theta) \right) e^{- t {\widetilde {\mathcal A}}^{2}} \right) dx \right\}  \nonumber \\
& =: & (\I_{1}) \hspace{0.1 cm} + \hspace{0.1 cm} (\I_{2}).
\end{eqnarray}

\noindent
Since $\phi^{\prime}(x) =0$ near $x=0$ and has a compact support, we have

\begin{eqnarray}  \label{E:5.28}
(\I_{2}) \hspace{0.1 cm} = \hspace{0.1 cm} O(e^{- \frac{c}{t}}).
\end{eqnarray}

\noindent
Using the assertion (1) in Lemma \ref{Lemma:5.2} and Lemma \ref{Lemma:3.2}, we have

\begin{eqnarray}  \label{E:5.30}
\Tr \left( {\widetilde {\mathcal \gamma}} (\theta) T^{\prime}(\theta) e^{- t {\widetilde {\mathcal A}}^{2}} \right) & = &
\hspace{0.1 cm} i \hspace{0.1 cm} \Tr \left( {\widetilde {\mathcal \gamma}} (\theta) {\frak B}^{\ast} \hspace{0.1 cm} {\frak A}(\theta) \hspace{0.1 cm}
\frac{I - i {\widetilde {\mathcal \gamma}}(\theta)}{2} \hspace{0.1 cm} e^{- t {\widetilde {\mathcal A}}^{2}} \right)
\hspace{0.1 cm} + \hspace{0.1 cm} \frac{i \theta}{2} \hspace{0.1 cm}
\Tr \left( {\widetilde {\mathcal \gamma}}(\theta) {\frak B}^{\ast} \hspace{0.1 cm} {\frak A}^{\prime}(\theta) \hspace{0.1 cm}
e^{- t {\widetilde {\mathcal A}}^{2}} \right)
\nonumber \\
& = & \hspace{0.1 cm} i \hspace{0.1 cm} \Tr \left( {\widetilde {\mathcal \gamma}} (\theta) {\frak B}^{\ast} \hspace{0.1 cm} {\frak A}(\theta) \hspace{0.1 cm} \frac{I - i {\widetilde {\mathcal \gamma}}(\theta)}{2} \hspace{0.1 cm} e^{- t {\widetilde {\mathcal A}}^{2}} \right)  \nonumber \\
& = &  \hspace{0.1 cm} \frac{i}{2} \hspace{0.1 cm} \Tr \left( {\widetilde {\mathcal \gamma}} (\theta) {\frak B}^{\ast} \hspace{0.1 cm} {\frak A}(\theta) \hspace{0.1 cm} e^{- t {\widetilde {\mathcal A}}^{2}} \right)
- \hspace{0.1 cm} \frac{1}{2} \hspace{0.1 cm} \Tr \left( {\frak B}^{\ast} \hspace{0.1 cm} {\frak A}(\theta) \hspace{0.1 cm}
e^{- t {\widetilde {\mathcal A}}^{2}} \right) .
\end{eqnarray}

\noindent
For $\K = \left( \begin{array}{clcr} -1 & 0 \\ 0 & 1 \end{array} \right)$ and the decomposition (\ref{E:1.3}), we have

\begin{eqnarray}  \label{E:5.31}
{\frak B}^{\ast} \hspace{0.1 cm} {\frak A}(\theta) & = & \hspace{0.1 cm} - \hspace{0.1 cm} \beta \hspace{0.1 cm} \Gamma^{Y} \hspace{0.1 cm}
\hspace{0.1 cm} \left( \B_{Y}^{2} \right)^{-1} \left( \left( \B_{Y}^{2} \right)^{-} -
\left( \B_{Y}^{2} \right)^{+} \right)
\left( \begin{array}{clcr} \cos \theta \hspace{0.1 cm} \K & - \sin \theta \hspace{0.1 cm} \K \\
- \sin \theta \hspace{0.1 cm} \K & - \cos \theta \hspace{0.1 cm} \K \end{array} \right),  \nonumber  \\
{\widetilde {\mathcal \gamma}}(\theta) \hspace{0.1 cm} {\frak B}^{\ast} \hspace{0.1 cm} {\frak A}(\theta) & = &
{\frak B}^{\ast} \hspace{0.1 cm} {\frak A}(\theta) \hspace{0.1 cm} {\widetilde {\mathcal \gamma}}(\theta) \hspace{0.1 cm} = \hspace{0.1 cm}  - \hspace{0.1 cm} i \hspace{0.1 cm}
\hspace{0.1 cm} \left( \B_{Y}^{2} \right)^{-1} \left( \left( \B_{Y}^{2} \right)^{-} - \left( \B_{Y}^{2} \right)^{+} \right)
\left( \begin{array}{clcr} 0 &  \K \\ - \K & 0 \end{array} \right),
\end{eqnarray}

\noindent
which imply that

\begin{eqnarray}  \label{E:5.32}
\Tr \left( {\frak B}^{\ast} \hspace{0.1 cm} {\frak A}(\theta) \hspace{0.1 cm} e^{- t {\widetilde {\mathcal A}}^{2}} \right)
\hspace{0.1 cm} = \hspace{0.1 cm}
\Tr \left( {\widetilde {\mathcal \gamma}} (\theta) {\frak B}^{\ast} \hspace{0.1 cm} {\frak A}(\theta) \hspace{0.1 cm}
e^{- t {\widetilde {\mathcal A}}^{2}} \right)
\hspace{0.1 cm} = \hspace{0.1 cm} 0
\end{eqnarray}

\noindent
and hence

\begin{eqnarray}  \label{E:5.33}
(\I) \hspace{0.1 cm} = \hspace{0.1 cm} \Tr \left\{ {\widetilde {\mathcal \gamma}} (\theta) e^{i \phi(x) T(\theta)}
\left( i \phi^{\prime}(x) T^{\prime}(\theta) \right) e^{- i \phi(x) T(\theta)}
e^{- t \widetilde{\B}(\theta)_{{\widetilde P}(\theta)}^{2}} \right\} \hspace{0.1 cm} = \hspace{0.1 cm} O(e^{- \frac{c}{t}}) \quad \text {for some} \quad
c > 0 .
\end{eqnarray}

\vspace{0.3 cm}

Finally, we are going to analyze $(\III) := \Tr \left\{ [ {\widetilde {\mathcal \gamma}} (\theta) {\widetilde {\mathcal A}}\hspace{0.1 cm},  {\mathcal Q}(x) ]
e^{- t \widetilde{\B}(\theta)_{{\widetilde P}(\theta)}^{2}} \right\}$ in (\ref{E:5.17}).
Using (\ref{E:5.5}) and Lemma \ref{Lemma:5.2}, we have

\begin{eqnarray}   \label{E:5.36}
(\III) & = & \Tr \left\{ [ \hspace{0.1 cm} {\widetilde {\mathcal \gamma}} (\theta) \hspace{0.1 cm} {\widetilde {\mathcal A}}\hspace{0.1 cm}, \hspace{0.1 cm}
{\mathcal Q}(x) \hspace{0.1 cm} ] \hspace{0.1 cm}
e^{- t \widetilde{\B}(\theta)_{{\widetilde P}(\theta)}^{2}} \right\}   \nonumber  \\
& = &
\Tr \left\{ [ \hspace{0.1 cm} {\widetilde {\mathcal \gamma}} (\theta) \hspace{0.1 cm} {\widetilde {\mathcal A}}\hspace{0.1 cm}, \hspace{0.1 cm}
{\mathcal Q}(x) \hspace{0.1 cm} ] \hspace{0.1 cm}
\left( \frac{1}{\sqrt{4 \pi t}} \hspace{0.1 cm} \left( I + \left( I - 2 {\widetilde P}(\theta) \right) e^{- \frac{x^{2}}{t}} \right) \right)
 \hspace{0.1 cm} e^{- t {\widetilde {\mathcal A}}^{2}} \right\}   \nonumber \\
& = &
\frac{1}{\sqrt{4 \pi t}} \hspace{0.1 cm} \int_{0}^{\infty} \left( ( 1 + e^{- \frac{x^{2}}{t}} ) \hspace{0.1 cm}
\Tr \left( [ \hspace{0.1 cm} {\widetilde {\mathcal \gamma}} (\theta) \hspace{0.1 cm} {\widetilde {\mathcal A}}\hspace{0.1 cm}, \hspace{0.1 cm}
{\mathcal Q}(x) \hspace{0.1 cm} ] \hspace{0.1 cm} e^{- t {\widetilde {\mathcal A}}^{2}} \right) \right) dx
\nonumber  \\
& &  - \hspace{0.1 cm}
\frac{2}{\sqrt{4 \pi t}} \hspace{0.1 cm} \int_{0}^{\infty} \left( e^{- \frac{x^{2}}{t}} \hspace{0.1 cm}
\Tr \left( [ \hspace{0.1 cm} {\widetilde {\mathcal \gamma}} (\theta) \hspace{0.1 cm} {\widetilde {\mathcal A}}\hspace{0.1 cm}, \hspace{0.1 cm}
{\mathcal Q}(x) \hspace{0.1 cm} ] \hspace{0.1 cm}
{\widetilde P}(\theta) \hspace{0.1 cm} e^{- t {\widetilde {\mathcal A}}^{2}} \right) \right) dx  \nonumber \\
& = &  - \hspace{0.1 cm}
\frac{2}{\sqrt{4 \pi t}} \hspace{0.1 cm} \int_{0}^{\infty} \left( e^{- \frac{x^{2}}{t}} \hspace{0.1 cm}
\Tr \left( [ \hspace{0.1 cm} {\widetilde {\mathcal \gamma}} (\theta) \hspace{0.1 cm} {\widetilde {\mathcal A}}\hspace{0.1 cm}, \hspace{0.1 cm}
{\mathcal Q}(x) \hspace{0.1 cm} ] \hspace{0.1 cm}
{\widetilde P}(\theta) \hspace{0.1 cm} e^{- t {\widetilde {\mathcal A}}^{2}} \right) \right) dx  \nonumber \\
& = &
\frac{\theta}{\sqrt{4 \pi t}} \hspace{0.1 cm} \int_{0}^{\infty} \left( e^{- \frac{x^{2}}{t}}  \hspace{0.1 cm}
\Tr \left( [ \hspace{0.1 cm} {\widetilde {\mathcal \gamma}} (\theta) \hspace{0.1 cm} {\widetilde {\mathcal A}}\hspace{0.1 cm}, \hspace{0.1 cm}
{\widetilde {\mathcal Q}}(x) \hspace{0.1 cm} ] \hspace{0.1 cm}
{\widetilde P}(\theta) \hspace{0.1 cm} e^{- t {\widetilde {\mathcal A}}^{2}} \right) \right) dx \nonumber \\
& &  + \hspace{0.1 cm}
\frac{2}{\sqrt{4 \pi t}} \hspace{0.1 cm} \int_{0}^{\infty} \left( \phi(x) \hspace{0.1 cm} e^{- \frac{x^{2}}{t}} \hspace{0.1 cm}
\Tr \left( [ \hspace{0.1 cm} {\widetilde {\mathcal \gamma}} (\theta) \hspace{0.1 cm} {\widetilde {\mathcal A}}\hspace{0.1 cm}, \hspace{0.1 cm}
\left[ \begin{array}{clcr}
{\frak B}^{\ast} \hspace{0.1 cm} {\frak A}(\theta) & 0 \\ 0 & 0 \end{array} \right] \hspace{0.1 cm} ] \hspace{0.1 cm}
{\widetilde P}(\theta) \hspace{0.1 cm} e^{- t {\widetilde {\mathcal A}}^{2}} \right) \right) dx ,
\end{eqnarray}

\vspace{0.2 cm}

\noindent
where in the last expression we used the decomposition (\ref{E:3.009}).
Using Lemma \ref{Lemma:5.2}, we have

\begin{eqnarray}  \label{E:5.38}
& & \Tr \left\{ [ \hspace{0.1 cm} {\widetilde {\mathcal \gamma}} (\theta) \hspace{0.1 cm} {\widetilde {\mathcal A}}\hspace{0.1 cm}, \hspace{0.1 cm}
{\widetilde {\mathcal Q}}(x) \hspace{0.1 cm} ] \hspace{0.1 cm}
{\widetilde P}(\theta) \hspace{0.1 cm} e^{- t {\widetilde {\mathcal A}}^{2}} \right\}
\hspace{0.1 cm} = \hspace{0.1 cm}
\Tr \left\{ \left( \hspace{0.1 cm} {\widetilde {\mathcal \gamma}} (\theta) \hspace{0.1 cm} {\widetilde {\mathcal A}} \hspace{0.1 cm}
{\widetilde {\mathcal Q}}(x)
\hspace{0.1 cm} - \hspace{0.1 cm} {\widetilde {\mathcal Q}}(x) \hspace{0.1 cm} {\widetilde {\mathcal \gamma}} (\theta) \hspace{0.1 cm}
{\widetilde {\mathcal A}} \hspace{0.1 cm} \right) \hspace{0.1 cm}
{\widetilde P}(\theta) \hspace{0.1 cm} e^{- t {\widetilde {\mathcal A}}^{2}} \right\}   \nonumber  \\
& = & \Tr \left\{ \hspace{0.1 cm} {\widetilde {\mathcal \gamma}} (\theta) \left( \hspace{0.1 cm} {\widetilde {\mathcal A}} \hspace{0.1 cm}
{\widetilde {\mathcal Q}}(x)
\hspace{0.1 cm} + \hspace{0.1 cm} {\widetilde {\mathcal Q}}(x) \hspace{0.1 cm}
{\widetilde {\mathcal A}} \hspace{0.1 cm} \right) \hspace{0.1 cm}
{\widetilde P}(\theta) \hspace{0.1 cm} e^{- t {\widetilde {\mathcal A}}^{2}} \right\}  \nonumber  \\
& = & \Tr \left\{ \hspace{0.1 cm} {\widetilde {\mathcal \gamma}} (\theta) \left( \hspace{0.1 cm} {\widetilde {\mathcal A}} \hspace{0.1 cm}
{\widetilde {\mathcal Q}}(x)
\hspace{0.1 cm} + \hspace{0.1 cm} {\widetilde {\mathcal Q}}(x) \hspace{0.1 cm}
{\widetilde {\mathcal A}} \hspace{0.1 cm} \right) \hspace{0.1 cm}
\left( I - {\widetilde P}(\theta) \right) \hspace{0.1 cm} e^{- t {\widetilde {\mathcal A}}^{2}} \right\} \nonumber  \\
& = & \frac{1}{2} \hspace{0.1 cm}
\Tr \left\{ \hspace{0.1 cm} {\widetilde {\mathcal \gamma}} (\theta) \left( \hspace{0.1 cm} {\widetilde {\mathcal A}} \hspace{0.1 cm}
{\widetilde {\mathcal Q}}(x)
\hspace{0.1 cm} + \hspace{0.1 cm} {\widetilde {\mathcal Q}}(x) \hspace{0.1 cm}
{\widetilde {\mathcal A}} \hspace{0.1 cm} \right) \hspace{0.1 cm} e^{- t {\widetilde {\mathcal A}}^{2}} \right\}
\hspace{0.1 cm} = \hspace{0.1 cm} 0,
\end{eqnarray}

\vspace{0.2 cm}

\noindent
since $ \hspace{0.1 cm}{\widetilde {\mathcal A}} \hspace{0.1 cm} e^{- t {\widetilde {\mathcal A}}^{2}} \hspace{0.1 cm}  =
\hspace{0.1 cm} e^{- t {\widetilde {\mathcal A}}^{2}} \hspace{0.1 cm} {\widetilde {\mathcal A}} \hspace{0.1 cm}$ and
$ \hspace{0.1 cm}{\widetilde {\mathcal A}} \hspace{0.1 cm} {\widetilde {\mathcal \gamma}} (\theta) \hspace{0.1 cm}  = - \hspace{0.1 cm}
\hspace{0.1 cm} {\widetilde {\mathcal \gamma}} (\theta) \hspace{0.1 cm} {\widetilde {\mathcal A}} \hspace{0.1 cm}$.
Using Lemma \ref{Lemma:3.2} and (\ref{E:4.3232}),
by simple computation we have

\begin{eqnarray}  \label{E:5.40}
& & \Tr \left( [ \hspace{0.1 cm} {\widetilde {\mathcal \gamma}} (\theta) \hspace{0.1 cm} {\widetilde {\mathcal A}}\hspace{0.1 cm}, \hspace{0.1 cm}
\left[ \begin{array}{clcr}
{\frak B}^{\ast} \hspace{0.1 cm} {\frak A}(\theta) & 0 \\ 0 & 0 \end{array} \right] \hspace{0.1 cm} ] \hspace{0.1 cm}
{\widetilde P}(\theta) \hspace{0.1 cm} e^{- t {\widetilde {\mathcal A}}^{2}} \right) \nonumber  \\
& = & - \frac{i}{2} \Tr \left\{ \left[ \begin{array}{clcr}
{\frak B}^{\ast} \hspace{0.1 cm} {\frak A}(\theta) \hspace{0.1 cm} {\widetilde {\mathcal A}} \hspace{0.1 cm} \left(\hspace{0.1 cm}
{\frak A}(\theta) \hspace{0.1 cm} \sin \theta + {\frak B} \hspace{0.1 cm} \cos \theta \hspace{0.1 cm}\right) &
{\frak B}^{\ast} \hspace{0.1 cm} {\frak A}(\theta) \hspace{0.1 cm} {\widetilde {\mathcal A}} \\
{\widetilde {\mathcal A}}\hspace{0.1 cm} {\frak B}^{\ast} \hspace{0.1 cm} {\frak A}(\theta)  &
{\widetilde {\mathcal A}}\hspace{0.1 cm}  {\frak B}^{\ast} \hspace{0.1 cm} \sin \theta -
{\widetilde {\mathcal A}}\hspace{0.1 cm} {\frak A} (\theta) \hspace{0.1 cm} \cos \theta \end{array} \right]
e^{- t {\widetilde {\mathcal A}}^{2}} \right\}  \nonumber  \\
& = & - \frac{i}{2} \hspace{0.1 cm} \Tr \left\{
\left(  {\frak B}^{\ast} \hspace{0.1 cm} {\frak A}(\theta) \hspace{0.1 cm} {\widetilde {\mathcal A}} \hspace{0.1 cm} \left( \hspace{0.1 cm}
{\frak A}(\theta) \hspace{0.1 cm} \sin \theta + {\frak B} \hspace{0.1 cm} \cos \theta \hspace{0.1 cm}\right)   \right) \hspace{0.1 cm}
\frac{I - i {\widetilde {\mathcal \gamma}}(\theta)}{2} \hspace{0.1 cm} e^{- t {\widetilde {\mathcal A}}^{2}} \right\}  \nonumber \\
& & \hspace{0.1 cm}  - \hspace{0.1 cm} \frac{i}{2} \hspace{0.1 cm}   \Tr \left\{
\left( {\widetilde {\mathcal A}}\hspace{0.1 cm}  {\frak B}^{\ast} \hspace{0.1 cm} \sin \theta -
{\widetilde {\mathcal A}}\hspace{0.1 cm} {\frak A} (\theta) \hspace{0.1 cm} \cos \theta \right) \hspace{0.1 cm}
\frac{I + i {\widetilde {\mathcal \gamma}}(\theta)}{2} \hspace{0.1 cm} e^{- t {\widetilde {\mathcal A}}^{2}} \right\}  \nonumber \\
& = & - \frac{i}{2} \hspace{0.1 cm} \sin \theta \Tr \left( {\widetilde {\mathcal A}}\hspace{0.1 cm}  {\frak B}^{\ast} \hspace{0.1 cm} e^{- t {\widetilde {\mathcal A}}^{2}} \right)
\hspace{0.1 cm} - \hspace{0.1 cm} \frac{1}{2} \hspace{0.1 cm} \cos \theta \hspace{0.1 cm}
\Tr \left(  {\widetilde {\mathcal A}}\hspace{0.1 cm}  {\frak A} (\theta) \hspace{0.1 cm} {\widetilde {\mathcal \gamma}} (\theta)
 \hspace{0.1 cm}  e^{- t {\widetilde {\mathcal A}}^{2}} \right) .
\end{eqnarray}

\vspace{0.2 cm}

\noindent
Setting $\hspace{0.1 cm} \K = \left( \begin{array}{clcr} -1 & 0 \\ 0 & 1 \end{array} \right)$,
$\hspace{0.1 cm} \LL = \left( \begin{array}{clcr} 0 & - 1 \\ -1 & 0 \end{array} \right)$ and
 $\hspace{0.1 cm} \J = \left( \begin{array}{clcr} 0 & 1 \\ -1 & 0 \end{array} \right)$, we have

\begin{eqnarray}  \label{E:5.41}
{\widetilde {\mathcal A}} \hspace{0.1 cm} {\frak B}^{\ast} & = &
\beta \left( \Gamma^{Y} \nabla^{Y} + \nabla^{Y} \Gamma^{Y} \right) \left( \begin{array}{clcr} 0 & - \J \\ \J & 0 \end{array} \right)
\nonumber  \\
{\widetilde {\mathcal A}}\hspace{0.1 cm}  {\frak A} (\theta) \hspace{0.1 cm} {\widetilde {\mathcal \gamma}} (\theta)  & = &
- i \beta  \left( \Gamma^{Y} \nabla^{Y} + \nabla^{Y} \Gamma^{Y} \right) \left( \left( \B_{Y}^{2} \right)^{-1} \left(
\left( \B_{Y}^{2} \right)^{-} - \left( \B_{Y}^{2} \right)^{+} \right) \right) \left( \begin{array}{clcr} \LL & 0 \\ 0 & \LL \end{array} \right),
\end{eqnarray}

\noindent
which shows that

\begin{eqnarray}  \label{E:5.42}
\Tr \left( {\widetilde {\mathcal A}}\hspace{0.1 cm}  {\frak B}^{\ast} \hspace{0.1 cm} e^{- t {\widetilde {\mathcal A}}^{2}} \right)
\hspace{0.1 cm}  = \hspace{0.1 cm}
\Tr \left(  {\widetilde {\mathcal A}}\hspace{0.1 cm}  {\frak A} (\theta) \hspace{0.1 cm} {\widetilde {\mathcal \gamma}} (\theta)
 \hspace{0.1 cm}  e^{- t {\widetilde {\mathcal A}}^{2}} \right)
\hspace{0.1 cm}  = \hspace{0.1 cm} 0.
\end{eqnarray}

\vspace{0.2 cm}

\noindent
By (\ref{E:5.38}) and (\ref{E:5.42}), we have

\begin{eqnarray}   \label{E:5.46}
\Tr \left( [ \hspace{0.1 cm} {\widetilde {\mathcal \gamma}} (\theta) \hspace{0.1 cm} {\widetilde {\mathcal A}}\hspace{0.1 cm}, \hspace{0.1 cm}
\left[ \begin{array}{clcr}
{\frak B}^{\ast} \hspace{0.1 cm} {\frak A}(\theta) & 0 \\ 0 & 0 \end{array} \right] \hspace{0.1 cm} ] \hspace{0.1 cm}
{\widetilde P}(\theta) \hspace{0.1 cm} e^{- t {\widetilde {\mathcal A}}^{2}} \right) \hspace{0.1 cm} = \hspace{0.1 cm} 0 \quad \text{and hence} \quad
(\III) = 0.
\end{eqnarray}

\noindent
Adding up the above arguments, we have

\begin{eqnarray}   \label{E:5.47}
\Tr \left\{ \left( \widetilde{\B}(\theta) {\mathcal Q}(x) -
{\mathcal Q}(x) \widetilde{\B}(\theta) \right) e^{- t \widetilde{\B}(\theta)_{{\widetilde P}(\theta)}^{2}} \right\}
\hspace{0.1 cm} \sim \hspace{0.1 cm} O(e^{- \frac{c}{t}}) \qquad \text{for} \quad t \rightarrow 0^{+},
\end{eqnarray}

\noindent
and this completes the proof of Theorem \ref{Theorem:4.2}.

%%%%%%%%%%%%%%%%%%%%%%%%%%%%%%%%%%%%%%%%%%%%%%%%%%%%%%%%%%%%%%%%%%%%%%%%%%%%%

\end{document}